\documentclass[10pt,openany,leqno,usenames,dvipsnames]{article}
\usepackage[colorlinks=true]{hyperref}
\usepackage{amsmath,amsthm,amsfonts,amssymb,amscd,hyperref}
\usepackage[capitalize]{cleveref}
\usepackage{graphicx}
\usepackage{soul}
\usepackage[percent]{overpic}
\usepackage{xcolor}
\usepackage{graphics,xcolor,url}
\numberwithin{equation}{section}
\input{xy} \xyoption{all}
\usepackage[latin1]{inputenc}
\usepackage{enumitem}
\usepackage{soul}
\usepackage{cases}
\usepackage{hyperref}
\usepackage{epsfig}
\input{xy} \xyoption{all}
\usepackage{mathrsfs}

\def\A{\mathbb A}

\def\R{\mathbb R}
\def\N{\mathbb N}
\def\T{\mathbb T}

\def\Z{\mathbb Z}

\def\det{\mathrm{det}}

\def\diam{\mathrm{diam}\:}

\newtheorem{proposition}{Proposition}[section]
\newtheorem{theorem}[proposition]{Theorem}

\newtheorem*{theorem*}{Theorem}
\newtheorem{coro}[proposition]{Corollary}

\newtheorem*{problem*}{Problem}
\newtheorem{definition}[proposition]{Definition}
\newtheorem{lemma}[proposition] {Lemma}
\newtheorem{corollary}[proposition]{Corollary}
\newtheorem{claim}[proposition]{Claim}

\newtheorem{theoalph}{Theorem}

\newtheorem{coralph}[theoalph]{Corollary}

\newtheorem{defalph}[theoalph]{Definition}

\theoremstyle{remark}

\newtheorem{remark}[proposition]{Remark}
\newtheorem{notation}[proposition]{Notation}

\makeatletter 

\DeclareTextFontCommand{\emph}{\em\bf}

\begin{document}
	\title{Birkhoff attractors of dissipative billiards}
	\author{Olga Bernardi, Anna Florio\thanks{Partially supported by the ANR project CoSyDy (ANR-CE40-0014) and the ANR project GALS (ANR-23-CE40-0001).}, Martin Leguil\thanks{Partially supported by the ANR project CoSyDy (ANR-CE40-0014) and the ANR project Padawan (ANR-21-CE40-0012-01).}}
	\date{\today}
	\maketitle
	
	\begin{abstract}
	\noindent We study the dynamics of dissipative billiard maps within planar convex domains. Such maps have a global attractor. We are interested in the topological and dynamical complexity of the attractor, in terms both of the geometry of the billiard table and of the strength of the dissipation. We focus on the study of an invariant subset of the attractor, the so-called Birkhoff attractor. On the one hand, we show that for a generic convex table with ``pinched'' curvature, the Birkhoff attractor is a normally contracted manifold when the dissipation is strong. On the other hand, for a mild dissipation, we prove that generically the Birkhoff attractor is complicated, both from the topological and the dynamical point of view.
	\end{abstract}
\section{Introduction} \label{intro}

\indent In the present paper, given a convex planar domain, we consider a variant of the usual billiard map in order to model some dissipative phenomena, which result in the existence of a global attractor. For such dissipative maps, Birkhoff \cite{Birkhoff1} introduced an invariant subset of the attractor, the so-called Birkhoff attractor; as we shall see, it is minimal in some sense among all invariant sets which separate phase-space, and it is essentially the place where interesting dynamics occurs. We investigate the properties of  the Birkhoff attractor, in particular, how they change as the dissipation  parameter is varied.

Loosely speaking, like for conservative billiards, we consider a massless particle moving with unit velocity inside the billiard table $\Omega \subset \R^2$ according to the usual law except at collisions with the boundary $\partial\Omega$, which we now assume to be {\it{inelastic}}. More precisely (see Fig. \ref{billliard}), 
\begin{itemize}
	\item[--] the motion happens along straight lines between two collisions;
	\item[--] at each orthogonal collision, the velocity vector is replaced with its opposite, while at a non-orthogonal collision, it is changed in such a way that the (unoriented) outgoing angle of reflection is strictly smaller than the incoming angle of incidence, both being measured with respect to the normal to $\partial \Omega$. 
\end{itemize}  
In other words, the reflected angle bends toward the inner normal at the incidence point. We refer to Definition \ref{definit diss bill} here below for more details, and to Section \ref{DBM} for further properties of these billiard maps. \\
~\newline
\indent Billiards exhibiting some form of dissipation have already been considered in previous works. To the best of the authors' knowledge, for outer billiards, dissipation was first introduced in \cite[Page 83]{Tabachnikov}. Subsequently, dynamical properties of dissipative polygonal outer billiards have been studied in \cite{DMGG15}. Regarding standard billiards, the paper \cite{MarPujSam} by Markarian-Pujals-Sambarino is dedicated to the study of limit sets of dissipative billiards (called here \textit{pinball billiards}) for various types of tables (close to a circle, with semi-dispersing walls, which possess some hyperbolicity...), through the existence of a dominated splitting. Motivated by these rigorous results, the paper \cite{ArrMarSan09} numerically investigates and characterizes the bifurcations of the resulting attractors as the contraction
parameter is varied. In \cite{MarKamPDC} the authors construct simple examples of non elastic convex billiards with dominated splitting and attractors supporting a rational or irrational rotation. Let us conclude this brief overview by mentioning some works about dissipative billiards for tables with flat walls. The paper \cite{ArrMarSan12} concentrates on inelastic billiard dynamics in an equilateral triangular table and studies --mainly numerically-- the structure of fractal strange attractors and their evolution as the contraction parameter changes. Finally, in a series of works \cite{DelMDuarteII,DelMDuarteI,DelMDuarteIII,Duarte}, Del Magno-Lopes Dias-Duarte-Gaiv\~{a}o-Pinheiro and Soufi investigate dissipative billiards within various types of polygonal tables; in particular, they study the structure of the nonwandering sets of such billiards, the existence of hyperbolic attractors, and prove the existence of countably many SRB measures on these attractors under suitable conditions. \\
~\newline
\indent Let us now move on to the formal definition of dissipative billiard maps considered in the present work. Let $\Omega\subset\R^2$ be a strictly convex domain with $C^k$ boundary $\partial \Omega$, $k\geq 2$. We say that $\Omega\subset\R^2$ is \textit{strongly convex} if --additionally-- its curvature never vanishes. We assume that the perimeter of $\partial \Omega$ is normalized to one. We fix an orientation of $\partial \Omega$ and parametrize $\partial \Omega$ in arclength by some map $\Upsilon\colon \T \to \R^2$, where $\T:=\R/\Z$. The phase-space is the set of pairs $(x,v)$ consisting of a point  $x$ on $\partial \Omega$, and a unit vector $v\in T_x \Omega$ pointing inward, or tangent to $\partial \Omega$. It is naturally identified with the cylinder $\A:=\T \times [-1,1]$; indeed, any point $(x,v)$ in  phase-space corresponds to a pair $(s,r)\in \T \times [-1,1]$, where $x=\Upsilon(s)\in \partial \Omega$, and $r = \sin \varphi$ is the sine of the oriented angle $\varphi \in \left[-\frac{\pi}{2},\frac{\pi}{2}\right]$ from the vector $v$ to the inward normal to $\partial \Omega$ at $x$. The usual conservative billiard map is then defined as
\begin{equation}\label{stnd billi}
	f=f_1\colon\left\{ 
	\begin{array}{rcl}
	\mathbb{A}& \to& \mathbb{A}\, , \\
		(s,r) &\mapsto&  f(s,r)=(s',r_1')\, ,
		\end{array}
		\right.
\end{equation}
where $\Upsilon(s')$ represents the point where the trajectory starting at $\Upsilon(s)$ along the direction making an angle $\arcsin r$ with the normal at $\Upsilon(s)$, 
hits the boundary again, and $r_1'$ is the sine of the reflected angle at $\Upsilon(s')$, according to the standard reflection law (angle of incidence = angle of reflection). Let us now fix a dissipation parameter $\lambda \in (0,1)$. 

\begin{figure}[h] \label{billliard}
		\centering   
	\includegraphics[scale=0.6]
	{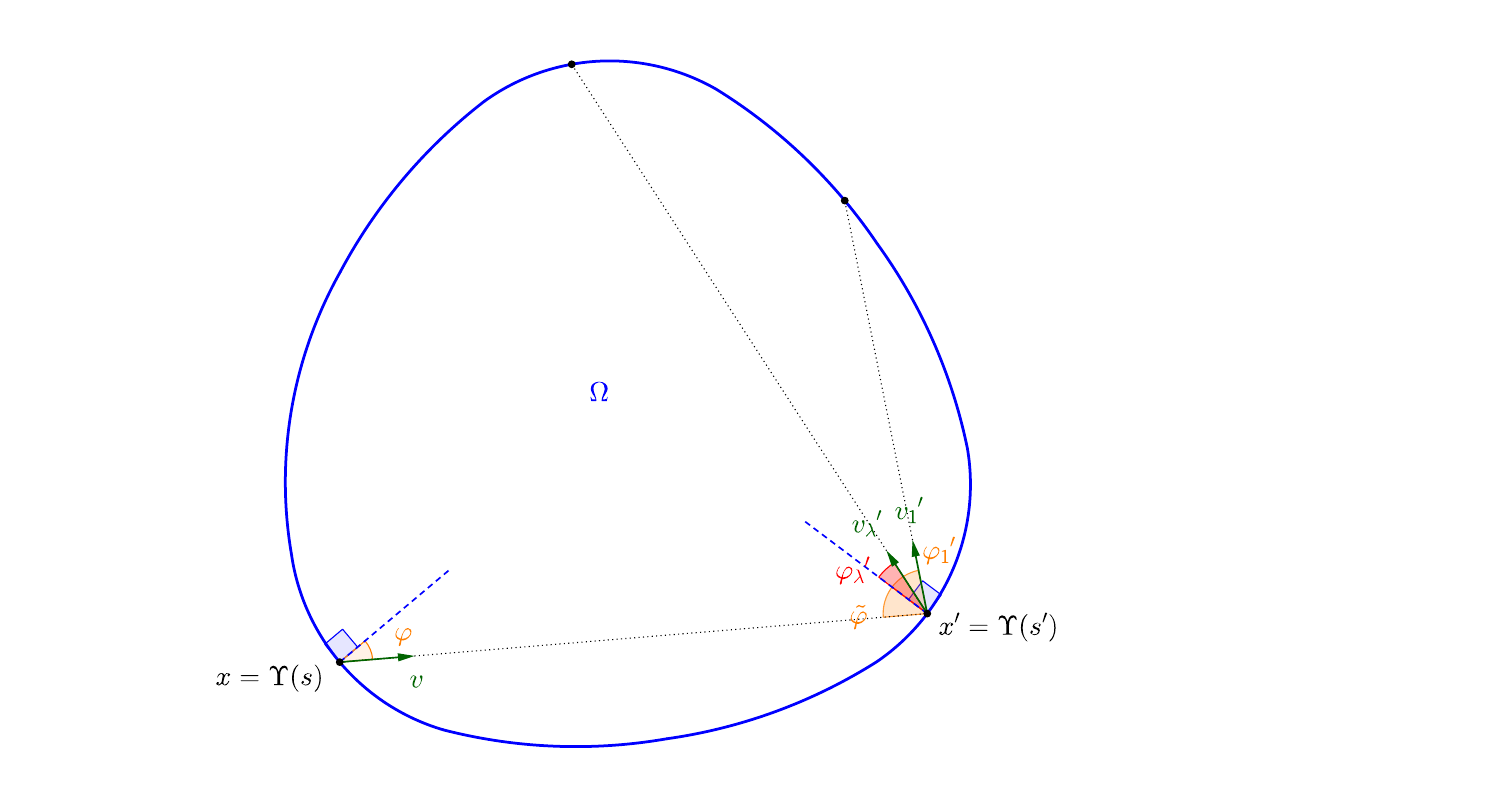}
	\caption{The standard billiard map and its dissipative counterpart.}
	\label{stdissbil}
\end{figure}
 
\begin{defalph}\label{definit diss bill}
Given a domain $\Omega$ as above, let us fix a $C^{k-1}$ function $\lambda\colon \A\to (0,1)$ such that  
\begin{equation}\label{cond lambda dis}
	0<\partial_r \lambda(s,r)r+\lambda(s,r)<1,\quad \forall\, (s,r)\in \mathrm{int}(A),
\end{equation}
	and let $\mathcal{H}_\lambda\colon (s,r)\mapsto(s,\lambda(s,r) r)$. The dissipative billiard map $f_\lambda$ associated to $\lambda$ is then defined as the map 
	\begin{equation*}
		f_\lambda:=\mathcal{H}_\lambda \circ f\colon\left\{ 
		\begin{array}{rcl}
			\mathbb{A}& \to& \mathbb{A}\, , \\
			(s,r) &\mapsto&  f_\lambda(s,r)=(s',r_\lambda')\, .
		\end{array}
		\right.
	\end{equation*}
	where
	$$
	r_\lambda'=r_\lambda'(s,r):=\lambda(s',r_1')r_1',
	$$ 
	for $r_1'=r_1'(s,r)$ as in \eqref{stnd billi}. 
	Note that for any $(s,r)\in \mathrm{int}(A)$, we have $\det D\mathcal{H}_\lambda(f_1(s,r))=\partial_r \lambda(s',r_1')r_1'+\lambda(s',r_1')$. In particular, since $\det Df_{1}(s,r)=1$, and by \eqref{cond lambda dis}, we obtain
	\begin{equation}\label{eq jac dfl}
		0 < \det Df_{\lambda}(s,r) =\det D\mathcal{H}_\lambda(f_1(s,r))< 1,\quad \forall\, (s,r) \in \mathrm{int}(\mathbb{A}).
	\end{equation}
\end{defalph}
	\noindent By \eqref{eq jac dfl}, the resulting billiard map $f_\lambda$ is no longer conservative; actually, it turns out to be a dissipative map in the sense of \cite{LeCalvez} (see Definition \ref{def diss map}). In particular, $f_{\lambda}$ contracts the standard area form $\omega = d r \wedge ds$. We refer to Section \ref{tre} for a few general facts about dissipative billiards. 

\begin{remark}
	For simplicity, in most of what follows, we will restrict ourselves to the case where $\lambda$ is actually a constant function, i.e., $\lambda \equiv \lambda_*\in (0,1)$. In that case, we will say that $f_\lambda$ has constant dissipation. Then, the dissipative billiard map associated to $\lambda$ simply becomes 
	\begin{equation*}
		f_\lambda\colon\left\{ 
		\begin{array}{rcl}
			\mathbb{A}& \to& \mathbb{A}\, , \\
			(s,r) &\mapsto&  f_\lambda(s,r)=(s',r_\lambda')\, ,
		\end{array}
		\right.
	\end{equation*}
	where
		$$
		r_\lambda'=r_\lambda'(s,r):= \lambda r_1'.
		$$ 
		In the following,  
		when it is clear from the context, we will 
		abbreviate $r_\lambda'=r'$. 
	
	For constant dissipation, there is a natural one-parameter family of dissipative billiard maps $\{f_\lambda\}_{\lambda\in (0,1)}$; in particular, we will study transitions in the behavior of the Birkhoff attractor as $\lambda$ changes. 
	However, the simplifying hypothesis that $\lambda$ is constant is not essential. 
	Indeed, as we will explain, most results shown in the present work hold under the more general assumption that $\lambda\colon \A \to (0,1)$ is a $C^1$ function as in Definition \ref{definit diss bill} that is close enough to being constant, namely $\|D\lambda\|\ll 1$. 
\end{remark} 
Due to the dissipative character of $f_{\lambda}$, there is a contraction of the phase-space which results in the existence of attractors.  Indeed, as $f_\lambda(\mathbb{A})\subset \mathrm{int}(\A)$,  there exists a {\it{global attractor}}
\begin{equation} \label{CON}
	\Lambda_\lambda^0:= \bigcap_{k \ge 0} f_\lambda^k(\A)\, .
\end{equation}
The attractor $\Lambda_\lambda^0$ is $f_\lambda$-invariant, non-empty, compact and connected. Moreover, $\Lambda_\lambda^0$ separates $\A$, i.e., $\A\setminus\Lambda_\lambda^0$ is the disjoint union of two connected open sets  $U_{\lambda},V_{\lambda}$. However, we can find a smaller invariant set --the so-called {\it{Birkhoff attractor}}-- by ``removing the hairs'' from $\Lambda_\lambda^0$ (see e.g. \cite[Page 91]{LeCalvez1990}). The Birkhoff attractor, here denoted $\Lambda_\lambda$, is then defined as
	\begin{equation}
		\Lambda_\lambda:= \overline{U}_{\lambda}\cap \overline{V}_{\lambda}\, .
	\end{equation}
We remark that, even if $\Lambda_\lambda$ is compact and $f_\lambda$-invariant, it may no longer be an attractor in the usual sense. Actually, $\Lambda_\lambda$ can also be characterized as the minimal element (with respect to inclusion) among all sets which are compact, connected, $f_\lambda$-invariant and separate $\A$. We refer to Section \ref{PRELIMINARI} for more details about the Birkhoff attractor and its properties. \\

 The notion of Birkhoff attractor was first introduced by Birkhoff in \cite{Birkhoff1}. In the framework of dissipative twist maps of the annulus, further properties of the Birkhoff attractor have been investigated by the works of Charpentier \cite{Charpentier} and of Le Calvez \cite{LeCalvez}. The Birkhoff attractor of the thickened Arnol'd family has been studied by Crovisier in \cite{CroPHD}. Different authors have derived criteria to guarantee the existence of chaotic behaviors for invariant annular continua, see \cite{BargeGillette,BargeGillette2,HockettHolmes,Koropecki,Casdagli,PasseggiPotrieSambarino,PasTal}. Recently, the notion of Birkhoff attractor has been generalised to higher dimensions for conformally symplectic maps of some symplectic manifolds by Arnaud, Humili\`ere and Viterbo, see \cite{ArnHumVit,Viterbo}. 
 \begin{notation}
 Fix some dissipative map $f\colon \A\to \A$, with a hyperbolic periodic point $p\in \A$, of period $q \geq 1$. If $p$ is of saddle type, we will denote its $1$-dimensional stable, resp. unstable manifold as
 \begin{align*}
 \mathcal{W}^{s}(p;f^q) &:= \{x\in\A :\ \lim_{n \to +\infty} d(f^{qn}(x),p)=0\}\,,\\
  \mathcal{W}^{u}(p;f^q) &:= \{x\in\A :\ \lim_{n \to +\infty} d(f^{-qn}(x),p)=0\}\,. 
 \end{align*}
 If $p$ is a sink, we will denote its $2$-dimensional stable manifold as
 \begin{equation*}
 	\mathcal{W}^{s}(p;f^q) := \{x\in\A :\ \lim_{n \to +\infty} d(f^{qn}(x),p)=0\}\,. 
 \end{equation*}
 In either case, for $*=s/u$, we will sometimes abbreviate $\mathcal{W}^{*}(\mathcal{O}_f(p)):=\cup_{i=0}^{q-1} \mathcal{W}^{*}(f^i(p);f^q)$, or simply $\mathcal{W}^{*}(\mathcal{O}_\lambda(p))$, when $f=f_\lambda$ is some dissipative billiard map. 
 \end{notation}  
\indent Considering the crucial role of elliptic tables in the conservative case, it is natural to start our study with dissipative billiard maps within ellipses. The detailed study of the corresponding dynamics is contained in Section \ref{quattro}, whose main result is the next theorem.

\begin{theoalph}\label{main theoreme ellipses}
Given an ellipse $\mathcal{E}$ of eccentricity $e\in (0,1)$, let $f_\lambda \colon \mathbb{A} \to f_\lambda(\mathbb{A})\subset \mathrm{int}(\A)$ be a dissipative billiard map within $\mathcal{E}$ in the sense of Definition  \ref{definit diss bill} (we allow non-constant dissipation). 
Then, the $2$-periodic orbits $\{H,f_\lambda(H)\}$ and  $\{E,f_\lambda(E)\}$ corresponding  to the trajectories along the major and minor axes are hyperbolic of saddle and sink type, respectively, and the Birkhoff attractor satisfies 
\begin{equation*} 
\Lambda_\lambda^0=\Lambda_\lambda=\mathcal{W}^u(\mathcal{O}_\lambda(H)) \cup \{E,f_\lambda(E)\} = \overline{\mathcal{W}^u(\mathcal{O}_\lambda(H))} \,.
\end{equation*}
Moreover, for $i=0,1$, $\mathcal{W}^u(f_\lambda^i(H);f_\lambda^2)\setminus \{f_\lambda^i(H)\}$ is the disjoint union of two branches $\mathscr{C}_i^1,\mathscr{C}_i^2$, with $\mathscr{C}_i^j\subset \mathcal{W}^s(f_\lambda^j(E);f_\lambda^2)$, $j=0,1$. 
\end{theoalph}

\begin{figure}[h] \label{ellisse figura}
	\centering
	\begin{overpic}[width=0.6\textwidth,tics=10]{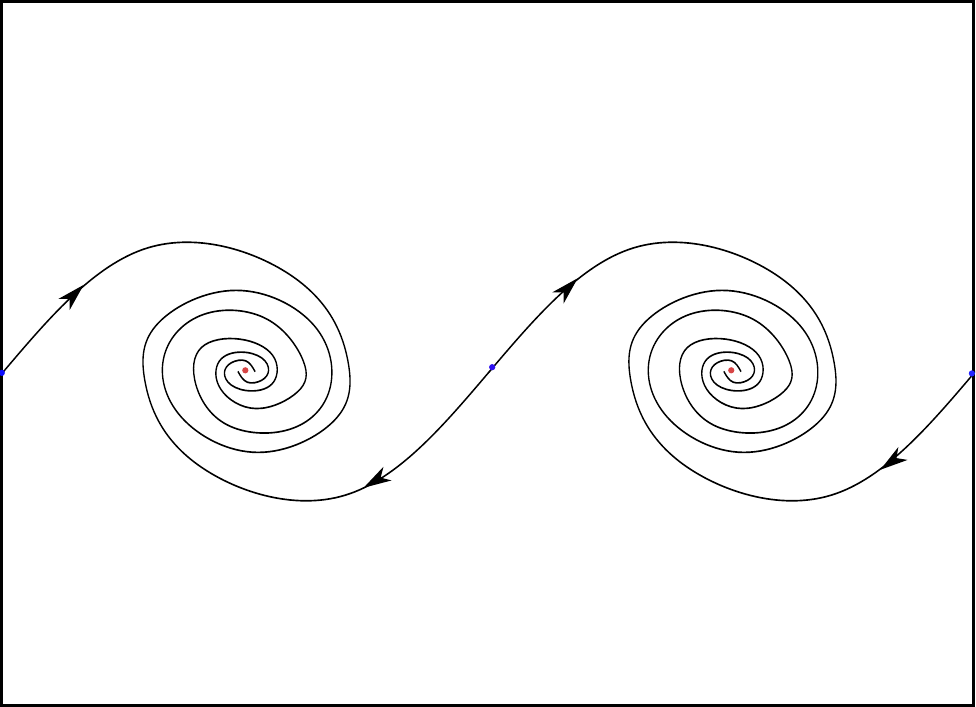}  
		\put(2,32){\color{blue}$H$}
		\put(49,30){\color{blue}$f_\lambda(H)$}
		\put(-6,35){\color{black}$\Lambda_\lambda$}
		\put(8,49){\color{black}$\mathcal{W}^u(H)$}
		\put(60,49){\color{black}$\mathcal{W}^u(f_\lambda(H))$}
		\put(25,30){\color{red}$E$}
		\put(75,30){\color{red}$f_\lambda(E)$}
	\end{overpic}
	\caption{Birkhoff attractor of a dissipative billiard map $f_\lambda$ within an ellipse of non-zero eccentricity when the dissipation is mild, i.e., $\lambda$ is close to $1$. }
	\label{figellipse}
\end{figure}
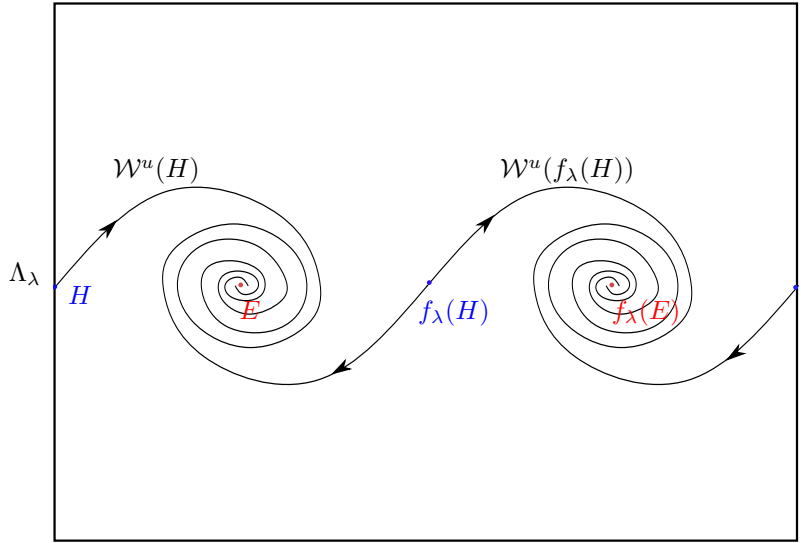

At the end of Section \ref{quattro}, we prove that the conclusion of Theorem \ref{main theoreme ellipses} remains true also for strictly convex domains whose boundary is sufficiently $C^2$-close to an ellipse, as stated in the next corollary. For simplicity, we state it in the case where the dissipation function $\lambda$ is a constant in $(0,1)$. 

\begin{coralph}\label{cor pert ellipse intro}
Let $\mathcal{E}$ be an ellipse of eccentricity $e\in (0,1)$. Let $\lambda \in (0,1)$. There exists $\epsilon=\epsilon(\mathcal{E},\lambda)>0$ such that for any $C^k$ ($k\geq 2$) domain $\Omega\subset \R^2$ 
 satisfying $d_{C^2}(\partial\Omega,\mathcal{E})<\epsilon$, 
	the following holds.
Denoting by $f_\lambda^\Omega$ the dissipative billiard map within $\Omega$, 
there exist $2$-periodic orbits $\{H,f_\lambda^\Omega(H)\}$ and $\{E,f_\lambda^\Omega(E)\}$  
		of saddle and sink type, respectively, and the Birkhoff attractor 
		is equal to 
		$$
		\Lambda_\lambda= 
		\mathcal{W}^u(\mathcal{O}_{f_\lambda^\Omega}(H)) \cup \{E,f_\lambda^\Omega(E)\}\, . 
		$$ 
		Moreover, the function $(\mathcal{E},\lambda)\mapsto \epsilon(\mathcal{E},\lambda)$ can be chosen to be continuous. 
\end{coralph}
\indent The first examples of Birkhoff attractors for a dissipative billiard map $f_\lambda$ within a circle or an ellipse (see Fig. \ref{figellipse} illustrating the Birkhoff attractor in the case of an ellipse when the dissipation is mild, i.e., $\lambda$ close to $1$) naturally lead us to consider topological properties of Birkhoff attractors, in particular investigate when $\Lambda_{\lambda}$ is topologically as simple as it can be, namely, a graph. The main results in this direction are contained in  Section \ref{cinque}. Through the following definition, we introduce the class of billiards for which we can guarantee such a simple behavior of the Birkhoff attractor. 

\begin{defalph}\label{defi set d k}
	For any $k \geq 2$, let $\mathcal{D}^k$ be the set of strongly convex domains $\Omega$ with $C^k$ boundary $\partial \Omega$ such that, given a parametrization $\Upsilon \colon \mathbb{T} \to \R^2$ of $\partial \Omega$, the following geometric condition holds (see Fig. \ref{ConditionTauK}):
	\begin{equation}\label{condition geom norm hyp}
	\max_{s \in \T} \tau(s)\mathcal{K}(s)<-1\,,
	\end{equation}
	where $\mathcal{K}(s)<0$ denotes the curvature of $\partial \Omega$ at the point $\Upsilon(s)$,  and $\tau(s)>0$ is the length of the first segment of the $f_1$-orbit starting at $\Upsilon(s)$ perpendicularly to $\partial\Omega$. Alternatively, condition \eqref{condition geom norm hyp} amounts to asking that the centers of the osculating circles at the points of $\partial \Omega$ remain in $\Omega$. 
\end{defalph}

\begin{figure}[h]
	\centering
	\includegraphics[scale=0.35]
	{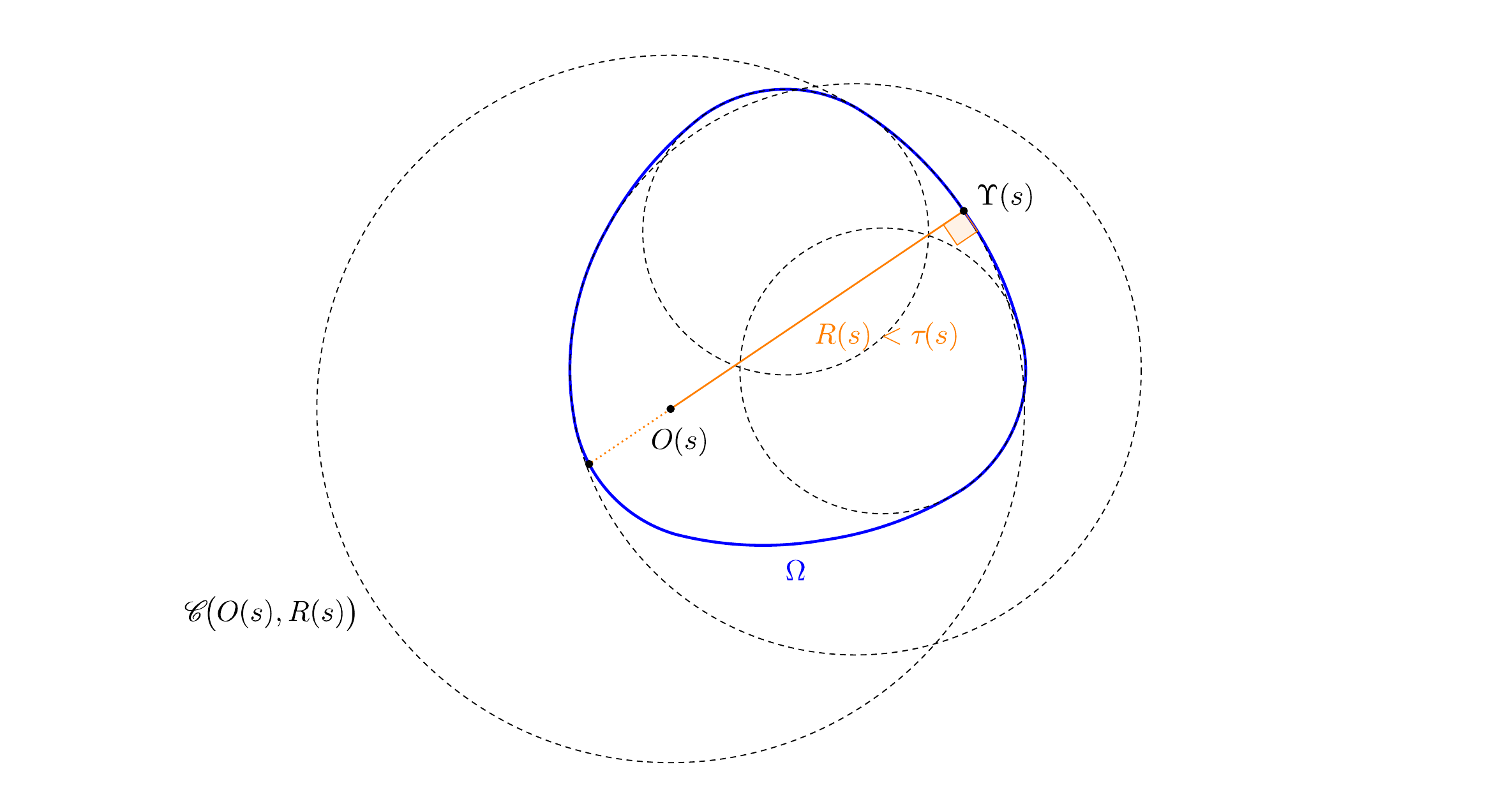}
	\caption{The geometric condition $\max_{s \in \T} \tau(s)\mathcal{K}(s)<-1$ in Definition \ref{defi set d k}. Here, $R(s):=-\frac{1}{\mathcal{K}(s)}$ is the radius of curvature, and  $\mathscr{C}\big(O(s),R(s)\big)$ is the osculating circle at $\Upsilon(s)$.}
	\label{ConditionTauK}
\end{figure} 

\noindent 
Clearly, the set $\mathcal{D}^k$ is $C^k$-open. More precisely, for any $\Omega \in \mathcal{D}^k$, there exists a $C^2$-open neighborhood $\mathcal{U}$ of $\Omega$ such that for any $C^k$ domain $\Omega'\in \mathcal{U}$, we have $\Omega'\in \mathcal{D}^k$. \\
The main result of Section \ref{cinque} is proving that the geometric condition contained in Definition \ref{defi set d k} together with strong dissipation ($\lambda$ close to $0$) suffice to guarantee that the corresponding Birkhoff attractor is a graph. Our result is also about the dynamics on the attractor and the graph's regularity. We use the notions of dominated splitting for an invariant set and of normally contracted manifold. We refer to Section \ref{cinque} for more details about such definitions.

\begin{theoalph}\label{main theoreme nh}
	Let $f_{\lambda} \colon \mathbb{A} \to f_\lambda(\mathbb{A})\subset \mathrm{int}(\A)$ be a dissipative billiard map with constant dissipation $\lambda \in (0,1)$ within some domain $\Omega \in \mathcal{D}^k$, $k\geq 2$. Then the following assertions hold.
	\begin{enumerate}
		\item There exists $\lambda(\Omega)\in (0,1)$ such that, for any $\lambda\in (0,\lambda(\Omega))$, the Birkhoff attractor $\Lambda_\lambda$ coincides with $\Lambda_\lambda^0$ and has a dominated splitting $E^s \oplus E^c$, where $E^s$ is uniformly contracted by $Df_\lambda$. Moreover, $\Lambda_\lambda$ is a normally contracted $C^{1}$ graph  over $\mathbb{T}\times\{0\}$ which is tangent to $E^c$.
		\item There exists $\lambda'(\Omega)\in(0,\lambda(\Omega))$ such that for any $\lambda \in (0,\lambda'(\Omega))$, $\Lambda_\lambda$ is actually a $C^{k-1}$ graph.
		\item There exists an open and dense set $\mathscr{U}$ of $C^k$ domains\footnote{We refer the reader e.g. to \cite[Section 2]{PintoC} for more details on the topology on the space of $C^k$ convex billiards.} such that if, moreover, $\Omega \in \mathcal{D}^k\cap\mathscr{U}$, then there exists $\lambda''(\Omega)\in (0,\lambda'(\Omega))$ such that, for any $\lambda \in (0,\lambda''(\Omega))$,   $\Lambda_\lambda$ is a $C^{k-1}$ normally contracted graph of rotation number $\frac 12$. Moreover, 
		\begin{equation*}
		\Lambda_\lambda=\bigcup_{i=1}^\ell \overline{\mathcal{W}^u(\mathcal{O}_\lambda(H_i))}\, ,
		\end{equation*}
for some finite collection $\{\mathcal{O}_\lambda(H_i)\}_{i=1,\cdots,\ell}=\{H_i,f_\lambda(H_i)\}_{i=1,\cdots,\ell}$ of $2$-periodic orbits of saddle type.  
	\end{enumerate}
\end{theoalph}
\begin{remark} 
	Given $k \geq 2$ and a domain $\Omega \in \mathcal{D}^k\cap\mathscr{U}$ as in the above statement, the conclusion of Theorem \ref{main theoreme nh} holds for general dissipative billiard maps $f_\lambda$ in the sense of Definition  \ref{definit diss bill}, provided that the dissipation function $\lambda \colon \A\to (0,1)$ satisfies $\|\lambda\|_{C^1}\ll 1$. See e.g. Remark \ref{argue variable} for more details. 
\end{remark}
\begin{remark}
	A consequence of Theorem \ref{main theoreme nh} 
	is that if $\partial \Omega$ is an ellipse $\mathcal{E}$ of eccentricity $e\in (0,\frac{\sqrt 2}{2})$, then $\mathcal{E} \in \mathcal{D}^\infty$ and, for any $\lambda \in (0,\lambda(\mathcal{E}))$, the corresponding Birkhoff attractor $\Lambda_\lambda$ is a normally contracted $C^1$ graph, which is actually $C^\infty$ except possibly at the $2$-periodic sink $\{E,f_\lambda(E)\}$, where $\Lambda_\lambda$ is tangent to the weak stable space of the sink, see Corollary \ref{coro dom splitting ellipse}. We will also see that, when the eccentricity $e$ is larger than $\frac{\sqrt 2}{2}$, then for $\lambda\in (0,1)$ small, the Birkhoff attractor $\Lambda_\lambda$ is no longer a graph (see Proposition \ref{propo not graph}). 
\end{remark}

We may wonder if 
Birkhoff attractors of dissipative billiards may exhibit more complex topological properties than the examples described in Sections \ref{quattro} and \ref{cinque}. In fact, following a result by M. Charpentier \cite{Charpentier}, a Birkhoff attractor for a dissipative diffeomorphism can be an ``indecomposable continuum'', and a sufficient condition for this to occur is that 
the Birkhoff attractor contains points with different rotation numbers. The aim of Section \ref{section different rho} is exploring this direction and discussing some topological and dynamical implications of such a phenomenon. \\
\noindent In order to state the main results, we need to premise the notion of upper and lower rotation number for $\Lambda_{\lambda}$. Denote by $V_\Lambda$ (resp. $U_\Lambda$) the connected component of $\A\setminus \Lambda_\lambda$ containing $\{(s,1)\in\A :\ s\in\T\}$ (resp. $\{(s,-1))\in\A :\ s\in\T\}$). For any $(s,r)\in\A$ the upper and lower vertical lines are, respectively
\[
V^+(s,r):=\{(s,y)\in\A :\ y\geq r\}\qquad \text{and}\qquad V^-(s,r):=\{(s,y)\in\A :\ y\leq r\}\, .
\]
Let us now define  
\[
\Lambda_\lambda^+:=\{ x\in\Lambda_\lambda :\  V^+(x)\setminus\{x\}\subset  V_\Lambda\}\qquad\text{and}\qquad\Lambda_\lambda^-:=\{ x\in\Lambda_\lambda :\  V^-(x)\setminus\{x\}\subset U_\Lambda\}\,.
\]
Given a covering $\pi\colon\R\times[-1,1]\to\T\times[-1,1]$ of $\A$, let $\tilde\Lambda_\lambda^\pm:=\pi^{-1}(\Lambda_\lambda^\pm)$. Moreover, let $\tilde \pi_1\colon\R\times[-1,1]\to\R$ be the first coordinate projection, and $F_\lambda\colon\R\times[-1,1]\to \R\times[-1,1]$ a continuous lift of $f_\lambda$. Then, by a result due to G.D. Birkhoff \cite{Birkhoff1} and rephrased in all details by P. Le Calvez  \cite{LeCalvez}, the sequence 
$$\left( \frac{\tilde \pi_1\circ F_\lambda^{-n}-\tilde \pi_1}{n} \right)_{n\in\N}$$
converges uniformly on $\tilde\Lambda_\lambda^+$ $(\text{resp.} \ \tilde\Lambda_\lambda^-)$ to a constant $\rho_\lambda^+$ $(\text{resp.} \ \rho_\lambda^-)$. The constants $\rho_\lambda^+$ and $\rho_\lambda^-$ -- called upper and lower rotation numbers -- do depend on the chosen lift, but not their difference.
We refer the reader to Subsection \ref{6 punto 1} for more details. \\
For the conservative billiard map $f=f_1$, let denote by $\mathscr{V}(f)$ the union of all $f$-invariant essential curves in $\A$. We recall that an instability region 
for $f$ is an open bounded connected component of $\A \setminus\mathscr{V}(f)$ that contains in its interior an essential curve. 
The main result of Section \ref{section different rho} is the next Theorem, whose proof is mainly based on an adaptation of some arguments of the work \cite{LeCalvez}. Let us recall that a continuum is a compact connected topological space. 

\begin{theoalph}\label{main theorem different nb rotation}
	Let $\Omega\subset\R^2$ be a strongly convex domain with $C^k$ boundary, $k\geq 2$. Let $f=f_1$ be the associated  conservative billiard map. If $f$ admits an instability region that contains the zero section $\mathbb{T}\times\{0\}$, then there exists $\lambda_0(\Omega)\in (0,1)$ such that, for any $\lambda\in[\lambda_0(\Omega),1)$, the Birkhoff attractor $\Lambda_\lambda$ of the dissipative billiard map $f_\lambda$ with constant dissipation $\lambda$ has $\rho_\lambda^+-\rho^-_\lambda>0$, with $\frac 12 \in (\rho_\lambda^-,\rho^+_\lambda)\mod \Z$.
\end{theoalph}

\begin{remark} 
	Theorem \ref{main theorem different nb rotation} was stated for a dissipative map $f_\lambda$ with constant dissipation. However, the result holds for general dissipative billiard maps in the sense of Definition \ref{definit diss bill}, as long as the dissipation function $\lambda \colon \A \to (0,1)$ is sufficiently close to the constant function $1$ in the $C^1$-topology, i.e., $\|1-\lambda\|_{C^1}\ll 1$. See e.g. Proposition \ref{prop different rho} for more details in this direction. 
\end{remark}

\noindent The above theorem has several interesting consequences for $\Lambda_\lambda$. In fact, in the case where $\rho_\lambda^+-\rho^-_\lambda>0$, the corresponding Birkhoff attractor turns out to be complicated both topologically and dynamically. In particular, 
\begin{itemize}
		\item[--] $\Lambda_\lambda$ is an indecomposable continuum, i.e., it cannot be written as the union of two proper continua  (directly from the work \cite{Charpentier} of M. Charpentier; see also \cite{BargeGillette2}).
		\item[--] Each rational $\frac{p}{q}\in(\rho_\lambda^-,\rho^+_\lambda)$ is the rotation number of a periodic orbit in $\Lambda_\lambda$ (as a straightforward application of \cite{BargeGillette}).
		\item[--] If $x$ is a saddle periodic point of type $(p,q)$, with $\frac{p}{q}\in (\rho_\lambda^-,\rho^+_\lambda)$, then its unstable manifold $\mathcal{W}^u(x;f^q)$ satisfies $\overline{\mathcal{W}^u(x;f^q)}\subset \Lambda_\lambda$ (by \cite[Proposition 14.3]{LeCalvez}). 
		\item[--] There exists $n_0\in \mathbb{N}$ so that $f_\lambda^{n_0}$ has a rotational horseshoe (by \cite[Theorem A]{PasseggiPotrieSambarino}). In particular, $f_\lambda|_{\Lambda_\lambda}$  has positive topological entropy. 
	\end{itemize} 

Applying essentially \cite{PintoC}, we prove that the conclusions of Theorem \ref{main theorem different nb rotation} hold generically for $C^k$ strongly convex domains, $k \ge 3$, as explained in the next corollary.   
\begin{coralph}\label{coro main theorem different nb rotation}
For $k \geq 3$, there exists an open and dense subset $\mathscr{U}$ of the set of $C^k$ strongly convex domains such that for every $\Omega\in\mathscr{U}$, the following assertions hold. 
\begin{enumerate}
		\item\label{pppppppun} There exists $\lambda_0(\Omega)\in (0,1)$ such that, for any $\lambda\in[\lambda_0(\Omega),1)$, the Birkhoff attractor $\Lambda_\lambda$ of the corresponding dissipative billiard map $f_\lambda$ has $\rho_\lambda^+-\rho^-_\lambda>0$, with $\frac 12 \in (\rho_\lambda^-,\rho^+_\lambda)\mod \Z$. 
		\item There exists $\lambda_1(\Omega)\in [\lambda_0(\Omega),1)$ such that, for any $\lambda\in[\lambda_1(\Omega),1)$ 
	and any $2$-periodic point $p$ of saddle type\footnote{e.g., when the $2$-periodic orbit $\{p,f_\lambda(p)\}$ corresponds to a diameter of the table}, there exists a horseshoe $K_\lambda(p)\subset \Lambda_\lambda$ in the homoclinic class of the $2$-periodic point $p$. 
	\end{enumerate}
\end{coralph}

\noindent Finally, as a consequence of \cite{Mather82}, we emphasize that the conclusion of point \eqref{pppppppun} in Corollary \ref{coro main theorem different nb rotation} also holds for \emph{any} 
convex domain $\Omega$ whose boundary is $C^2$ and contains some point at which the curvature vanishes. In this case  (see Corollary \ref{coro mather}), for any $\epsilon>0$, there exists $\lambda_0=\lambda_0(\Omega,\epsilon)\in (0,1)$ such that for any $\lambda\in[\lambda_0,1)$, the corresponding Birkhoff attractor $\Lambda_\lambda$ has $\rho^+_\lambda-\rho^-_\lambda\in (1-\epsilon,1)$. \\


From the results presented above, it is possible to highlight a phase transition for Birkhof attractors of dissipative billiards when the parameter $\lambda$ varies. We would like to emphasize how the topological and dynamical properties of the Birkhoff attractor change in terms of the dissipative parameter. 
\noindent From Corollary \ref{cor pert ellipse intro} and Corollary \ref{coro main theorem different nb rotation}, we obtain the following conclusion.
\begin{coralph}\label{coralph ellipse}
	Let $\mathcal{E}$ be an ellipse of eccentricity $e\in (0,1)$. Fix $k \ge 3$. There exists an open and dense set $\mathscr{G}$ of $C^k$ domains such that the following holds. For any $0<\lambda_1<\lambda_2<1$ there exists $\delta>0$ so that if $\Omega\in\mathscr{G}$ and $d_{C^2}(\partial \Omega,\mathcal{E})<\delta$, then:
	\begin{enumerate}
		\item there are $2$-periodic orbits $\{H,f_\lambda(H)\}$ and $\{E,f_\lambda(E)\}$ of saddle and sink type, respectively;
		\item\label{pppppoint un cor}  for any $\lambda\in [\lambda_1,\lambda_2]$,	
		\begin{equation}\label{conclusion point un}\Lambda_\lambda=\overline{\mathcal{W}^u(\mathcal{O}_\lambda(H))}=\mathcal{W}^u(\mathcal{O}_\lambda(H))
			\cup \{E,f_\lambda(E)\} \, .
		\end{equation} 
		In particular, $\Lambda_\lambda$ has rotation number $\frac 12$;
		\item\label{final point phase tran} 
		there exists $\lambda_0 (\Omega)>\lambda_2$ such that, if $\lambda \in [\lambda_0(\Omega),1)$, then $\rho_\lambda^+-\rho^-_\lambda>0$, with $\frac 12 \in (\rho_\lambda^-,\rho^+_\lambda)\mod \Z$. In particular, $\Lambda_\lambda$ is an indecomposable continuum that contains a horseshoe. 
	\end{enumerate}
\end{coralph}

\begin{proof}
	Let $\lambda\in(0,1)\mapsto \epsilon(\mathcal{E},\lambda)>0$ be the continuous function given by Corollary \ref{cor pert ellipse intro}; let $$\delta:=\frac 12 \min_{\lambda \in [\lambda_1,\lambda_2]} \epsilon(\mathcal{E},\lambda)>0\, .$$ Fix $k \geq 2$. Then, for any $C^k$ domain $\Omega$ with $d_{C^2}(\partial \Omega,\mathcal{E})<\delta$, and for any $\lambda \in [\lambda_1,\lambda_2]$, we have $d_{C^2}(\partial\Omega,\mathcal{E})<\epsilon(\mathcal{E},\lambda)$, hence 
	\eqref{conclusion point un} holds for $2$-periodic orbits $\{H,f_\lambda(H)\}$ and $\{E,f_\lambda(E)\}$  
	of saddle and sink type, respectively. Now, by Corollary \ref{coro main theorem different nb rotation}, we obtain point \eqref{final point phase tran} of the Corollary if, moreover, $\Omega$ is chosen within an open and dense set of $C^k$ domains.  
\end{proof}

\noindent We refer to Fig. \ref{phtransition}. The phase transition described for perturbations of elliptic tables also holds for domains in $\mathcal{D}^k$, $k \geq 3$. The following corollary is a consequence of Theorem \ref{main theoreme nh} and Corollary \ref{coro main theorem different nb rotation}.
\begin{coralph}\label{coralph bis}
	For any $k\geq 3$, there exists an open and dense set $\mathscr{U}$ of $C^k$ domains such that for every $\Omega \in \mathcal{D}^k\cap\mathscr{U}$, the following holds. 
	There exist $0<\lambda''(\Omega)<\lambda_0(\Omega)<1$ such that:
	\begin{enumerate}
		\item\label{pppppoint un cor bisbis} if $\lambda\in (0,\lambda''(\Omega))$, then $\Lambda_\lambda$ is equal to the attractor $\Lambda_\lambda^0$ and it is a $C^{k-1}$ normally contracted graph of rotation number $\frac 12$ satisfying 
		$$
		\Lambda_\lambda=\bigcup_{i=1}^\ell \overline{\mathcal{W}^u(\mathcal{O}_\lambda(H_i))}\, ,
		$$
		for some finite collection $\{\mathcal{O}_\lambda(H_i)\}_{i=1,\cdots,\ell}=\{H_i,f_\lambda(H_i)\}_{i=1,\cdots,\ell}$ of $2$-periodic orbits of saddle type;
		\item if $\lambda\in [\lambda_0(\Omega),1)$, then $\rho_\lambda^+-\rho^-_\lambda>0$, with $\frac 12 \in (\rho_\lambda^-,\rho^+_\lambda)\mod \Z$. In particular, $\Lambda_\lambda$ is an indecomposable continuum that contains a horseshoe. 
	\end{enumerate}
\end{coralph}

\begin{figure}[h]
	\centering
	\includegraphics[scale=0.7]
	{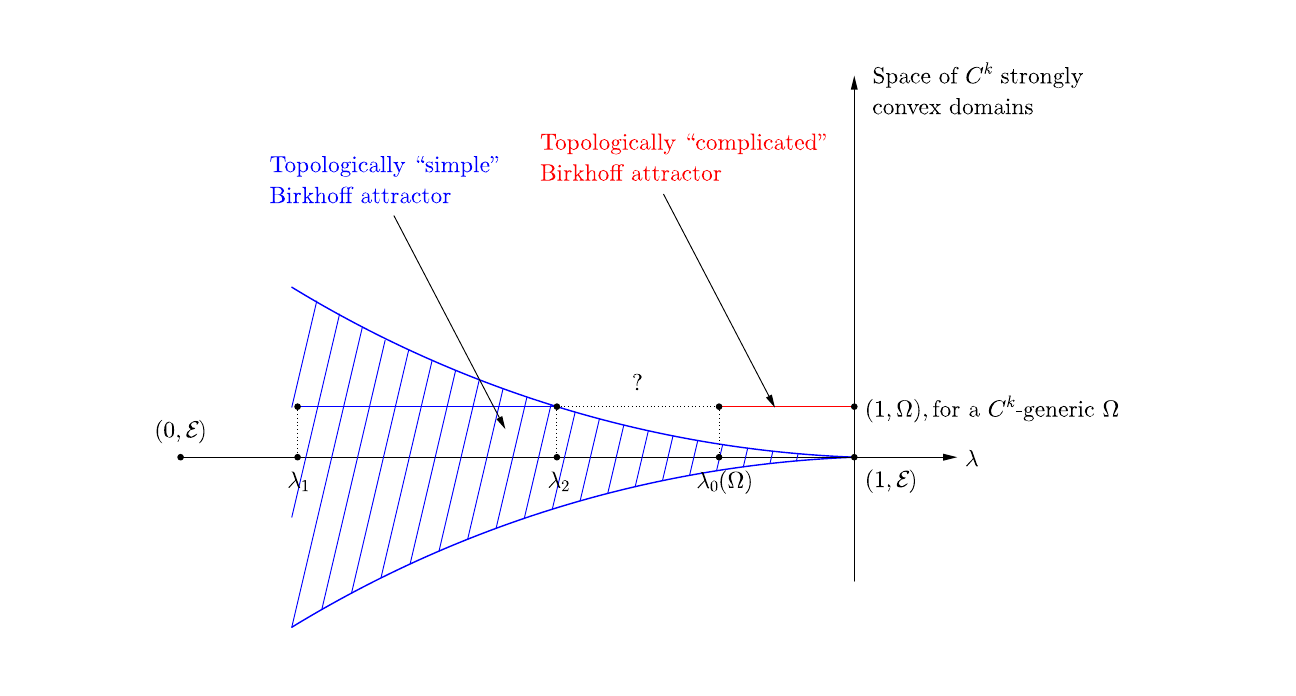}
	\caption{Phase transition for a $C^k$-generic domain near an ellipse of non-zero eccentricity, $k \geq 3$.}
	\label{phtransition}
\end{figure} 

\subsection*{Acknowledgements}

We are grateful to Marie-Claude Arnaud for some useful discussions and suggestions, and to Patrice Le Calvez for suggesting an alternative proof of Proposition \ref{prop different rho}.


\section{Dissipative maps and Birkhoff attractors} \label{PRELIMINARI}

Let $\T:=\R/\Z$ and $\A:=\T\times[-1,1]$, with coordinates $(s,r)\in\A$. Endow $\A$ with the standard 2-form $\omega=dr\wedge ds= d\alpha$, where $\alpha= rds$ is the standard Liouville 1-form. The form $\omega$ induces then an orientation on $\A$ and the Lebesgue measure denoted by $m$. For the following definition of dissipative map we refer to \cite[Page 245]{LeCalvez}.

\begin{definition}\label{def diss map}
	For two continuous maps $\phi^-,\phi^+\colon\T\to\R$ such that $\phi^-<\phi^+$, let us denote
	$$C := \{(s,r)\in\A :\ \phi^-(s)\leq r\leq \phi^+(s)\}.$$
A map $f\colon C\to\mathrm{int}(C)$ is a \emph{dissipative map} if:
	\begin{enumerate}
		\item $f$ is a homeomorphism of $C$ into its image, homotopic to the identity;
		\item $f$ is a $C^1$ diffeomorphism of $\mathrm{int}(C)$ into its image;
		\item\label{cond tttrois} there exists $\lambda\in(0,1)$ such that for any Borel set $Y\subset C$ it holds $m(f(Y))\leq \lambda m(Y)$.
	\end{enumerate}
\end{definition}

\noindent Observe that the following condition is equivalent to condition \ref{cond tttrois} in Definition \ref{def diss map}: 
\begin{itemize}
	\item[\textit{3'.}] there exists $\lambda\in(0,1)$ such that for every $(s,r)\in\mathrm{int}(C)$ it holds $0<\det Df(s,r)\leq\lambda$.
\end{itemize}

\noindent When considering dynamical systems with dissipation, it is natural to mention the notion of conformally symplectic maps, which is stated here in the more general framework of symplectic manifolds. See \cite{MaroSorr,ArnFej,ArnFloRoo}. 

\begin{definition}
	Let $(\mathcal{M},\omega)$ be a symplectic manifold. A diffeomorphism $f$ from $\mathcal{M}$ into its image (contained in $\mathcal{M}$) is \emph{conformally symplectic} if there exists a smooth function $a\colon\mathcal{M}\to\R$ such that $f^*\omega=a\omega$.
\end{definition}

\noindent As shown by Libermann in \cite[Page 210]{Libermann}, 
 if the dimension of $\mathcal{M}$ is greater than or equal to $4$, then the smooth function $a$ is a constant, called the \textit{conformality ratio}. In our case, i.e., for $\dim\mathcal{M}=2$, the function $a$ is not \textit{a priori} constant. This motivates the next definition.
\begin{definition}
	Let $(\mathcal{M},\omega)$ be a symplectic manifold with $\dim\mathcal{M}=2$. A diffeomorphism $f$ from $\mathcal{M}$ into its image (contained in $\mathcal{M}$) is constant conformally symplectic if there exists a constant $a>0$ such that $f^*\omega=a\omega$.
\end{definition}
\noindent Observe that conformally symplectic maps are stable under symplectic changes of coordinates. This is not true anymore for constant conformally symplectic maps.\\

\noindent Let $\mathcal{X}$ be the set of compact sets of $\mathrm{int}(C)$, endowed with the Hausdorff distance. We say that an element $X \in \mathcal{X}$ {\it{separates}} $C$ if its complement has two connected components: a lower one 
$U_{X}\supset \{(s,\phi^-(s))\in\A :\ s\in\T\}$, and an upper one $V_X\supset \{(s,\phi^+(s))\in\A:\ s\in\T\}$. 
We denote by $\mathcal{X}(f)\subset \mathcal{X}$ the 
subset of $\mathcal{X}$ consisting of the sets which are compact, connected, $f$-invariant and separate $C$. \\
\indent Let $f$ be a dissipative map according to Definition \ref{def diss map}. We can define the \emph{Birkhoff attractor} of $f$ as follows, see \cite{Birkhoff1}, \cite[Section 2]{LeCalvez} and \cite[Chapter 6]{LeCalvez1990} for an exhaustive treatment of the argument.
First, observe that, by the dissipative character of $f$ and as $f(C)\subset \mathrm{int}(C)$,  there exists an attractor
\begin{equation} \label{CON}
	\Lambda^0:= \bigcap_{k \ge 0} f^k(C)\,
	\end{equation}
which is $f$-invariant, non-empty, compact and connected. Moreover, it separates $C$, i.e., $C\setminus\Lambda^0$ is the disjoint union of two connected open sets  $U_{\Lambda^0},V_{\Lambda^0}$ as above. In other words, $\Lambda^0 \in \mathcal{X}(f)$. 
\begin{definition}\label{def Birk attr}
	Let $f$ a dissipative map and let $\Lambda^0$ be its corresponding attractor, see \eqref{CON}. Let $U_{\Lambda^0},V_{\Lambda^0}$ be the two connected components of $C\setminus\Lambda^0$. Then, the Birkhoff attractor $\Lambda$ is defined as 
	\begin{equation}
		\Lambda:= \overline{U}_{\Lambda^0}\cap \overline{V}_{\Lambda^0}\, .
	\end{equation}
	\end{definition}
\noindent The Birkhoff attractor of a dissipative map can also be described as the minimal set, with respect to the inclusion, among the elements of $\mathcal{X}(f)$, as we will state immediately.

\begin{proposition}{\cite[Proposition 6.1]{LeCalvez1990}}\label{minimality}
	\begin{enumerate}
		\item\label{minimality Lambda} The set $\mathcal{X}(f)$ contains a minimal element with respect to the inclusion, which is the Birkhoff attractor $\Lambda$ for $f$.
		\item\label{pt due} If $X \in \mathcal{X}(f)$ then $\Lambda =\overline{U}_X \cap \overline{V}_X$; in particular, $\Lambda = \mathrm{Fr}\, U_\Lambda =  \mathrm{Fr} \, V_\Lambda$.
	\end{enumerate}
\end{proposition}

\noindent From the properties of the Birkhoff attractor, we can deduce the following two lemmas.

\begin{lemma}\label{lemma disconnected annulus}
	Let $X\in \mathcal{X}(f)$. Let $x\in X$ be such that $C\setminus(X\setminus\{x\})$ is connected. Then, $x\in\Lambda$.
	\end{lemma}
\begin{proof}
	By Proposition \ref{minimality}\eqref{minimality Lambda}, it holds that $\Lambda\subset X$; in particular, $U_X\subset U_\Lambda$ and $V_X\subset V_\Lambda$. By hypothesis, $U_X\cup V_X\cup\{x\}$ is connected.
	Assume by contradiction that $x\in X\setminus\Lambda\subset C\setminus\Lambda=U_\Lambda\cup V_\Lambda$. Without loss of generality, we can assume that $x\in U_\Lambda$. Therefore
	$$
	U_X\cup\{x\}\cup V_X\subset U_\Lambda\cup V_\Lambda\, .
	$$
The set on the left hand side is connected, while the set on the right hand side is the disjoint union of two open sets. We then conclude that the left term is contained in one of the two open sets. Without loss of generality, we suppose that $U_X\cup\{x\}\cup V_X\subset U_\Lambda$. In particular, $V_X\subset U_\Lambda\cap V_\Lambda=\emptyset$, which is the required contradiction.
\end{proof}
\begin{lemma} \label{X f invariant}
		Let $X\subset \mathrm{int}(\A)$ be $f$-invariant and such that $\Lambda \subset X$. 
		Then $X$ separates the annulus.
	\end{lemma}
	\begin{proof}
		Observe that, since $X\subset \mathrm{int}(\A)$, clearly both $\T\times\{1\}$ and $\T\times\{-1\}$ are contained in $\A\setminus X$. By contradiction, assume that $X$ does not separate the annulus, i.e. $\A\setminus X$ is connected.\footnote{Indeed, it cannot disconnect $\A$ into more than $2$ connected components, by the dissipative character of $f$.} Moreover, since $\Lambda \subset X$, it holds that $\A\setminus X$ is contained in either $U_{\Lambda}$ or $V_{\Lambda}$, the two connected components of $\A\setminus\Lambda$. Without loss of generality, assume that $\A\setminus X\subset U_{\Lambda
	}$. Thus, since $\T\times\{1\}\subset \A\setminus X$, we have $\T\times\{1\}\subset U_{\Lambda}\cap V_{\Lambda}=\emptyset$. This provides the required contradiction and concludes the proof.
	\end{proof}
	
Given a dissipative map $f \colon  C \to \mathrm{int} (C)$ and a hyperbolic periodic point $p\in C$ of $f$ of period $q \geq 1$, we will denote by $\mathcal{W}^s(p;f^q)$ its stable manifold and by $\mathcal{W}^u(p;f^q)$ its  unstable manifold for the iterate $f^q$. 
 Note that $p$ cannot be a source by the dissipative character of $f$. We will also sometimes abbreviate $\mathcal{W}^*(p; f^q)$ simply as $\mathcal{W}^*(p)$ when the information about the map and the period are clear from the context. 
Observe that, for any hyperbolic point in the attractor $\Lambda^0$, its unstable manifold is also contained in $\Lambda^0$. 
Similarly, the next proposition guarantees that for any periodic point of saddle type in the Birkhoff attractor $\Lambda$, certain branches of its stable/unstable manifold have to belong to the Birkhoff attractor. See also \cite[Section 14.3]{LeCalvez} for related results.

\begin{proposition}\label{hyp:stable-unstable} Let $f\colon C \to \mathrm{int} (C)$ be a dissipative map. Assume that $p \in \Lambda$ is a hyperbolic periodic point of saddle type, with period $q \geq 1$. 
	Then at least one the two branches of $\mathcal{W}^u(p;f^q)\setminus \{p\}$ is contained in $\Lambda$, unless $\Lambda$ locally coincides near $p$ with the local stable manifold $\mathcal{W}_{\mathrm{loc}}^s(p;f^q)$ of $p$; in the latter case, $\mathcal{W}^s(p;f^q)\subset \Lambda$. 
In particular, at least one of the following non-exclusive properties holds:
\begin{itemize}
	\item $\overline{\mathcal{W}^s(p;f^q)}\subset \Lambda$; 
	\item $\overline{\mathcal{W}^u(p;f^q)}\subset \Lambda$; 
	\item $\exists\,\delta>0$ such that $\Lambda\cap B(p,\delta)\setminus \{p\}=\mathcal{B}^u \cup \mathcal{B}^s$, where $\mathcal{B}^*$ is a branch of $\mathcal{W}_{\mathrm{loc}}^*(p;f^q)\setminus \{p\}$, $*=s,u$. 
\end{itemize} 
\end{proposition}

\begin{proof}
Let $p \in \Lambda$ be a hyperbolic periodic point of saddle type, with period $q \geq 1$. 
We denote by $0<\mu_1<1<\mu_2$ the eigenvalues of $Df^q(p)$. By Hartman-Grobman Theorem, the dynamics can be linearized near $p$. Let then $U$ be an open neighborhood of $p$, and $\Psi\colon U\to\R^2$ be a homeomorphism  such that $\Psi(p)=0$ and $\Psi\circ f^q|_U= A\circ \Psi|_U$, where
\[
A=\begin{bmatrix}
	\mu_1 & 0 \\ 0 & \mu_2
\end{bmatrix}\, ,\quad 0<\mu_1<1<\mu_2\,.
\]
\indent Let us assume that $\Lambda$ does not locally coincide with the local stable manifold of $p$, that is, for any neighborhood $V$ of $p$ there exists a point in $\Lambda\setminus\mathcal{W}_{\mathrm{loc}}^s(p;f^q)$. Let $\epsilon>0$ be such that the ball of radius $\epsilon$ centred at the origin $B(0,\epsilon)$ is contained in $\Psi(U)$. Then there exists $n_0\in\N$ such that for every $n\geq n_0$ there exists a point $q_n\in\Lambda\cap U$ such that
\begin{itemize}
	\item $\Psi(q_n)=(x_n,y_n)\in B(0,\epsilon)$;
	\item $|y_n|\in(\frac{\epsilon}{2}\mu_2^{-n},\epsilon\mu_2^{-n})$.
\end{itemize}
Observe that each $q_n$ does not belong to the local stable manifold of $p$ and that$f^{qk}(q_n)\in U$ for all $0\leq k\leq n$.
Thus, we can consider the sequence $(f^{qn}(q_n))_n$ of points in $\Lambda$. Denote $\Psi\circ f^{qn}(q_n)=:(X_n,Y_n)$; in particular, for every $n \in \mathbb{N}$ it holds $|X_n|<\epsilon\mu_1^n$ and $|Y_n|\in(\frac{\epsilon}{2},\epsilon)$. Therefore --up to passing to a 
subsequence-- the sequence $(f^{qn}(q_n))_n$ converges to a point $Q\neq p$ belonging to the local unstable manifold of $p$. Since $\Lambda$ is closed and invariant, we conclude that
\begin{equation}\label{points dans lambda un}
	f^{qn}(Q)\in(\Lambda\cap \mathcal{W}^u(p;f^q))\setminus \{p\},\quad \forall\, n \in \mathbb{Z}.
\end{equation}

Let us now prove that the branch $\mathcal{B}^u(p,Q)$ of $\mathcal{W}^u_{\mathrm{loc}}(p;f^q)\setminus \{p\}$ containing $Q$ is contained in $\Lambda$, and therefore, by the invariance of $\Lambda$, that the whole branch of the unstable manifold of $p$ also does. Assume this is not the case: since $\Lambda$ is closed, and since, by \eqref{points dans lambda un}, in any neighborhood of $p$ there exists a point in $\Lambda\cap \mathcal{W}^u_{\mathrm{loc}}(p;f^q)$, for every neighborhood $V\subset U$ of $p$, we can find two points $q_1=\Psi^{-1}(0,a_1),q_2=\Psi^{-1}(0,a_2)\in (\Lambda\cap U)\setminus\{p\}$ belonging to the local unstable manifold of $p$ such that 
\[
\Gamma_0:=\Psi^{-1}(\{ 0 \}\times (a_1,a_2))\subset V,\quad \text{and}\quad \Gamma_0\cap\Lambda=\emptyset.
\] 
Let $\mathcal{U}$ be a bounded neighborhood of $\Lambda$ such that $f^{-1}|_{\mathcal{U}}$ is a diffeomorphism onto its image. Up to considering $V$ small enough, we have that
\[
\bigcup_{n \in \mathbb{N}}f^{-qn}\big(\Lambda\cup \Gamma_0\big)\subset \mathcal{U}\,.
\]
Let $O_0\subset\mathcal{U}$ be a bounded open set among the connected components of $C\setminus(\Lambda\cup \Gamma_0)$, with $\Gamma_0 \subset\partial O_0$. 
Then, the family $(f^{-qn}(O_0))_{n \in \mathbb{N}}$ is uniformly bounded in measure, which contradicts the dissipative character of $f$.
\end{proof}

\noindent With the notations of Proposition \ref{hyp:stable-unstable}, we can show that, in some cases,  $\Lambda$ contains both the stable manifold $\mathcal{W}^s(p;f^q)$ and the unstable manifold $\mathcal{W}^u(p;f^q)$ of the point $p$. In the following, for $*=s,u$, we denote $\mathcal{W}^*(\mathcal{O}_f(p))=\bigcup_{k=0}^{q-1}\mathcal{W}^*(f^k(p);f^q)$. 
\begin{lemma}\label{remark intersection} 
Let $p$ be a periodic hyperbolic point belonging to $\Lambda$. Let us assume that one connected component $\mathcal{B}$ of $\mathcal{W}^s(p;f^q)\setminus\{p\}$ belongs to $\Lambda$, and that $\mathcal{B}$ intersects  
$\mathcal{W}^u(p;f^q)$ 
transversally. 
Then $$\overline{\mathcal{W}^u(\mathcal{O}_f(p))\cup\mathcal{W}^s(\mathcal{O}_f(p))}\subset\Lambda\, . $$ 
\end{lemma}

\begin{proof}
Let $z$ be a point of transverse intersection between $\mathcal{B}$ and 
$\mathcal{W}^u(p;f^q)$. The standard $\lambda$-lemma (see e.g. \cite[Chapter 2.7]{PalisMelo}) guarantees that there exists a $1$-dimensional disk $D\subset \mathcal{B}$ whose past iterates $(f^{-qn}(D))_{n \geq 0}$ accumulate $\mathcal{W}^s(p;f^q)$. 
As $f^{-qn}(D)\subset \Lambda$ for every $n \geq 0$ and since $\Lambda$ is closed, we deduce that $\overline{\mathcal{W}^s(p;f^q)}\subset\Lambda$. Let us denote by $\mathcal{B}'$ the branch of $\mathcal{W}^u(p;f^q)\setminus \{p\}$ containing $z$. By invariance of $\Lambda$, all the points $(f^{-qn}(z))_{n \geq 0}$ belong to $\mathcal{W}^u(p;f^q)\cap \Lambda$; in particular, for any $\delta>0$, $\Lambda\cap B(p,\delta)$ contains points of the branch $\mathcal{B}'\subset\mathcal{W}^u(p;f^q)\setminus \{p\}$. Thus, arguing as in the proof of Proposition \ref{hyp:stable-unstable}, we deduce that  $\mathcal{B}'\subset\Lambda$. Another application of the $\lambda$-lemma (now for a small disk $D'\subset \mathcal{B}'$) gives that in fact, $\overline{\mathcal{W}^u(p;f^q)}\subset \Lambda$. By $f$-invariance of $\Lambda$, we conclude that  $\overline{\mathcal{W}^u(\mathcal{O}_f(p))\cup\mathcal{W}^s(\mathcal{O}_f(p))}\subset\Lambda$. 
\end{proof}

\begin{coro}\label{coro fer a chev}
Let $p$ be a periodic hyperbolic point belonging to $\Lambda$.  Let us assume that for each pair of branches $\mathcal{B}^u\subset\mathcal{W}^u(p;f^q)\setminus \{p\}$ and $\mathcal{B}^s\subset\mathcal{W}^s(p;f^q)\setminus \{p\}$, $\mathcal{B}^u$ and $\mathcal{B}^s$  intersect transversally. Then, $\Lambda$ contains a horseshoe $K(p)$; more precisely, it holds
$$
K(p)\subset H(p):=\overline{\mathcal{W}^s(\mathcal{O}_f(p))\pitchfork \mathcal{W}^u(\mathcal{O}_f(p))}\subset \overline{
	\mathcal{W}^u(\mathcal{O}_f(p))}\subset \Lambda.
$$ 
\end{coro}

\begin{proof}
	The fact that $\mathcal{W}^u(p;f^q)\setminus \{p\}$ and $\mathcal{W}^s(p;f^q)\setminus \{p\}$   intersect transversally guarantees the existence of a horseshoe $K(p)$ which is contained in the homoclinic class $H(p):=\overline{\mathcal{W}^s(\mathcal{O}_f(p))\pitchfork \mathcal{W}^u(\mathcal{O}_f(p))}$ of $p$ (see \cite{SmaleI,SmaleII} and \cite[Section 2]{Newhouse}). If $\Lambda$ contains one branch $\mathcal{B}^s\subset\mathcal{W}^s(p;f^q)\setminus \{p\}$, then by Lemma \ref{remark intersection}, $\overline{\mathcal{W}^s(\mathcal{O}_f(p))\cup \mathcal{W}^u(\mathcal{O}_f(p))}\subset \Lambda$; otherwise, by Proposition \ref{hyp:stable-unstable}, it holds that $\overline{\mathcal{W}^u(p;f^q)}\subset\Lambda$. In either case, $K(p)\subset H(p)\subset \overline{
		\mathcal{W}^u(\mathcal{O}_f(p))}\subset \Lambda$.
\end{proof}

\section{General facts about dissipative billiards} \label{tre} 
\subsection{Dissipative billiard map} \label{DBM} 
	Examples of dissipative maps are given by dissipative billiard maps within convex domains of the plane.  
Let $\Omega$ be a convex domain of the plane $\R^2$ with $C^k$ boundary, $k \geq 2$.  Let $\Upsilon\colon\T\to\R^2$ be an arclength parametrization of $\partial\Omega$. We denote by $f_1$ the conservative billiard map associated to $\Omega$. Now, as in Definition \ref{definit diss bill}, we take a $C^{k-1}$ function  $\lambda\colon \A\to (0,1)$ satisfying condition \eqref{cond lambda dis}.  
As in Definition \ref{definit diss bill}, the dissipative billiard map within $\Omega$ associated to the dissipation function $\lambda$ is then given by
\begin{equation*}
	f_\lambda\colon 
	\left\{
	\begin{array}{rcl}
	\mathbb{A} &\to& \mathbb{A}\, ,\\
	(s,r) &\mapsto& f_\lambda(s,r)=(s',r')\, ,
	\end{array}
	\right.
\end{equation*}
where $\Upsilon(s')$ is the point where the half line, starting at $\Upsilon(s)$ and making an angle $\varphi=\arcsin r$ with the normal, hits the boundary $\partial \Omega$ again, and $r'= \lambda(s',r_1') r_1'$, $r_1'$ being the sine of the outgoing angle of reflection in the case of an elastic collision. 
Letting $\mathcal{H}_\lambda \colon \A\ni (s,r)\mapsto (s,\lambda(s,r) r)\in\A$, we observe that 
	\begin{equation}\label{dec f lambda}
	f_\lambda=\mathcal{H}_\lambda\circ f_1\, .
	\end{equation} 
Here are some basic properties of the dissipative billiard map.
\begin{enumerate}
	\item\label{perp bounces} The equality $r' = r_1'$ happens if and only if $r' = r_1'= 0$, i.e., the bounce at $x'$ is perpendicular. 
	\item We recall that the standard billiard map $f_1$ preserves the area form $\omega = d r \wedge ds = d\alpha$, where $\alpha$ denotes the 1-form $r\,ds$. 
	Thus, by \eqref{dec f lambda}, $f_{\lambda}$ is a conformally symplectic map; indeed, 
	$$
	f_{\lambda}^*\omega=\big(\partial_r \lambda(s',r_1')r_1'+\lambda(s',r_1')\big)\omega.
	$$ 
	In particular, if $\lambda$ is a constant function, then $f_\lambda$ is a conformally symplectic map with constant factor $\lambda$, as  
	$f_{\lambda}^*\omega= f_1^*(\mathcal{H}_\lambda^*\omega) = \lambda \omega$. 
	\item The map $f_{\lambda}$ is a dissipative map of $\A$, according to Definition \ref{def diss map}. 
		In particular, $f_{\lambda}$ verifies: 
		\begin{enumerate}[label=(\roman*)]
			\item for any $(s,r) \in \mathrm{int}(\mathbb{A})$, it holds
			\begin{equation}\label{determinant appl billard}
			0 < \det Df_{\lambda}(s,r) =\partial_r \lambda(s',r_1')r_1'+\lambda(s',r_1') < 1\,;
			\end{equation} 
			in particular, if $\lambda$ is constant, then for any $(s,r) \in \mathrm{int}(\mathbb{A})$, it holds $0 < \det Df_{\lambda}(s,r) =\lambda< 1$; 
			\item it holds $f_{\lambda}(\mathbb{T} \times [-1,1])\subset \mathbb{T} \times (-1,1)$; in particular, if $\lambda$ is constant, then
			\begin{equation} \label{inside}
			f_{\lambda}(\mathbb{T} \times [-1,1]) = \mathbb{T}\times[-\lambda,\lambda]\subset \mathbb{T} \times (-1,1)\,.
			\end{equation} 
	\end{enumerate}
	Moreover, we will see in Section \ref{section different rho} that the map $f_\lambda$ is a positive twist map. 
	\item Again from \eqref{dec f lambda}, for every $(s,r)\in\mathrm{int}(\A)$, and $(s',r'):=f_\lambda(s,r)$, we have
	\begin{align}
	Df_{\lambda}(s,r) &=D\mathcal{H}_\lambda(f_1(s,r))\, Df_1(s,r) \nonumber\\
	&=
	\begin{bmatrix}
		1 & 0\\
		\partial_s \lambda(s',r_1') & \partial_r \lambda(s',r_1')r_1'+\lambda(s',r_1')
	\end{bmatrix}
	\begin{bmatrix} \label{matrice differ}
	-\frac{\tau \mathcal{K} + \nu}{\nu'} & \frac{\tau}{\nu \nu'} \\ \\
	\tau \mathcal{K} \mathcal{K}' + \mathcal{K} \nu'+ \mathcal{K}' \nu & -\frac{\tau \mathcal{K}'+ \nu'}{\nu}
	\end{bmatrix},
	\end{align}
	where $\tau=\ell(s,s'):=\|\Upsilon(s)-\Upsilon(s')\|$ is the Euclidean distance between the points $x=\Upsilon(s)$, $x'=\Upsilon(s')$, $\mathcal{K}$, $\mathcal{K}'$ denote the curvatures at $\Upsilon(s),\Upsilon(s')$ respectively and $\nu = \sqrt{1 - r^2}$, $\nu' = \sqrt{1 - (r_1')^2}=\sqrt{1 - \left(\frac{r'}{\lambda}\right)^2}$. Formula \eqref{matrice differ} can be deduced from \cite[Section 2.11]{CheMar_book}, by applying the change of coordinates $(s,\varphi)\mapsto (s,r=\sin \varphi)$. 
	
	If the dissipation is constant, equal to $\lambda \in (0,1)$, then with same notations as above, for   $(s,r)\in\mathrm{int}(\A)$,
	\begin{equation*}
				Df_{\lambda}(s,r) =
				\begin{bmatrix} \label{matrice differ}
						-\frac{\tau \mathcal{K} + \nu}{\nu'} & \frac{\tau}{\nu \nu'} \\ \\
						\lambda\big(\tau \mathcal{K} \mathcal{K}' + \mathcal{K} \nu'+ \mathcal{K}' \nu \big) & -\lambda \frac{\tau \mathcal{K}'+ \nu'}{\nu}
					\end{bmatrix}\,.
	\end{equation*}
\end{enumerate} 

	\begin{notation}
		In the sequel, while considering the dissipative billiard map $f_\lambda$, we denote its attractor by $\Lambda_\lambda^0$ and its Birkhoff attractor by $\Lambda_\lambda$.
	\end{notation}


\subsection{Properties of the Birkhoff attractor for axially symmetric billiards}

This subsection is devoted to proving some properties of the Birkhoff attractor for a dissipative billiard map under some symmetric assumptions on the domain $\Omega$. In particular, they can be applied to the case of elliptic billiards considered in Section \ref{quattro}. In this subsection, we assume that the billiard map $f_\lambda$ has constant dissipation $\lambda\in(0,1)$. In fact, this would more generally as long as the dissipation function $\lambda \colon \A\to (0,1)$ respects the symmetries of $\Omega$. 
\begin{definition}[Axially symmetric billiard table]
Let $\Omega\subset \R^2$ be a strictly convex domain with $C^2$ boundary. We say that $\Omega$ is axially symmetric with respect to some line $\Delta\subset\R^2$ if $\Omega$ is invariant under the reflection along the line $\Delta$. 
\end{definition}


\begin{lemma}\label{lemme sym axiale}
		Let $\Omega\subset \R^2$ be a strictly convex domain with $C^2$ boundary, with perimeter $2L>0$, which is axially symmetric with respect to some line $\Delta$. Let $s_0$ and $s_0+L$ the arclength parameters of the  points in $\partial\Omega\cap \Delta$. For $\lambda \in (0,1)$, let $f_\lambda$ be the associated dissipative billiard map. Then:
		\begin{enumerate}
			\item\label{point un sym ax} the pair $\{(s_0,0),(s_0+L,0)\}$ corresponds to a $2$-periodic orbit; 
			\item let us denote by $\mathcal{I}_\Delta$ the involution $\mathcal{I}_\Delta\colon (s_0+s,\varphi)\mapsto (s_0-s,-\varphi)$; then,
			\begin{equation}\label{point un sym ax bis} 
			\mathcal{I}_\Delta \circ f_\lambda = f_\lambda \circ \mathcal{I}_\Delta. 
			\end{equation} 
		\end{enumerate}  
\end{lemma}

\begin{proof}
	Item \eqref{point un sym ax} follows by noticing that the osculating circles at $s_0$ and $s_0+L$ are invariant under $\mathcal{I}_\Delta$, hence the line segment connecting the points $s_0,s_0+L$ (which is collinear to $\Delta$) is perpendicular to the boundary of $\Omega$ at the points $s_0,s_0+L$. The second point is immediate once we have observed that $f_\lambda=\mathcal{H}_\lambda\circ f_1$, as $\mathcal{H}_\lambda$ commutes with $\mathcal{I}_\Delta$ and, under the assumption of axial symmetry, $f_1$ and $\mathcal{I}_\Delta$ also commute. 
\end{proof}

\begin{corollary} \label{COR a}
	Let $\Omega\subset \R^2$ be a strictly convex domain with $C^2$ boundary with perimeter $2L>0$, axially symmetric with respect to some line $\Delta$. Let $s_0$ and $\mathcal{I}_\Delta$ be the arclength parameters of the points in $\partial\Omega\cap\Delta$. For $\lambda \in (0,1)$, let $f_\lambda$ be the associated dissipative billiard map and $\Lambda_\lambda$ the corresponding Birkhoff attractor. Then:
	\begin{enumerate}
		\item\label{point un bill centr bir} $\mathcal{I}_\Delta(\Lambda_\lambda)=\Lambda_\lambda$;
		\item\label{point deux bill centr bir} $\{(s_0,0),(s_0+L,0)\}\subset \Lambda_\lambda$. 
	\end{enumerate}
\end{corollary}

\begin{proof}
	Item \eqref{point un bill centr bir} follows from the fact that $\Lambda_\lambda$ is the smallest, with respect to inclusion, compact connected $f_\lambda$-invariant subset which separates the annulus. Indeed, $\mathcal{I}_\Delta(\Lambda_\lambda)$ is compact, connected; it is $f_\lambda$-invariant, by the fact that $\Lambda_\lambda$ is $f_\lambda$-invariant, and by \eqref{point un sym ax bis}; moreover, $\mathcal{I}_\Delta(\Lambda_\lambda)$ separates the annulus. Thus, $\Lambda_\lambda\subset \mathcal{I}_\Delta(\Lambda_\lambda)$. We conclude because $\mathcal{I}^2_\Delta=\mathrm{Id}$ and so $\mathcal{I}_\Delta(\Lambda_\lambda)\subset \mathcal{I}^2_\Delta(\Lambda_\lambda)=\Lambda_\lambda$.
	
	To show point \eqref{point deux bill centr bir}, let us argue by contradiction, assuming that $(s_0,0)\notin \Lambda_\lambda$. By compactness of $\Lambda_\lambda$, there exists a connected open neighborhood $U$ of $(s_0,0)$ that is disjoint from $\Lambda_\lambda$. Since $\mathcal{I}_\Delta(s_0,0)=(s_0,0)$, the set $U':=U\cap \mathcal{I}_\Delta(U)$ is a connected open neighborhood of $(s_0,0)$ that is $\mathcal{I}_\Delta$-invariant and disjoint from $\Lambda_\lambda$. Recall that $\Lambda_\lambda$ separates the annulus $\mathbb{A}$, i.e.,  $\mathbb{A}\setminus\Lambda_\lambda$ is the disjoint union of two open connected components, denoted by $U_{\Lambda_\lambda}$ and $ V_{\Lambda_\lambda}$. Since $\mathcal{I}_\Delta$ maps the top boundary $\T\times\{1\}$ to  the bottom boundary $\T\times\{-1\}$ and viceversa,  we have $\mathcal{I}_\Delta(U_{\Lambda_\lambda})=V_{\Lambda_\lambda}$. As $U'\cap\Lambda_\lambda=\emptyset$, we can assume without loss of generality that $U' \subset U_{\Lambda_\lambda}$. But then, $U'=\mathcal{I}_\Delta(U')\subset\mathcal{I}_\Delta(U_{\Lambda_\lambda})= V_{\Lambda_\lambda}$, which is a contradiction. Thus, $(s_0,0)\in \Lambda_\lambda$, and $(s_0+L,0)=f_\lambda(s_0,0)\in \Lambda_\lambda$.
\end{proof}

%
%

\subsection{Bifurcations of $2$-periodic points} \label{conti 2 orbite}

As we are going to see, $2$-periodic points play a special role for dissipative billiards; this is partly due to the fact that by point \eqref{perp bounces} in Subsection \ref{DBM}, the usual reflection law and the dissipative one have the same effect at an orthogonal collision. In the previous section, we already saw that for convex billiards with symmetries, symmetric $2$-periodic orbits have to belong to the Birkhoff attractor. Here we investigate the eigenvalues of $2$-periodic points and their bifurcations as the dissipation parameter changes. In fact, as we will see here, although the set of $2$-periodic orbits is independent of the value of the dissipation $\lambda\colon \A\to (0,1)$, their type will depend on the strength of the perturbation.  Throughout the rest of Subsection \ref{conti 2 orbite}, except in Lemma \ref{lemma noncst disp} and in the last point of Corollary \ref{nondeg billiard in dk}, we will assume for simplicity that the  dissipation is constant, equal to some $\lambda \in (0,1)$. Yet, the result of Lemma \ref{lemme vp reelles} could be fully adapted to the case of a non-constant dissipation $\lambda$, but the proof would be even more computational; the version we give in Lemma \ref{lemma noncst disp} is a little less precise but suffices for our purpose. 
The proofs of the main technical lemma of this subsection, namely  Lemma \ref{lemme vp reelles} and Lemma \ref{lemma noncst disp}  are mainly computational: for this reason, we postpone them to Appendix \ref{appendix proof bifurcation eigenvalues}. 


Let us denote by $\mathrm{II}$ the set of $2$-periodic points for $\{f_\lambda\}_{\lambda \in [0,1]}$.  
In the following, for $p=(s,0) \in \mathrm{II}$,  let $(s',0):=f_\lambda(p)$ (the point $(s',0)$ is independent of the value of $\lambda$ as observed above), and denote by $\tau=\ell(s,s'):=\|\Upsilon(s)-\Upsilon(s')\|$ the Euclidean distance between the points $\Upsilon(s),\Upsilon(s')$. We also denote by $\mathcal{K}_1$, $\mathcal{K}_2$ the respective curvatures at $\Upsilon(s),\Upsilon(s')$, and let 
\begin{equation}\label{k un deuxx}
	k_{1,2}=k_{1,2}(p):=(\tau \mathcal{K}_1 + 1)(\tau \mathcal{K}_2 + 1).
\end{equation}
\begin{lemma}\label{lemme vp reelles}
	Let $p\in \mathrm{II}$, and let $\tau,\mathcal{K}_1,\mathcal{K}_2,k_{1,2}$ be as above. Fix $\lambda \in (0,1)$, and denote by $\{\mu_1=\mu_1(\lambda),\mu_2=\mu_2(\lambda)\}$  the eigenvalues of $Df_\lambda^2(p)$, with $|\mu_1|\leq |\mu_2|$. 
\begin{enumerate}[label=(\alph*)]
\item\label{item aa en} If $k_{1,2}>1$, then $0<\mu_1<\lambda^2<1<\mu_2$, and the $2$-periodic orbit $\{p,f_\lambda(p)\}$ is a \emph{saddle}. 
\item\label{item bb en} If $k_{1,2}=1$, then $\mu_1=\lambda^2$, $\mu_2=1$, and the $2$-periodic orbit $\{p,f_\lambda(p)\}$ is \emph{parabolic}.  
\item\label{item cc en} If $k_{1,2}\in (0,1)$, then the $2$-periodic orbit $\{p,f_\lambda(p)\}$ is a \emph{sink};  moreover, let
\begin{equation} \label{serve dopo}
\lambda_-=\lambda_-(p):=\frac{1-\sqrt {1-k_{1,2}}}{1+\sqrt {1-k_{1,2}}}\in (0,1).
\end{equation}
It holds:
\begin{enumerate}[label=(\roman*)]
	\item if $\lambda \in (0,\lambda_-)$, then $\mu_1,\mu_2$ are real, with $\lambda^2<\mu_1< \mu_2 <1$;
	\item if $\lambda=\lambda_-$, then $\mu_1=\mu_2=\lambda\in (0,1)$;
	\item if $\lambda \in (\lambda_-,1)$, then $\mu_1,\mu_2$ are complex conjugate of modulus $\lambda$.
\end{enumerate} 
\item \label{item ii en} If $k_{1,2}=0$, then the $2$-periodic orbit $\{p,f_\lambda(p)\}$ is a \emph{sink}, with $\mu_1=\mu_2=-\lambda$, and  $Df_\lambda^2(p)=-\lambda\, \mathrm{id}$. 
\item\label{item dd en} If $k_{1,2}\in (-1,0)$, 
let 
$$
\bar \lambda=\bar \lambda(p):=\frac{1-\sqrt{-k_{1,2}}}{1+\sqrt{-k_{1,2}}}\in (0,1).
$$
It holds:
\begin{enumerate}[label=(\roman*)]
	\item if $\lambda \in (0,\bar \lambda)$, then $-1<\mu_2<-\lambda<\mu_1<0$, and the $2$-periodic orbit $\{p,f_\lambda(p)\}$ is a \emph{sink};
	\item\label{ppoint parab cas deux} if $\lambda=\bar \lambda$, then $\mu_1=-\lambda^2$, $\mu_2=-1$, and the $2$-periodic orbit $\{p,f_\lambda(p)\}$ is \emph{parabolic};
	\item if $\lambda \in (\bar \lambda,1)$, $\mu_2<-1<-\lambda^2<\mu_1<0$, and the $2$-periodic orbit $\{p,f_\lambda(p)\}$ is a \emph{saddle}.  
\end{enumerate} 
\item\label{il reste un cas} If $k_{1,2}\leq -1$, then $\mu_2<-1<-\lambda^2<\mu_1<0$, and the $2$-periodic orbit $\{p,f_\lambda(p)\}$ is a \emph{saddle}.
\end{enumerate}
Moreover, for $\lambda=0$, the eigenvalues of $Df_0^2(p)$ are $\mu_1=0$ and $\mu_2=k_{1,2}$, and the respective eigenspaces are vertical and horizontal. 
\end{lemma}
\begin{proof}
	See Appendix \ref{appendix proof bifurcation eigenvalues}. 
\end{proof}

\noindent \begin{remark}\label{lemme vp reelles sym}
	We use the same notations as above. In the special case of a point $p \in \mathrm{II}$ with  $\mathcal{K}_1=\mathcal{K}_2=:\mathcal{K}\leq 0$, we have $k_{1,2}=(1+\tau \mathcal{K})^2\geq 0$, and $\tau \mathcal{K}\leq 0$, hence the previous result gives the following outcome.
	\begin{enumerate}[label=(\alph*)]
		\item\label{item aa en bis} If $\tau\mathcal{K}<-2$, then the 
		$2$-periodic orbit $\{p,f_\lambda(p)\}$ is a saddle. 
		\item\label{item bb en bis} If $\tau\mathcal{K}\in\{-2,0\}$, then $\mu_1=\lambda^2$,  $\mu_2=1$ and the 
		$2$-periodic orbit $\{p,f_\lambda(p)\}$ is parabolic.
		\item\label{item cc en bis} If $-2<\tau\mathcal{K}< 0$, then the 2-periodic orbit $\{p,f_\lambda(p)\}$ is a sink; for $\tau\mathcal{K}=-1$, it holds $Df_\lambda^2(p)=-\lambda\,\mathrm{id}$.
	\end{enumerate}
\end{remark}
 
In the case of dissipative billiard maps associated to a non-constant dissipation $\lambda \in (0,1)$, we also have:
\begin{lemma}\label{lemma noncst disp}
	Let $\lambda\colon \A\to (0,1)$ be a $C^{k-1}$ function such that  $f_\lambda:=\mathcal{H}_\lambda \circ f$ is a dissipative billiard map in the sense of Definition \ref{cond lambda dis}, where  $\mathcal{H}_\lambda\colon (s,r)\mapsto(s,\lambda(s,r) r)$. In particular, $f_\lambda$ has the same set $\mathrm{II}$ of $2$-periodic points as $f$. Fix a $2$-periodic orbit $\{p,f_\lambda(p)\}$, and assume that $k_{1,2}\geq 0$, with $k_{1,2}=k_{1,2}(p):=(\tau \mathcal{K}_1 + 1)(\tau \mathcal{K}_2 + 1)$ as in \eqref{k un deuxx}. Then $\{p,f_\lambda(p)\}$  is parabolic for $f_\lambda$ if and only if it is for the conservative billiard map $f=f_1$, and this happens if and only if $k_{1,2}=1$. 
	
	Otherwise, the orbit $\{p,f_\lambda(p)\}$ is either a sink or a saddle for $f_\lambda$; more precisely, it is a saddle if and only if $k_{1,2}>1$, and it is a sink if and only if $k_{1,2}<1$.
\end{lemma}
\begin{proof}
	See Appendix \ref{appendix proof bifurcation eigenvalues}. 
\end{proof}

In the next statement, we summarize some results obtained by Dias Carneiro, Oliffson Kamphorst and Pinto-de-Carvalho, see \cite[Theorem 1]{PintoC} and \cite{CarneiroOliffsonPC}, and Xia-Zhang \cite[Theorem 1.1-Corollary 4.4]{XiaZhang}. 

\begin{theorem}\label{nondeg pc}
	Fix $k \geq 2$. 
	\begin{enumerate}
		\item For every $q\geq 2$, there exists an open and dense set $\mathscr{U}^q$ of strongly convex domains with $C^k$ boundary such that the number of $q$-periodic points (for the usual conservative billiard map) is finite; moreover, all the $q$-periodic points are either elliptic or hyperbolic. 
		\item There exists a $G_\delta$-dense set $\mathcal{G}^k$ of strongly convex domains with $C^k$ boundary such that, for each $q \ge 2$, the number of $q$-periodic points is finite. Moreover, all the periodic points are either hyperbolic or elliptic.
		\item There exists an open and dense set $\mathscr{U}$ of strongly convex domains with $C^k$ boundary, $k\geq 3$, such that for each $2$-periodic point $p \in \mathrm{II}$ of saddle type, each branch of $\mathcal{W}^s(p;f^2)\setminus \{p\}$ and $\mathcal{W}^u(p;f^2)\setminus \{p\}$ contains a transverse homoclinic point.
		\end{enumerate}
\end{theorem}

\begin{coro}\label{nondeg billiard in dk}
	Let $k \geq 3$. There exists an open and dense set $\mathscr{U}$ of strongly convex domains with $C^k$ boundary such that, for every $\Omega\in\mathscr{U}$, the following assertions hold.
	
	\noindent If $f_\lambda$ is a dissipative billiard maps with constant dissipation $\lambda \in (0,1)$, then:
	\begin{enumerate} 
		\item for any $\lambda \in [0,1]$, the set $\mathrm{II}$ of $2$-periodic points of $f_\lambda$ is finite;
		\item for all but at most finitely many $\lambda \in (0,1)$, all the $2$-periodic points are non-degenerate, i.e., they are either saddles or sinks;
		\item\label{ppppppoint trois} there exists $\lambda_*(\Omega)\in (0,1)$ such that for any $\lambda \in [\lambda_*(\Omega),1)$, and for any point $p \in \mathrm{II}$ of saddle type, each branch of $\mathcal{W}^s(p;f_\lambda^2)\setminus \{p\}$ and $\mathcal{W}^u(p;f_\lambda^2)\setminus \{p\}$ contains a transverse homoclinic point.
		\end{enumerate} 
	\noindent If, moreover, $\Omega \in \mathcal{D}^k \cap \mathscr{U}$, where $\mathcal{D}^k$ is the subset of strongly convex domains with $C^k$ boundary as in Definition \ref{defi set d k}, then for a general dissipative billiard map $f_\lambda$ in the sense of Definition \ref{definit diss bill} (with possibly non-constant dissipation),
		all the $2$-periodic points are either saddles or sinks. It is also true for the (degenerate) map $f_0$, namely, when the dissipation $\lambda$ vanishes.  
\end{coro}

\begin{proof}
By the definition of the (dissipative) billiard law given in Definition \ref{definit diss bill}, given a convex domain, the set $\mathrm{II}$ of $2$-periodic points is common to all the maps $\{f_\lambda\}_{\lambda \in [0,1]}$. By Theorem \ref{nondeg pc}, we deduce that there exists an open and dense set $\mathscr{U}$ of strongly convex $C^k$ domains, $k\geq 3$, such that, for every $\Omega\in\mathscr{U}$ 
and for any $\lambda \in [0,1]$, the billiard map $f_\lambda$ has finitely many $2$-periodic points. Fix $\Omega\in \mathscr{U}$, and denote by $\mathrm{II}$ its finite set of $2$-periodic points. 
	For any $p \in \mathrm{II}$, let $\tau,\mathcal{K}_1,\mathcal{K}_2$ be as in Lemma \ref{lemme vp reelles}, and let $k_{1,2}=k_{1,2}(p):= (\tau\mathcal{K}_1+1)(\tau\mathcal{K}_2+1)$. For any $\lambda\in (0,1)$, the $2$-periodic $f_\lambda$-orbit $\{p,f_\lambda(p)\}$ is a saddle or a sink, unless $k_{1,2}=1$ (see Lemma \ref{lemme vp reelles}\ref{item bb en}), or $k_{1,2}\in (-1,0)$ and $\lambda=\bar \lambda(p)$ (see Lemma \ref{lemme vp reelles}\ref{item dd en}\ref{ppoint parab cas deux}). On the one hand, let us examine the case where $k_{1,2}=1$ for the $2$-periodic orbit $\{p,f_1(p)\}$ of the conservative billiard map $f_1$.  
 By equation \eqref{trace d f lambda carre} for $\lambda=1$, $k_{1,2}=1$ if and only if $\mathrm{tr} Df_1^2(p)=2$, i.e., the $2$-periodic $f_1$-orbit $\{p,f_1(p)\}$ is parabolic, which does not occur for the domain $\Omega$, since it is in $\mathscr{U}$. On the other hand, again since $\Omega\in\mathscr{U}$, $\mathrm{II}$ is finite, hence so is the set 
$$
\mathcal{F} := \bigcup_{p\in \mathrm{II}: \ k_{1,2}(p)\in (-1,0)} \{\bar\lambda(p)\}\subset (0,1)\,.
$$ 
By the above discussion, and by Lemma \ref{lemme vp reelles}, we conclude that for any $p\in \mathrm{II}$, 
and for any $\lambda \in (0,1)\setminus \mathcal{F}$, the $2$-periodic orbit $\{p,f_\lambda(p)\}$ is non-degenerate, i.e. it is either a saddle or a sink.

Point \eqref{ppppppoint trois} follows immediately from Theorem \ref{nondeg billiard in dk}; indeed, for the conservative billiard map $f_1$, there exist transverse homoclinic points on each of the branchs of any $2$-periodic point of saddle type; the existence of transverse homoclinic points is stable under $C^1$-small perturbations of the dynamics, hence the same property holds true for any $f_\lambda$ with $\lambda \in (0,1)$ close enough to $1$. 

\indent Finally assume that $\Omega$ belongs to the set $\mathcal{D}^k\cap \mathscr{U}$. Let $f_\lambda$ be a general dissipative billiard map for $\Omega$ in the sense of Definition \ref{cond lambda dis}.  It has the same set $\mathrm{II}$ of $2$-periodic points as $f_1$. 
	By condition \eqref{condition geom norm hyp}, for any $p \in \mathrm{II}$, we have  $k_{1,2}(p)> 0$. Moreover, $k_{1,2}(p)\neq 1$, since $\Omega\in \mathscr{U}$, hence by Lemma \ref{lemma noncst disp}, 
	the $2$-periodic orbit $\{p,f_\lambda(p)\}$ is either a saddle or a sink.  Besides, for $\lambda=0$, Lemma \ref{lemme vp reelles} says that the eigenvalues of $p$ are $\mu_1=0$ and $\mu_2=k_{1,2}(p)$; as $k_{1,2}(p)>0$ and $k_{1,2}(p)\neq 1$, we deduce that the $2$-periodic $\{p,f_0(p)\}$ of $f_0$ is either a saddle or a sink. 
\end{proof}

Let us remind some classical definitions and results for conservative billiard maps. Let $\Omega$ be a strongly convex domain with $C^2$ boundary. Then, the billiard map expressed in coordinates $(s,\varphi)\in \T\times[-\frac{\pi}{2},\frac{\pi}{2}]$ is a $C^1$ diffeomorphism, see \cite[Proposition I.3.2]{Douady} and \cite[Page 11]{LeCalvez1990}.
Let $F_1\colon\R\times[-1,1]\to\R\times[-1,1]$ be a lift of the conservative billiard map. Let $\tilde\pi_1\colon\R\times[-1,1]\to\R$ be the projection onto the first coordinate. Then the $F_1$-orbit of a point $(S,r)$ is completely determined by the bi-infinite sequence $(S_i)_{i\in\Z}:=(\tilde\pi_1\circ F_1^i(S,r))_{i\in\Z}$.
\begin{definition}
	Let $(S,r)\in\R\times(-1,1)$. The rotation number of $(S,r)$ is
	$$
	\rho(S,r):=\lim_{n\to \infty}\dfrac{\tilde\pi_1\circ F_1^n(S,r)}{n} \, ,
	$$
	whenever the limit exists. 
\end{definition}
\noindent Observe that the rotation number depends on the chosen lift $F_1$. Up to the choice of the lift, i.e., a lift such that $F_1$ is the identity on the lower boundary $\R\times\{-1\}$,  the rotation number of any point belongs to the interval $[0,1]$. Observe that, for such a lift, the rotation number of a periodic (conservative) billiard trajectory corresponds to $\frac{\text{winding number}}{\text{number of reflections}}$.

Let $\ell\colon\R^2\to\R$ be the generating function of the conservative billiard map. From the geometric point of view, the quantity $\ell(S_i,S_{i+1})$ corresponds to the Euclidean distance on $\R^2$ between $\Upsilon(s_i)$ and $\Upsilon(s_{i+1})$, where $\pi\colon\R\to\T$ is a covering and $s_i=\pi(S_i)$. This is also the quantity previously denoted as $\tau(s_i,s_{i+1})$. In the next proposition, with an abuse of notation, since $\ell$ is invariant under the action of $\Z$, we also denote by $\ell$ the function induced on $\T^2$.  \\
By a standard construction due to Birkhoff, see e.g. \cite[Theorem 1.2.4]{Siburg}, it is well-known that there exist at least two periodic orbits for every rational rotation number. They are obtained by considering the length functional, given by $\sum_{i \in \mathbb{Z}} \ell(S_i,S_{i+1})$. In particular, the first orbit is given by maximizing the functional, while the other one is given by a min-max procedure (sometimes referred to as the \textit{``Mountain Pass Lemma''}). In particular, for the rotation number $\frac 1 2$, we obtain two $2$-periodic orbits for the conservative billiard map. Then --as remarked at the beginning of the section-- the dissipative billiard map $f_\lambda$ has two 2-periodic orbits for any $\lambda\in [0,1]$. More precisely, for every $\lambda\in(0,1)$, the set of 2-periodic points is non empty and it contains at least 2 different orbits. 
\begin{proposition}\label{remark point crit func}
	Let $f_\lambda$ be the dissipative billiard map of a strongly convex $C^k$domain $\Omega$ with $C^k$ boundary, $k \geq 2$, that belongs to the open and dense set $\mathscr{U}$ of Theorem \ref{nondeg pc}. Assume that $\{p=(s_1,0),f_\lambda(p)=(s_2,0)\}$ is a $2$-periodic orbit. We denote by $\mathcal{K}_1,\mathcal{K}_2<0$ the respective curvatures at the points $\Upsilon(s_1)$ and $\Upsilon(s_2)$, where 
	$\Upsilon\colon\T\to\R^2$ is an arclength parametrization of the boundary. 
	Let $\tau:=\ell(s_1,s_2)$ and $k_{1,2}:=(\tau \mathcal{K}_1 + 1)(\tau \mathcal{K}_2 + 1)$. 
	Denote by $\{\mu_1,\mu_2\}$ the eigenvalues of $Df_\lambda^2(p)$, with $|\mu_1|\leq |\mu_2|$. Then: 
	\begin{enumerate}[label=(\alph*)]
		\item\label{premier ppooiinntt} if $(s_1,s_2)$ corresponds to a 
		local maximum of $\ell$ 
		(e.g. when $[\Upsilon(s_1),\Upsilon(s_2)]$ is a diameter), then $k_{1,2}>1$, and  
		the $2$-periodic orbit $\{p,f_\lambda(p)\}$ is a saddle;
		\item\label{sssecond ppooiinntt} if $(s_1,s_2)$ corresponds to a 
		critical point of saddle type of $\ell$, 
		then $k_{1,2}<1$, and 
		it holds:
		\begin{enumerate}[label=(\roman*)]
			\item if $k_{1,2}\geq 0$\footnote{Note that it is always the case when $\mathcal{K}_1=\mathcal{K}_2$, or when $\Omega \in \mathcal{D}^k$.}, then the $2$-periodic orbit $\{p,f_\lambda(p)\}$ is a sink;
			\item if $k_{1,2}\in (-1,0)$, 
			let
			$
			\bar \lambda=\bar \lambda(p):=\frac{1-\sqrt{-k_{1,2}}}{1+\sqrt{-k_{1,2}}}\in (0,1)
			$; 
			then, we have:
			\begin{itemize} 
				\item for any $\lambda \in (0,\bar \lambda)$, the $2$-periodic orbit $\{p,f_\lambda(p)\}$ is a sink;
				\item for $\lambda=\bar \lambda$, the $2$-periodic orbit $\{p,f_\lambda(p)\}$ is parabolic;
				\item for any $\lambda \in (\bar \lambda,1)$, the $2$-periodic orbit $\{p,f_\lambda(p)\}$ is a saddle; 
			\end{itemize} 
			\item  if $k_{1,2}\leq-1$, then for any $\lambda \in (0,1)$, the $2$-periodic orbit $\{p,f_\lambda(p)\}$ is a saddle. 
			\end{enumerate}
	\end{enumerate}
\end{proposition}

\begin{proof}
	Fix a  $2$-periodic orbit $\{p=(s_1,0),f_\lambda(p)=(s_2,0)\}$. 
	Since $\partial_1 \ell(s,s')=-\sin \varphi$ and $\partial_2 \ell(s,s')=\sin \varphi'$, 
the point $(s_1,s_2)$ is a critical point of $\ell$. Moreover, 
we have (see e.g. \cite[Lemma 2.1]{KaloshinZhang})
	$$
	\ell(s_1+\delta s,s_2+\delta s')-\ell(s_1,s_2)=\frac 12\begin{bmatrix} \delta s & \delta s'
	\end{bmatrix}\underbrace{\begin{bmatrix}
		\mathcal{K}_1+\frac{1}{\tau} & \frac{1}{\tau}\\
		\frac{1}{\tau} & \mathcal{K}_2+\frac{1}{\tau}
		\end{bmatrix}}_{=:A}\begin{bmatrix} \delta s \\ \delta s'
	\end{bmatrix}+o((\delta s)^2+(\delta s')^2).
	$$
	Let us distinguish between two cases, namely, when the pair $(s_1,s_2)$ corresponds to a local maximum or a critical point of saddle type of the length functional. 
	\begin{enumerate}[label=(\alph*)]
		\item In the first case, when $\ell$ is locally maximal at $(s_1,s_2)$, the Hessian matrix $A$ of $\ell$ is negative semi-definite, i.e., $\mathrm{tr}A=\mathcal{K}_1+\mathcal{K}_1+\frac{2}{\tau}\leq 0$, and $\det A=\frac{1}{\tau^2}(k_{1,2}-1)\geq 0$ with $k_{1,2}:=(1+\tau\mathcal{K}_1)(1+\tau\mathcal{K}_2)$. Since $\mathcal{K}_1,\mathcal{K}_2<0$, from the inequality for $\det A$, we deduce that $
		k_{1,2}\geq 1
		$.
		Therefore, by Lemma \ref{lemme vp reelles}\ref{item aa en}-\ref{item bb en}, the real eigenvalues $\mu_1\leq \mu_2$ of $Df_\lambda^2(p)$ satisfy $0<\mu_1\leq\lambda^2<1\leq\mu_2$. In particular, if the local maximum is non-degenerate, then $\det A>0$, hence $k_{1,2}>1$, and by Lemma \ref{lemme vp reelles}\ref{item aa en}, the $2$-periodic point is a saddle, with $0<\mu_1<\lambda^2<1<\mu_2$. Note that local maxima of $\ell$ are always non-degenerate if $\Omega \in \mathscr{U}$; indeed, as in the proof of Corollary \ref{nondeg billiard in dk}, we see that in that case, $k_{1,2}\neq 1$.  
		\item In the second case, the matrix $A$ satisfies $\det A=\frac{1}{\tau^2}(k_{1,2}-1)\leq 0$, hence
		$
		k_{1,2}\leq 1
		$. 
		When the critical point is non-degenerate (in particular, when $\Omega \in \mathscr{U}$), it holds $\det A<0$, hence $k_{1,2}<1$. Then, the conclusion of point \ref{sssecond ppooiinntt} in the above statement follows respectively from Lemma \ref{lemme vp reelles}\ref{item cc en}-\ref{item ii en}, when $k_{1,2}\in [0,1)$, from Lemma \ref{lemme vp reelles}\ref{item dd en}, when $k_{1,2}\in (-1,0)$, and from Lemma \ref{lemme vp reelles}\ref{il reste un cas}, when $k_{1,2}\leq-1$. \qedhere
	\end{enumerate}
\end{proof} 
 
\section{Birkhoff attractor for circular and elliptic billiards} \label{quattro}

This section is devoted to the study of the Birkhoff attractor for the dissipative billiard map of (circles and) ellipses. To fix ideas, we can imagine that in what follows, the billiard maps have constant dissipation $\lambda$, but in fact, all the results presented in this section hold for a general dissipative billiard map as in Definition \ref{definit diss bill}. A useful tool through the whole section is the notion of Lyapunov function.
\begin{definition}\label{def lyap function}
	Let $(X,d)$ be a metric space and let $f\colon X\to X$ be a continuous map. 
	A continuous function $L\colon X\to\R$ is a Lyapunov function for $f$ if $L\circ f(x)\leq L(x)$ for every $x\in X$. 
	If $L$ is a Lyapunov function for $f$, the corresponding neutral set is defined as 
	$\mathcal{N}(L):=\{x\in X :\ L\circ f(x)=L(x)\}$.
\end{definition}
As in the Birkhoff case, the simplest example of dissipative billiard is when the boundary of the billiard table is a circle

$$
\mathcal{C}:=\left\{x=(x_1,x_2)\in \R^2: x_1^2 + x_2^2 = R^2 \right\}. 
$$
\noindent The proof of the next result is straightforward.

\begin{proposition} Let $f_\lambda \colon \mathbb{A} \to \mathbb{A}$ be a dissipative billiard map within a circle $\mathcal{C}$. The corresponding Birkhoff attractor $\Lambda_\lambda$ is equal to the attractor $\Lambda_\lambda^0$, and
	$$
	\Lambda_\lambda = \Lambda_\lambda^0 = \mathbb{T} \times \{0\}.
	$$
\end{proposition}

\begin{proof} We notice that, since for any $M \in \left( 0,1 \right]$,
	\begin{equation*}
		f_{\lambda}(\mathbb{T} \times [-M,M]) \subset \mathbb{T} \times \big[-\max_\A\lambda \, M,\max_\A\lambda\, M\big] \subset \mathbb{T} \times (-M,M),
	\end{equation*}
	the attractor (see \ref{CON}) corresponds to $\Lambda_\lambda^0= \mathbb{T} \times \{0\}$. Since $\mathbb{T} \times \{0\}$ is the minimal element, with respect to the inclusion, in $\mathcal{X}(f_{\lambda})$, this concludes the proof.
\end{proof}
\begin{remark}The following are easy observations about dissipative maps inside a circular billiard. For this remark, we assume that the dissipation $\lambda$ is constant.  
	\begin{enumerate}
		\item Since $\mathcal{D}$ is axially symmetric with respect to every line passing through its center, the fact that $\mathbb{T} \times \{0\} \subset \Lambda_\lambda$ is a direct application of Corollary \ref{COR a}.
		\item It is worth noting that, in the case of the map $f_1$ on the disc, the angle $\varphi=\arcsin r$ stays constant along every orbit and it represents an integral of motion; as a consequence, in the dissipative case, $L(s,r) = r$ is a Lyapunov function for $f_{\lambda}$ and $\Lambda_\lambda$ corresponds to the neutral set $\mathcal{N}(L)$ of $L$ (see Definition \ref{def lyap function}). 
		\item The foliation $\{\T\times\{r\} :\ r\in[-1,1]\}$ is $f_\lambda$-invariant, i.e., for every $r\in[-1,1]$ there exists $r'\in [-1,1]$ such that $f_\lambda(\T\times\{r\})=\T\times\{r'\}$. In particular, $r$ and $r'$ have the same sign and $\vert r'\vert\leq\vert r\vert$.
	\end{enumerate} 
\end{remark}
\indent In the following, we investigate the dynamics of the dissipative billiard map within an ellipse $\mathcal{E}$ of non-zero eccentricity $e$. As the dynamics is unchanged under rigid motion of the table (affine maps of $\R^2$), without loss of generality, we assume that the major axis is horizontal, and the minor axis is vertical, i.e., for some parameters $a_1>a_2>0$, we have
$$
\mathcal{E}:=\left\{x=(x_1,x_2)\in \R^2: \frac{x_1^2}{a_1^2}+\frac{x_2^2}{a_2^2}=1\right\},\quad e:=\frac{\sqrt{a_1^2-a_2^2}}{a_1}\in (0,1)\,.
$$ 
Denoting by $\cdot$ the Euclidean scalar product, and by $B$ the diagonal matrix 
$$B:=\begin{bmatrix}\frac{1}{a_1^2} & 0 \\
	0 & \frac{1}{a_2^2}\end{bmatrix},$$ 
the equation of $\mathcal{E}$ can be abbreviated as $\mathcal{E}=\{Bx\cdot x=1\}$. In particular, for any $x \in \mathcal{E}$, the vector $Bx$ is collinear with the normal to $\mathcal{E} $ at $x$; in fact, $Bx$ points outside the convex domain bounded by $\mathcal{E}$. For $\lambda \in (0,1)$, let $f_\lambda \colon \mathbb{A}\to \mathbb{A}$ be the associated dissipative billiard map where, in such a case, it is convenient to describe the phase-space with the $(x,v)$ coordinates:
$$\big\{(x,v)\in \mathcal{E} \times T^1 \mathcal{E}: Bx\cdot v\leq 0\big\}\,.$$
\noindent With an abuse of notation, we will refer to this set of coordinates $\{(x,v)\in \mathcal{E} \times T^1 \mathcal{E}: Bx\cdot v\leq 0\}$ also as $\A$. In order to lighten notation, $f_\lambda$ will also denote the dissipative billiard map in $(x,v)$-coordinates.

\noindent Since a point $(s,r)\in \mathbb{A}$ is $2$-periodic for $f_\lambda$ if and only if it is $2$-periodic for  the standard billiard map $f_1$, the set of $2$-periodic points is reduced to 
$$
\mathrm{II}:=\{E_1,E_2,H_1,H_2\}\,, 
$$
where we denote by 
$$
\{E_1,E_2=f_\lambda(E_1)\} \quad \text{ and } \quad \{H_1,H_2=f_\lambda(H_1)\}
$$ 
the $2$-periodic orbits of $f_\lambda$ corresponding to the minor and the major axis respectively. \\ 

\indent The next result is a direct outcome of Lemma \ref{lemma noncst disp}. 

\begin{lemma}\label{lemme vp reelles ellipse} Let $f_\lambda \colon \mathbb{A} \to \mathbb{A}$ be a dissipative billiard map within an ellipse
	$\mathcal{E}$ of non-zero eccentricity. 
	The $2$-periodic orbit $\{E_1,E_2=f_{\lambda}(E_1)\}$, corresponding to the minor axis, is a sink. The $2$-periodic orbit $\{H_1, H_2 =f_{\lambda}(H_1)\}$, corresponding to the major axis, is a saddle. 
\end{lemma} 
\begin{proof} 
	The statement immediately follows from Lemma \ref{lemma noncst disp}. Indeed, as in Section \ref{conti 2 orbite}, for a 2-periodic orbit $\{p,f_\lambda(p)\}=\{p,f_1(p)\}$, denote by $\tau$ and $\mathcal{K}$  respectively the distance between the two bounces and the common curvature at these points, and as in \eqref{k un deuxx}, let 
	$$
	k_{1,2}(p):=(\tau \mathcal{K} + 1)^2\geq 0\,.
	$$ 
	On the one hand, when $p\in\{E_1,E_2\}$, $\tau \mathcal{K} = 2 a_2 (-\frac{a_2}{a_1^2}) = -2 (\frac{a_2}{a_1})^2 \in (-2,0)$. Thus, $k_{1,2}(p)\in (0,1)$, and then, the $2$-periodic orbit $\{E_1,E_2=f_{\lambda}(E_1)\}$ is a sink. On the other hand, when $p\in\{H_1,H_2\}$, $\tau \mathcal{K} = 2 a_1( -\frac{a_1}{a_2^2}) = -2 (\frac{a_1}{a_2})^2 <-2$. Thus, $k_{1,2}(p)>1$, and then, the $2$-periodic orbit $\{H_1,H_2=f_{\lambda}(H_1)\}$ is a saddle. 
\end{proof}

\begin{notation}For $i = 1,2$, we denote by $\mathcal{W}^s(H_i;f_\lambda^2)$ (resp. $\mathcal{W}^u(H_i;f_\lambda^2)$) the ($1$-dimensional) stable (resp. unstable) manifold of $H_i$ for $f_\lambda^2$. In order to lighten the notation, for $*=s,u$, $i=1,2$, we also denote by $\mathcal{W}^*(\mathcal{O}_\lambda(H_i))$ the union $\mathcal{W}^*(H_1;f_\lambda^2)\cup\mathcal{W}^*(H_2;f_\lambda^2)$. Similarly, let $\mathcal{W}^s(E_i;f_\lambda^2)$ be the ($2$-dimensional) stable manifold of $E_i$ for $f_\lambda^2$. Again, to lighten the notation, for $i=1,2$, we denote by $\mathcal{W}^s(\mathcal{O}_\lambda(E_i)) $ the union $\mathcal{W}^s(E_1;f_\lambda^2)\cup\mathcal{W}^s(E_2;f_\lambda^2)$.\end{notation}

\indent The main result of the present section is the following characterization of the Birkhoff attractor for dissipative billiard maps within an ellipse. 
\begin{theorem} \label{ellisse} Let $f_\lambda \colon \mathbb{A} \to \mathbb{A}$ be a dissipative billiard map within an ellipse
	$\mathcal{E}$ of non-zero eccentricity. The corresponding Birkhoff attractor $\Lambda_\lambda$ is equal to the attractor $\Lambda_\lambda^0$, and we have
	\begin{equation}\label{egalite lambda lambda zero w u}
		\Lambda_\lambda^0=\Lambda_\lambda=\mathcal{W}^u(\mathcal{O}_\lambda(H_1))\cup \{E_1,E_2\} = \overline{\mathcal{W}^u(\mathcal{O}_\lambda(H_1))} \,.
	\end{equation}
	Moreover, for $i=1,2$, $\mathcal{W}^u(H_i;f_\lambda^2)\setminus \{H_i\}$ is the disjoint union of two branches $\mathscr{C}_i^1,\mathscr{C}_i^2$, with $\mathscr{C}_i^j\subset \mathcal{W}^s(E_j;f_\lambda^2)$, $j=1,2$. 
\end{theorem}

The next two propositions will be used in the proof of Theorem \ref{ellisse}. 

\begin{proposition}\label{ELLE} Let $f_\lambda \colon \mathbb{A} \to \mathbb{A}$ be a dissipative billiard map within an ellipse
	$\mathcal{E}$ of non-zero eccentricity. The function 
	$$
	L\colon \mathbb{A} \to \R, \quad (x,v) \mapsto Bx \cdot v
	$$ 
	is a Lyapunov function for $f_\lambda$. Moreover, its neutral set $\mathcal{N}(L)$ is equal to 
	$
	f_\lambda^{-1}(\A_\perp)$, where $\A_\perp
	$ is the set of points $\{(x,v) \in \A :\ \text{$v$ is collinear to $x$}\}$. 
	More precisely, there exists a continuous function $\delta \colon \mathbb{R}_+\to \mathbb{R}_+$ with $\lim_{\varepsilon\to 0} \delta(\varepsilon)=0$ such that for any $(x,v)\in \mathbb{A}$, 
	\begin{equation}\label{proximite ensemble neutre}
		|L(f_\lambda(x,v))-L(x,v)|<\varepsilon\implies d((x,v),f_\lambda^{-1}(\A_\perp))<\delta(\varepsilon)\,,
	\end{equation}
	where $d$ is the usual distance on $\A$.
\end{proposition}

\begin{proof}
	For $(x,v)\in \mathbb{A}$, let $(x',v'):=f_\lambda (x,v)$. We first observe that
	$$
	B(x'-x)\cdot (x+x') = Bx' \cdot x + Bx' \cdot x' - Bx \cdot x -Bx \cdot x' =  0$$
	\noindent since $x,x' \in \mathcal{E}$ and the matrix $B$ is symmetric. 
	As $x'-x$ is collinear to the vector $v$, the previous relation yields $Bv \cdot (x+x')=0$; by the symmetry of $B$, we thus obtain
	\begin{equation}\label{xprime v}
		-Bx' \cdot v=Bx \cdot v\,.
	\end{equation}
	Moreover, due to the reflection law, $Bx' \cdot (v+v') < 0$, except when $v'$ is collinear with the normal at $x'$, in which case $Bx' \cdot (v+v')=0$. By \eqref{xprime v}, we conclude that 
	$$
	L(x',v')=Bx' \cdot v' \le Bx\cdot v=L(x,v)\,,
	$$
	with equality exactly when the bounce at $x'$ is perpendicular to $\mathcal{E}$. This means that $|L(x',v')-L(x,v)|\ll 1$ if and only if $d((x,v),f_\lambda^{-1}(\A_\perp))\ll 1$. 
\end{proof}

\begin{remark} It is worth noting that $f^{-1}_{\lambda}(\A_{\perp})$ can be alternatively detected as a neutral set in the following way. For $\zeta \in [0,a_2)\cup(a_2,a_1)$, let consider the family of quadrics: 
	$$
	\mathcal{E}_\zeta:=\left\{x=(x_1,x_2)\in \R^2: \frac{x_1^2}{a_1^2-\zeta}+\frac{x_2^2}{a_2^2-\zeta}=1\right\}\,.
	$$
	For $\zeta \in [0,a_2)$, $\mathcal{E}_\zeta$ is an ellipse confocal to $\mathcal{E}$ and, for $\zeta \in (a_2,a_1)$, $\mathcal{E}_\zeta$ is a hyperbola confocal to $\mathcal{E}$. Let $F_1=(-c,0)$ and $F_2=(c,0)$ the two foci of $\mathcal{E}$, where $c:=\sqrt{a_1^2-a_2^2}$. We extend the previous definition by letting $\mathcal{E}_{a_2}:=((-\infty,0),F_1]\cup [F_2,(0,+\infty))$ and $\mathcal{E}_{a_1}:=\{0\}\times \R$. 
	By the theory of usual elliptic billiards, for $(x,v)\in \mathbb{A}\setminus \mathrm{II}$, there exists a unique $\zeta=\zeta(x,v)>0$ such that any orbit segment of the $f_1$-trajectory starting at $(x,v)$ is tangent to $\mathcal{E}_{\zeta}$; moreover, $\mathcal{E}_{\zeta}$ is an ellipse when the segment  $[x,x']$ does not intersect $[F_1,F_2]$, and it is a (possibly degenerate) 
	hyperbola when $[x,x']$ intersects $[F_1,F_2]$. Finally, we set $\zeta(H_1)=\zeta(H_2):=a_2$ and $\zeta(E_1)=\zeta(E_2):=a_1$. Then, comparing the standard reflection law to the dissipative one, it can be proved that the function $-\zeta$ is a Lyapunov function for $f_{\lambda}$, with neutral set $\mathcal{N}(-\zeta)=f_\lambda^{-1}(\A_\perp)$.
	
	\begin{figure}[h]
		\centering
		\includegraphics[scale=0.6, trim=2cm 2cm 2cm 1.2cm]{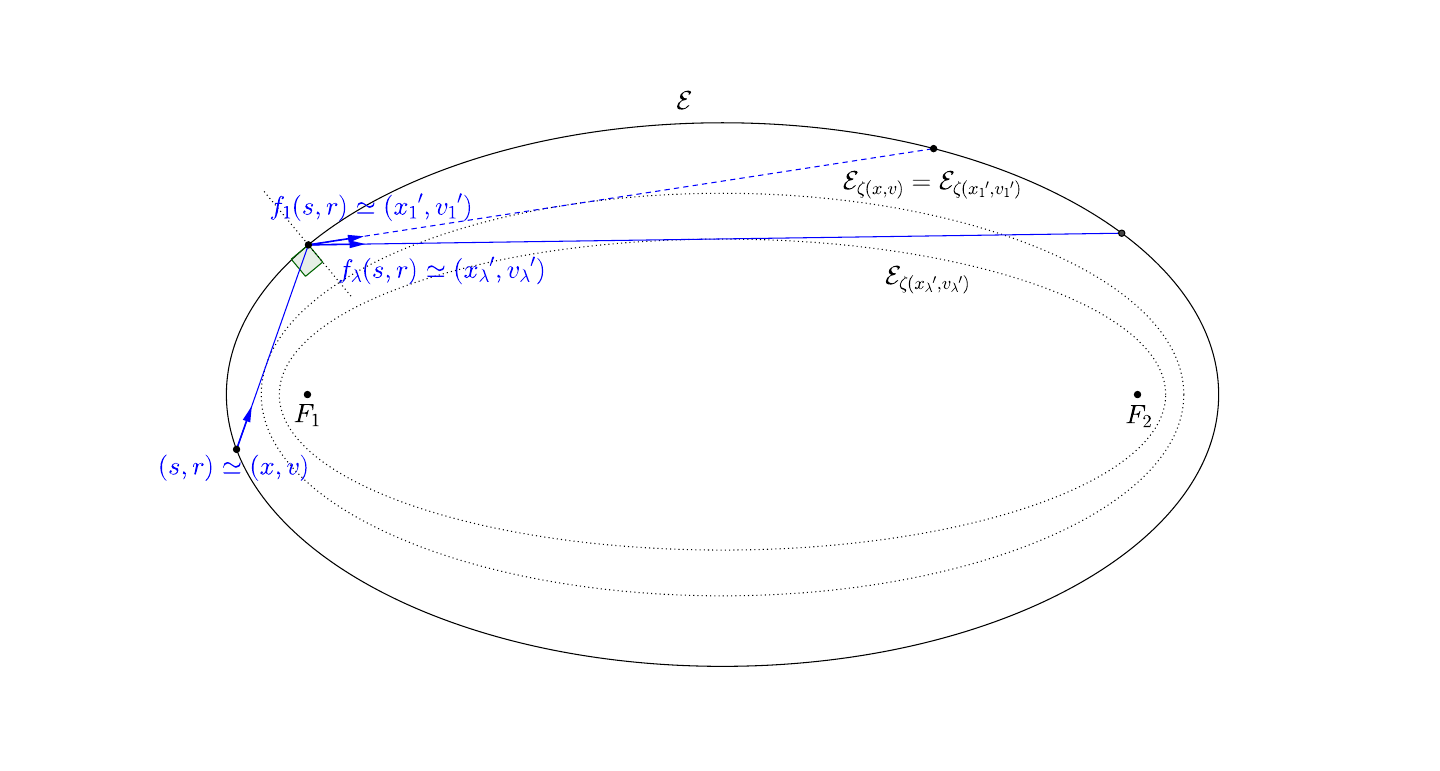}
		\caption{The Lyapunov function $-\zeta$. 
		}
		\label{figcone}
	\end{figure} 
\end{remark}

\begin{remark}\label{remark min max}\quad
	\begin{enumerate}[label=(\alph*)]
		\item\label{remark ensemble neutre mathcal l}  Let us recall that $\mathrm{II}:=   \{E_1,E_2,H_1,H_2\}$ is the set of $2$-periodic points. By Proposition \ref{ELLE}, the function $\mathcal{L}_\lambda:=L + L \circ f_{\lambda}$ is also a Lyapunov function for $f_{\lambda}$, with neutral set 
		$$
		\mathcal{N}(\mathcal{L}_\lambda) = f_\lambda^{-1}(\A_\perp) \cap f_\lambda^{-2}(\A_\perp) =f_\lambda^{-1}\big(f_\lambda^{-1}(\A_\perp) \cap \A_\perp\big)=f_\lambda^{-1}(\mathrm{II})=\mathrm{II}\,.
		$$ 
		Indeed, an orbit with two consecutive perpendicular bounces is necessarily $2$-periodic. Moreover, by \eqref{proximite ensemble neutre}, there exists a continuous function $\hat\delta \colon \mathbb{R}_+\to \mathbb{R}_+$ with $\lim_{\varepsilon\to 0} \hat\delta(\varepsilon)=0$ such that for any $(x,v)\in \mathbb{A}$, 
		\begin{equation}\label{proximite ensemble neutre bis}
			0\leq \mathcal{L}_\lambda(x,v)-\mathcal{L}_\lambda(f_\lambda(x,v))<\varepsilon\implies d((x,v),\mathrm{II})<\hat \delta(\varepsilon)\,,
		\end{equation}
		where $d$ is the usual distance on $\A$. 
		\item\label{remark bis ensemble neutre mathcal l} For any $(x,v)\in \A_\perp$, $Bx$ and $v$ are collinear, with opposite orientations, hence 
		$$
		L(x,v) =-\Vert Bx\Vert= -\sqrt{\frac{x_1^2}{a_1^4}+\frac{x_2^2}{a_2^4}}=-\sqrt{\frac{x_2^2}{a_2^2}\big(\frac{1}{a_2^2}-\frac{1}{a_1^2}\big)+\frac{1}{a_1^2}}\,.
		$$
		Therefore, $L|_{\A_\perp}$ is maximal when $x_2=0$ (and takes the value $-\frac{1}{a_1}$), i.e., for $(x,v)\in \{H_1,H_2\}$, and $L|_{\A_\perp}$ is minimal when $x_2^2=a_2^2$ (and takes the value $-\frac{1}{a_2}$), i.e.,  for $(x,v)\in \{E_1,E_2\}$. 
		
		\item\label{Lyap funct in the past} Let $X\subset \mathrm{int}(\A)$ be a $f_\lambda$-invariant set. Then, both $-L\vert_{X}$ and $-\mathcal{L}_\lambda\vert_{X}$ are Lyapunov functions for $f_\lambda^{-1}$.
	\end{enumerate}
\end{remark}

\begin{proposition}\label{coro attra} Let $f_\lambda \colon \mathbb{A} \to \mathbb{A}$ be a dissipative billiard map within an ellipse
	$\mathcal{E}$ of non-zero eccentricity. All the orbits are attracted by a $2$-periodic orbit,  i.e., for any $(s,r)\in \mathbb{A}$, there exists $p_{+}=p_+(s,r)\in \mathrm{II} = \{E_1,E_2,H_1,H_2\}$ such that 
	$$
	\lim_{n \to +\infty}f_\lambda^{2n}(s,r)=p_{+},\quad \lim_{n \to +\infty}f_\lambda^{2n+1}(s,r)=f_{\lambda}(p_{+})\,.
	$$
	In particular, the set of periodic points for $f_\lambda$ is reduced to the set $\mathrm{II}$ of $2$-periodic points. 
	Moreover, $\mathrm{II}\subset\Lambda_\lambda\subset \Lambda_\lambda^0$, and for any $(s,r)\in \Lambda_\lambda^0\setminus \mathrm{II}$, there exist $i_-,i_+\in \{1,2\}$ such that 
	\begin{align*}
		\lim_{n \to -\infty}f_\lambda^{2n}(x,v)=H_{i_-},\quad \lim_{n \to -\infty}f_\lambda^{2n-1}(x,v)=f_{\lambda}(H_{i_-})\,,\\
		\lim_{n \to +\infty}f_\lambda^{2n}(x,v)=E_{i_+},\quad \lim_{n \to +\infty}f_\lambda^{2n-1}(x,v)=f_{\lambda}(E_{i_+})\,.
	\end{align*} 
\end{proposition}

\begin{proof}
	Fix $(s,r)\in \mathbb{A}$. By Remark \ref{remark min max}\ref{remark ensemble neutre mathcal l}, the function $\mathcal{L}_\lambda$ is a Lyapunov function for $f_\lambda$ whose neutral set is $\mathrm{II}$; consequently the  omega-limit set $\omega_{f_\lambda}(s,r)$ satisfies $\omega_{f_\lambda}(s,r)\subset\mathcal{N}(\mathcal{L}_\lambda)=\mathrm{II}$. For each $n\geq 0$, we set  $u_{n}:=\mathcal{L}_\lambda(f_\lambda^{n}(s,r)).
	$ 
	The sequence $(u_{n})_{\geq 0}$ is decreasing and bounded from below by $\min_{\mathbb{A}} \mathcal{L}_\lambda>-\infty$, hence is convergent. In particular, $\lim_{n \to +\infty} (u_n-u_{n+1})= 0$; by \eqref{proximite ensemble neutre bis}, we deduce that 
	$$\lim_{n \to +\infty} d(f_{\lambda}^{n}(s,r),\mathrm{II})=0\,,$$ 
	where $d$ is the usual distance inherited from $\A$. Recall that $\mathrm{II}$ is formed of only four different points and let $\epsilon:=\frac{1}{3}\min_{p\neq q \in \mathrm{II}}d(p,q)>0$. 
	From previous limit, there exists $n_0 \in \mathbb{N}$ such that for any $n \geq n_0$, we have  $d(f_{\lambda}^{n}(s,r),\mathrm{II})<\epsilon$. Actually, for $n \geq n_0$, there exists a unique point $p(s,r,n)\in \mathrm{II}$ such that  $d(f_{\lambda}^{n}(s,r),p(s,r,n))<\epsilon$. In particular, since the bounce at $f_{\lambda}^{n+1}(s,r)$ gets closer and closer to being perpendicular as $n\to +\infty$, we have that 
	$$\lim_{n\to+\infty}d(f_{\lambda}^{n+2}(s,r),f_{\lambda}^{n}(s,r))=0\,.$$ 
	Let us then fix $n_1\geq n_0$ such that for any $n\geq n_1$, it holds 
	$$
	d(f_{\lambda}^{n+2}(s,r),f_{\lambda}^{n}(s,r))<\epsilon\, .
	$$
	Then, for any $n\geq n_1$, we have
	\begin{align*}
		&d(p(s,r,n),p(s,r,n+2))\\
		&\leq d(p(s,r,2n),f_{\lambda}^{n}(s,r))+d(f_{\lambda}^{n}(s,r),f_{\lambda}^{n+2}(s,r))+d(f_{\lambda}^{n+2}(s,r),p(s,r,n+2))<3 \epsilon\,.
	\end{align*}
	By the choice of $\epsilon > 0$, it follows that $p(s,r,n+2)=p(s,r,n)$, for any $n \geq n_1$. Let us then set $p_+=p_+(s,r):=p(s,r,2n)$, for any $2n\geq n_1$ (it is well-defined by the previous discussion).
	Then we conclude that  
	$$
	\lim_{n \to +\infty}f_\lambda^{2n}(s,r)=p_{+},\quad \lim_{n \to +\infty}f_\lambda^{2n+1}(s,r)=f_{\lambda}(p_{+})\,.
	$$
	In particular, we deduce also that $p_+(f_\lambda(s,r))=f_\lambda(p_+(s,r))$ and that $\omega_{f_\lambda}(s,r)=\{p_+,f_\lambda(p_+)\}$. \\
	\noindent The sets $\Lambda_\lambda,\Lambda_\lambda^0$ are $f_\lambda$-invariant; moreover, $f_{\lambda}|_{\Lambda_\lambda}$, resp. $f_{\lambda}|_{\Lambda_\lambda^0}$ is invertible, and $\tilde{\mathcal{L}}_\lambda:=-\mathcal{L}_\lambda|_{\Lambda_\lambda}$, resp. $\tilde{\mathcal{L}}_\lambda^0:=-\mathcal{L}_\lambda|_{\Lambda_\lambda^0}$ is a Lyapunov function for  $(f_{\lambda}|_{\Lambda_\lambda})^{-1}$, resp. $(f_{\lambda}|_{\Lambda_\lambda^0})^{-1}$, see Remark \ref{remark min max}\ref{Lyap funct in the past}. Since $\Lambda_\lambda,\Lambda_\lambda^0$ are compact, for any $(s,r)\in \Lambda_\lambda$, resp. $(s,r)\in \Lambda_\lambda^0$, we have that the alpha-limit set $\alpha_{f_\lambda}(s,r)$ satisfies $\emptyset\neq\alpha_{f_\lambda}(s,r)\subset \mathcal{N}(\tilde{\mathcal{L}}_\lambda)\cap \Lambda_\lambda\subset\mathrm{II}$, resp. $\emptyset\neq\alpha_{f_\lambda}(s,r)\subset \mathcal{N}(\tilde{\mathcal{L}}_\lambda^0)\cap \Lambda_\lambda^0\subset\mathrm{II}$; in particular, we deduce that $\emptyset\neq \Lambda_\lambda \cap \mathrm{II} \subset \Lambda_\lambda^0 \cap \mathrm{II}$. Fix $(s,r)\in \Lambda_\lambda$, resp. $(s,r)\in \Lambda_\lambda^0$. Arguing as above, we see that there exists $p_-=p_-(s,r)\in \mathrm{II}\cap \Lambda_\lambda$, resp. $p_-=p_-(s,r)\in \mathrm{II}\cap \Lambda_\lambda^0$, such that 
	$$
	\lim_{n \to -\infty}f_\lambda^{2n}(s,r)=p_{-}, \quad \lim_{n \to -\infty}f_\lambda^{2n-1}(s,r)=f_{\lambda}(p_{-})\,.$$ 
	Then, it holds that $\alpha_{f_\lambda}(s,r)=\{p_-,f_\lambda(p_-)\}$. There are two cases for the point $(s,r)\in \Lambda_\lambda$, resp. $\Lambda^0_\lambda$:
	\begin{itemize}
		\item either $\alpha_{f_\lambda}(s,r) \cap \omega_{f_\lambda}(s,r)\neq \emptyset$, and then, the whole orbit $(f_\lambda^k(s,r))_{k \in \mathbb{Z}}$ is in the neutral set  $\mathcal{N}(\mathcal{L}_\lambda)$, i.e., $(s,r)=p_-(s,r)=p_+(s,r)$ is $2$-periodic; 
		\item otherwise, $(s,r)\notin \mathrm{II}$,\footnote{This case clearly occurs, as $\Lambda_\lambda$ separates $\mathbb{A}$, while $\mathrm{II}$ is a finite set.} and the $2$-periodic orbits $\alpha_{f_\lambda}(s,r)=\{p_-,f_\lambda(p_-)\}$ and $\omega_{f_\lambda}(s,r)=\{p_+,f_\lambda(p_+)\}$ are distinct. By Remark \ref{remark min max}\ref{remark bis ensemble neutre mathcal l} and since the orbits of $p_+$ and $p_-$ are different, we have that $\mathcal{L}_\lambda(p_-)>\mathcal{L}_\lambda(p_+)$ and actually,  $\alpha_{f_\lambda}(s,r)=\{H_1,H_2\}$ and $\omega_{f_\lambda}(s,r)=\{E_1,E_2\}$. As $\Lambda_\lambda,\Lambda_\lambda^0$ are closed, we also deduce that $\mathrm{II}\subset \Lambda_\lambda\subset \Lambda_\lambda^0$. \qedhere
	\end{itemize}
\end{proof}

\begin{remark}
	When the dissipation $\lambda$ is constant, the fact that $\mathrm{II}\subset\Lambda_\lambda\subset \Lambda_\lambda^0$ proven in Proposition \ref{coro attra} also follows from Lemma \ref{lemme sym axiale}, due to the symmetries of the ellipse $\mathcal{E}$.  
\end{remark} 


\noindent Let us also note that, by Proposition \ref{coro attra}, for any $(s,r) \in \mathbb{A}\setminus \mathcal{W}^s(H_1;f_\lambda)$, the forward orbit of $(s,r)$ converges to the $2$-periodic orbit $\{E_1,f_{\lambda}(E_1) = E_2\}$. \\ \\
\indent We are now ready to give the proof of Theorem \ref{ellisse}. 
\begin{notation}Given some small $\delta>0$, we denote by $\mathcal{W}_{\delta}^{u}(H_i;f_\lambda^2)$ the $\delta$-local unstable manifold of $H_i$ with respect to $f_\lambda^2$, i.e., the set 
	$$\mathcal{W}_{\delta}^{u}(H_i;f_\lambda^2) := \{(s,r)\in\A :\ d(f_\lambda^{2n}(s,r),H_i)\leq \delta,\ \forall\, n\leq 0\}\,.$$ 
	Similarly, the $\delta$-local stable manifold of $H_i$ with respect to $f_\lambda^2$ is
	$$\mathcal{W}_{\delta}^{s}(H_i;f_\lambda^2):=\{(s,r)\in\A:\ d(f_\lambda^{2n}(s;r),H_i)\leq \delta,\ \forall\, n\geq 0\}\,.$$
	 For $*=s,u$, $i=1,2$, the notation $\mathcal{W}_\delta^{*}(\mathcal{O}_\lambda(H_i))$ refers to $\mathcal{W}_{\delta}^{*}(H_1;f_\lambda^2)\cup\mathcal{W}_{\delta}^{*}(H_2;f_\lambda^2)$. 
	We denote by 
	$$\mathcal{W}_{\delta}^{s}(E_i;f_\lambda^2):=\{(s,r)\in\A:\ d(f_\lambda^{2n}(s,r),E_i)\leq\delta,\ \forall\, n\geq 0\}$$ 
	the $\delta$-local stable manifold of $E_i$ with respect to $f_\lambda^2$. Similarly, for $i=1,2$, the notation $\mathcal{W}_{\delta}^{s}(\mathcal{O}_\lambda(E_i))$ refers to $\mathcal{W}_{\delta}^{s}(E_1;f_\lambda^2)\cup\mathcal{W}_{\delta}^{s}(E_2;f_\lambda^2)$.
\end{notation}

\begin{proof}[Proof of Theorem \ref{ellisse}]
	
	 By Proposition \ref{coro attra}, for any $(s,r) \in \Lambda_\lambda\setminus\{E_1,E_2\}$, resp. $(s,r) \in \Lambda_\lambda^0\setminus\{E_1,E_2\}$, there exists $i_-\in \{1,2\}$ such that  
	$\lim_{n \to -\infty}f_\lambda^{2n}(s,r)=H_{i_-}$, hence $(s,r)\in\mathcal{W}^{u}(\mathcal{O}_\lambda(H_{i_-}))$. We deduce that
	\begin{equation}\label{premiere inclusion}
	\Lambda_\lambda\subset \Lambda_\lambda^0\subset
	\mathcal{W}^{u}(\mathcal{O}_\lambda(H_1))\cup \{E_1,E_2\}=\overline{	\mathcal{W}^{u}(\mathcal{O}_\lambda(H_1))}\, ,
	\end{equation}
	where the last equality follows again from Proposition \ref{coro attra}; indeed, the points $E_1,E_2$ are accumulated by the forward orbit of any point $(s,r) \in \Lambda_\lambda\setminus\{E_1,E_2\}\subset 	\mathcal{W}^{u}(\mathcal{O}_\lambda(H_1))$. From \eqref{premiere inclusion}, and applying Lemma \ref{X f invariant} to $\overline{ 	\mathcal{W}^{u}(\mathcal{O}_\lambda(H_1))}$, we deduce that $\overline{ 	\mathcal{W}^{u}(\mathcal{O}_\lambda(H_1))}$ separates $\A$. Since it is also compact, connected and $f_\lambda$-invariant, it holds that $\overline{ 	\mathcal{W}^{u}(\mathcal{O}_\lambda(H_1)) } \in \mathcal{X}(f_\lambda)$. 
	\begin{claim}\label{claim local unstable}
		There is $\delta>0$ such that, for any $x\in \mathcal{W}_\delta^{u}(\mathcal{O}_\lambda(H_1))$, $\overline{ \mathcal{W}^{u}(\mathcal{O}_\lambda(H_1)) } \setminus \{x\}$ does not separate the annulus. 
	\end{claim}
	
	\begin{proof}[Proof of the claim.]
		Let $\eta>0$ be small enough such that the balls of radius $\eta$ centered at points in $\{H_1,H_2,E_1,E_2\}$ are pairwise disjoint. From Proposition \ref{coro attra}, for $i=1,2$, we have that $\mathcal{W}^u(H_i;f_\lambda^2)\setminus\{H_i\}$ is contained in the stable manifold of $\{E_1,E_2\}$; in particular, there exists $N\in\N$ such that for every $n>N$, it holds
			\begin{equation}\label{cc unstable}
			f_\lambda^n(\mathcal{W}^u_\eta(\mathcal{O}_\lambda(H_1))\setminus f_\lambda^N(\mathcal{W}^u_\eta(\mathcal{O}_\lambda(H_1))\subset B(E_1,\eta)\cup B(E_2,\eta)\, .
			\end{equation}
			We can then choose $\delta>0$ small enough such that, for $i=1,2$, the $\delta$-local unstable manifold of $H_i$ is a $C^1$ graph over the first coordinate projection of $B(H_i,\delta)$ and such that $B(H_i,\delta)\cap \mathcal{W}^u(H_i;f_\lambda^2)= \mathcal{W}^u_\delta(H_i;f_\lambda^2)$, i.e., the unstable manifold meets the ball only at the local unstable manifold. This last property is possible thanks to \eqref{cc unstable}. The $\delta$-local unstable manifold $\mathcal{W}^u_\delta(H_i;f_\lambda^2)$ separates the ball $B(H_i,\delta)$, i.e., we have $B(H_i,\delta)\setminus \mathcal{W}^u_\delta(H_i;f_\lambda^2)= \mathcal{U}\cup\mathcal{V}$, for two disjoint connected open sets $\mathcal{U}$ and $\mathcal{V}$. For any $x\in\mathcal{W}^u_\delta(H_i;f_\lambda^2)$, the set $B(H_i,\delta)\setminus(\mathcal{W}^u_\delta(H_i;f_\lambda^2)\setminus\{x\})$ is path-connected. 
			We conclude that $\overline{\mathcal{W}^u(\mathcal{O}_{f_\lambda}(H_1))}\setminus\{x\}$ does not separate the annulus if $$
			\mathcal{U}\subset U_{\overline{\mathcal{W}^u(\mathcal{O}_{f_\lambda}(H_1))}}\quad \text{and}\quad \mathcal{V}\subset V_{\overline{\mathcal{W}^u(\mathcal{O}_{f_\lambda}(H_1))}}\, ,
			$$
			where $\A\setminus\overline{\mathcal{W}^u(\mathcal{O}_{f_\lambda}(H_1))}=U_{\overline{\mathcal{W}^u(\mathcal{O}_{f_\lambda}(H_1))}}\cup V_{\overline{\mathcal{W}^u(\mathcal{O}_{f_\lambda}(H_1))}}$.
			This follows from the fact that $H_i\in\Lambda_\lambda$ and that $\overline{\mathcal{W}^u(\mathcal{O}_{f_\lambda}(H_1))}\in\mathcal{X}(f_\lambda)$: indeed, by Proposition \ref{minimality}\eqref{pt due}, $H_i\in \mathrm{Fr}(U_{\overline{\mathcal{W}^u(\mathcal{O}_{f_\lambda}(H_1))}}\cap\mathrm{Fr}(V_{\overline{\mathcal{W}^u(\mathcal{O}_{f_\lambda}(H_1))}})$ and in particular the two connected open sets $\mathcal{U}$ and $\mathcal{V}$ cannot be contained in the same connected component of $\A\setminus\overline{\mathcal{W}^u(\mathcal{O}_{f_\lambda}(H_1))}$.
	\end{proof}
	\noindent Thus, by Lemma \ref{lemma disconnected annulus}, we have that
	\begin{equation}\label{deuxie inclusion}
	\mathcal{W}_{\delta}^u(\mathcal{O}_\lambda(H_1))\subset \Lambda_\lambda\, .
	\end{equation}
	Let us also recall that for $i=1,2$, we have  $\mathcal{W}^u(\mathcal{O}_\lambda(H_1))=\bigcup_{j\geq 0}f_\lambda^{j}(\mathcal{W}_{\delta}^u(\mathcal{O}_\lambda(H_1)))$. By \eqref{deuxie inclusion}, and as $\Lambda_\lambda$ is $f_\lambda$-invariant and closed, we obtain
	\begin{equation}\label{troisie inclusion}
	\overline{\mathcal{W}^u(\mathcal{O}_{f_\lambda}(H_1))}\subset \Lambda_\lambda.
	\end{equation}
	Comparing \eqref{premiere inclusion} and \eqref{troisie inclusion}, we deduce that all the inclusions are actually equalities, which concludes the proof of \eqref{egalite lambda lambda zero w u}. \\
\noindent By the previous discussion, $\Lambda_\lambda \setminus \mathrm{II}$ is the disjoint union of four connected components $\mathscr{C}_1,\mathscr{C}_1',\mathscr{C}_2=f_\lambda(\mathscr{C}_1),\mathscr{C}_2'=f_\lambda(\mathscr{C}_1')$, where $\mathscr{C}_i$ and $\mathscr{C}_i'$ correspond to the two branches of $ \mathcal{W}^u(H_i;f_\lambda^2)\setminus \{H_i\}$, for $i=1,2$. Moreover, by Corollary \ref{coro attra}, $\mathscr{C}_1\subset \mathcal{W}^s(E_j;f_\lambda^2)$ and $\mathscr{C}_1'\subset \mathcal{W}^s(E_k;f_\lambda^2)$ for some $j,k\in \{1,2\}$. We claim that $j\neq k$. Assume by contradiction that $j=k$ and set $\widehat{\mathscr{C}}_1:=\mathscr{C}_1\cup\mathscr{C}_1'\cup \{H_1,E_j\}=\mathcal{W}^u(H_1;f_\lambda^2)\cup \{E_j\}$. Since $\mathcal{W}^u(H_1;f_\lambda^2)$ has no self-intersection, we have that $\widehat{\mathscr{C}}_1$ is a $f_\lambda^2$-invariant simple closed curve. We distinguish between two cases:
	\begin{enumerate}
		\item either $\widehat{\mathscr{C}}_1$ separates the annulus $\A$; then, we would have $\Lambda_\lambda=\widehat{\mathscr{C}}_1\sqcup f_\lambda(\widehat{\mathscr{C}}_1)$, where $\widehat{\mathscr{C}}_1$ and $f_\lambda(\widehat{\mathscr{C}}_1)$ are compact, connected, and both separate $\A$. This would imply that $\A\setminus\Lambda_\lambda$ is the disjoint union of $3$ connected open sets, one of whose is a $f_\lambda$-invariant bounded open set. This would contradict the dissipative character of the map;
		\item or the curve $\widehat{\mathscr{C}}_1$ is homotopic to a point; in particular it bounds a $f_\lambda^2$-invariant open set. Again, this would contradict the dissipative character of $f_\lambda^2$. 
	\end{enumerate} 
	Thus, $j\neq k$. setting $\mathscr{C}_1^j:=\mathscr{C}_1$, $\mathscr{C}_1^k:=\mathscr{C}_1'$, $\mathscr{C}_2^k:=\mathscr{C}_2=f_\lambda(\mathscr{C}_1^j)$, $\mathscr{C}_2^j:=\mathscr{C}_2'=f_\lambda(\mathscr{C}_1^k)$, this concludes the proof. 
\end{proof}

\noindent In the next corollary, we prove that the conclusion of Theorem \ref{ellisse} remains true for strictly convex domains whose boundary is sufficiently $C^2$-close to an ellipse. For simplicity, in the rest of this section, we will assume that the dissipation $\lambda$ is constant. 
	\begin{corollary}\label{coro proche ellipse} 
		Let $\mathcal{E}$ be an ellipse and fix $\lambda \in (0,1)$. Then, there exists $\epsilon=\epsilon(\mathcal{E},\lambda)>0$ such that for any domain $\Omega\subset \R^2$ whose boundary $\partial \Omega$ is $C^k$, $k \geq 2$, and satisfies $d_{C^2}(\partial\Omega,\mathcal{E})<\epsilon$, 
		the following holds. Let $f_{\lambda}\colon \mathbb{A} \to \mathbb{A}$ be the dissipative billiard map within $\Omega$. There exist $2$-periodic orbits $\{H_1(\Omega),H_2(\Omega)\}$ and $\{E_1(\Omega),E_2(\Omega)\}$  
		of saddle and sink type respectively, and the Birkhoff attractor 
		is equal to 
		$$
		\Lambda_\lambda= 
		\mathcal{W}^u(\mathcal{O}_{f_\lambda}(H_1(\Omega)) \cup \{E_1(\Omega),E_2(\Omega)\} =
		\overline{\mathcal{W}^u(\mathcal{O}_{f_\lambda}(H_1(\Omega)) }\,,
		$$ 
		where $	\mathcal{W}^u(\mathcal{O}_{f_\lambda}(H_1(\Omega)) :=\mathcal{W}^u(H_1(\Omega);f_\lambda^2) \cup \mathcal{W}^u(H_2(\Omega);f_\lambda^2)$.
		Moreover, the function $(\mathcal{E},\lambda)\mapsto \epsilon(\mathcal{E},\lambda)$ can be chosen to be continuous. 
\end{corollary}

\begin{proof} 
	Denote by $g_\lambda$ the dissipative billiard map of $\mathcal{E}$ with dissipation parameter $\lambda\in (0,1)$. 
		Let $\{H_1,H_2\}$ and $\{E_1,E_2\}$ be the $2$-periodic orbits for $g_\lambda$ of saddle and sink type respectively, see Lemma \ref{lemme vp reelles ellipse}.
		For $\partial\Omega$ sufficiently $C^2$-close to $\mathcal{E}$, the associated domain $\Omega$ is still strongly convex, as the curvature function depends continuously on the domain in the $C^2$-topology, and $\mathcal{E}$ is strongly convex. Without loss of generality, we can assume that the perimeter of $\partial \Omega$ is still one (as the dynamics is invariant under rescaling), so that the dissipative billiard map $f_\lambda$ is defined on the same phase space $\A$ as $g_\lambda$. Moreover, for any $\eta>0$, there exists $\epsilon_0(\eta)>0$ such that if $d_{C^2}(\partial \Omega,\mathcal{E})<\epsilon_0(\eta)$, then $d_{C^1}(f_\lambda,g_\lambda)<\eta$.\footnote{Indeed, by \eqref{matrice differ}, the differentials of $f_\lambda,g_\lambda$ depend continuously on the curvature function.}  
		Fix $\eta>0$ small enough such that for any $\Omega$ with $d_{C^2}(\partial \Omega,\mathcal{E})<\epsilon_0:=\epsilon_0(\eta)$, the $2$-periodic orbits $\{H_1,H_2\}$ and $\{E_1,E_2\}$ have continuations $\{H_1(\Omega),H_2(\Omega)\}$ (of saddle type) and $\{E_1(\Omega),E_2(\Omega)\}$ (of sink type) for $f_\lambda$. \\
Let $\delta_0>0$ be such that for every $0<\delta<\delta_0$, the balls $B(E_1,\delta), B(E_2,\delta)$ are both contractible, i.e., any closed path contained in $B(E_1,\delta)$, resp. $B(E_2,\delta)$, is homotopic to a point. As already noticed with inclusion (\ref{cc unstable}) in Proposition \ref{coro attra}, we can fix $0<\delta<\delta_0$ small enough such that there exists $N\in\N$ with the property that for any $n>N$, the set
		$$
		g_\lambda^n(\mathcal{W}^u_{\delta/2}(\mathcal{O}_{g_\lambda}(H_1)))\setminus g_\lambda^N(\mathcal{W}^u_{\delta/2}(\mathcal{O}_{g_\lambda}(H_1))) \subset B(E_1,{\delta/2})\cup B(E_2,{\delta/2})\, .
		$$
		where $\mathcal{W}^u_{\delta/2}(\mathcal{O}_{g_\lambda}(H_1)):=\mathcal{W}^u_{\delta/2}(H_1;g_\lambda^2)\cup\mathcal{W}^u_{\delta/2}(H_2;g_\lambda^2)$.
By Hadamard-Perron Theorem --see e.g. \cite[Proposition 5.6.1]{BrinStuck}-- local invariant manifolds of hyperbolic fixed points depend continuously on the dynamics in the $C^1$-topology (hence depend continuously on $\partial\Omega$ in the $C^2$-topology). Consequently, there exists $0<\epsilon_1<\epsilon_0$ such that for every domain $\Omega$ with $d_{C^2}(\partial\Omega,\mathcal{E})<\epsilon_1$, for any $n>N$, the set
		$$
		f_\lambda^n(\mathcal{W}^u_\delta(\mathcal{O}_{f_\lambda}(H_1(\Omega))))\setminus f_\lambda^N(\mathcal{W}^u_\delta(\mathcal{O}_{f_\lambda}(H_1(\Omega))))$$
is contained in $B(E_1,{\delta})\cup B(E_2,{\delta})$, where $\mathcal{W}^u_\delta(\mathcal{O}_{f_\lambda}(H_1(\Omega))):=\mathcal{W}^u_{\delta}(H_1(\Omega);f_\lambda^2)\cup\mathcal{W}^u_{\delta}(H_2(\Omega);f_\lambda^2)$. Observe that $E_i(\Omega)\in B(E_i,\delta)$, for $i=1,2$.
In particular, the $2$-periodic orbit $\{E_1(\Omega),E_2(\Omega)\}$ belongs to $$\overline{\mathcal{W}^u(\mathcal{O}_{f_\lambda}(H_1(\Omega))}:=\overline{\mathcal{W}^u(H_1(\Omega);f_\lambda^2)\cup\mathcal{W}^u(H_2(\Omega);f_\lambda^2)}$$ and the latter is an $f_\lambda$-invariant, compact and connected set.
\begin{claim} The set $\overline{\mathcal{W}^u(\mathcal{O}_{f_\lambda}(H_1(\Omega))}$ separates $\A$.
		\end{claim}
		\begin{proof}[Proof of the claim.]
			As for $\overline{\mathcal{W}^u(\mathcal{O}_{f_\lambda}(H_1(\Omega))}$, we use the notation 
$$\overline{\mathcal{W}^u(\mathcal{O}_{g_\lambda}(H_1))} := \overline{\mathcal{W}^u(H_1(\Omega);g_\lambda^2)  \cup \mathcal{W}^u(H_2(\Omega);g_\lambda^2)}\, . $$

Observe that the second one separates the annulus because, by Theorem \ref{ellisse}, it is the Birkhoff attractor of $g_\lambda$. The claim then follows if we show that they are homotopic. From the previous choice of $\epsilon_1>0$ we can decompose
			$$
			\overline{\mathcal{W}^u(\mathcal{O}_{f_\lambda}(H_1(\Omega)))}=\gamma_1\cup\gamma_2\cup\hat\gamma_1\cup \hat \gamma_2\,,
			$$
			where for $i=1,2$, $\gamma_i\subset B(E_i,\delta)$ is a continuous path containing $E_i(\Omega)$ whose endpoints $q_i^1,q_i^2$ lie on the circle $\mathcal{C}(E_i,\delta)$ (centered at $E_i$ of radius $\delta$), while $\hat\gamma_i\subset f_\lambda^N(\mathcal{W}^u_\delta(H_i(\Omega);f_\lambda^2))$ is an unstable arc with endpoints $q_1^i$ and $q_2^i$. Similarly,
			$$
			\overline{\mathcal{W}^u(\mathcal{O}_{g_\lambda}(H_1))}=\Gamma_1\cup\Gamma_2\cup\hat\Gamma_1\cup \hat \Gamma_2\,,
			$$
			where for $i=1,2$, $\Gamma_i\subset B(E_i,\delta)$ is a continuous path containing $E_i$ whose endpoints $Q_i^1,Q_i^2$ lie on the circle $\mathcal{C}(E_i,\delta)$, while $\hat\Gamma_i\subset g_\lambda^N(\mathcal{W}^u_\delta(H_i;g_\lambda^2))$ is an unstable arc with endpoints $Q_1^i$ and $Q_2^i$. \\
Fix $0<\delta'\ll \delta$. We can retract a $\delta'$-neighborhood of the circles $\mathcal{C}(E_1,\delta)$ and $\mathcal{C}(E_2,\delta)$ in such a way that the respective images $\{\gamma_i',\hat \gamma_i',\Gamma_i',\hat \Gamma_i'\}_{i=1,2}$ of $\{\gamma_i,\hat \gamma_i,\Gamma_i,\hat \Gamma_i\}_{i=1,2}$ after retraction satisfy that for $i=1,2$, $\gamma_i',\Gamma_i'$ have the same endpoints, and $\hat\gamma_i',\hat\Gamma_i'$ have the same endpoints. In particular, it holds that $\gamma_1\cup\gamma_2\cup\hat\gamma_1\cup\hat\gamma_2$ is homotopic to $\Gamma_1\cup\Gamma_2\cup\hat\Gamma_1\cup\hat\Gamma_2$ if and only if $\gamma_1'\cup\gamma_2'\cup\hat\gamma_1'\cup\hat\gamma_2'$ is homotopic to $\Gamma_1'\cup\Gamma_2'\cup\hat\Gamma_1'\cup\hat\Gamma_2'$. We denote by $B_1',B_2'$ the respective images of $B(E_1,\delta)$ and $B(E_2,\delta)$ after retraction. \\
Since for $i=1,2$, the set $B_i'$ is contractible, and $\gamma_i'\cup\Gamma_i'$ is a closed loop, we deduce that $\gamma_i'$ and $\Gamma_i'$ are homotopic. Besides, by the continuous dependence of the local unstable manifolds on the dynamics, the unstable arcs $\hat \gamma_i$ and $\hat \Gamma_i$ are $C^0$-close to each other, and then, the paths $\hat \gamma_i'$ and $\hat \Gamma_i'$ are homotopic. We conclude that
			$\gamma_1'\cup\gamma_2'\cup\hat\gamma_1'\cup \hat \gamma_2'$ is homotopic to $\Gamma_1'\cup\Gamma_2'\cup\hat\Gamma_1'\cup \hat \Gamma_2'$; by construction, it follows that  $\overline{\mathcal{W}^u(\mathcal{O}_{f_\lambda}(H_1(\Omega)))}=\gamma_1\cup\gamma_2\cup\hat\gamma_1\cup \hat \gamma_2$ is also homotopic to $	\overline{\mathcal{W}^u(\mathcal{O}_{g_\lambda}(H_1))}=\Gamma_1\cup\Gamma_2\cup\hat\Gamma_1\cup \hat \Gamma_2$. \qedhere 

		\end{proof} 
\noindent As a consequence of the previous claim, we have  $\overline{\mathcal{W}^u(\mathcal{O}_{f_\lambda}(H_1(\Omega))} \in\mathcal{X}(f_\lambda)$, hence --by Proposition \ref{minimality}-- the Birkhoff attractor of $f_\lambda$ is contained in it. Since the Birkhoff attractor cannot be reduced to $\{E_1(\Omega),E_2(\Omega)\}$, it must contain points in the unstable manifold of the saddle periodic orbit. Since the Birkhoff attractor is invariant and closed, it holds that $H_1(\Omega), H_2(\Omega)\in \Lambda_\lambda$. Repeating the proof of Claim \ref{claim local unstable}, we can show that for any point $x$ in the local unstable manifold of $\{H_1(\Omega),H_2(\Omega)\}$, the set $\overline{\mathcal{W}^u(\mathcal{O}_{f_\lambda}(H_1(\Omega)))} \setminus\{x\}$ does not separate the annulus. By Lemma \ref{lemma disconnected annulus}, we deduce that the local unstable manifold is contained in the Birkhoff attractor. Again, since $\Lambda_\lambda$ is invariant and closed, we deduce that $\Lambda_\lambda=\overline{\mathcal{W}^u(\mathcal{O}_{f_\lambda}(H_1(\Omega))) }$. \\
Finally, the fact that the function $(\mathcal{E},\lambda)\mapsto \epsilon(\mathcal{E},\lambda)$ can be chosen to be continuous follows from the fact that the objects for which certain conditions have to be satisfied while choosing $\epsilon_0,\epsilon_1$ above depend continuously on the dynamics $f_\lambda$ (in the $C^1$-topology), which itself depends continuously on the eccentricity of the ellipse $\mathcal{E}$ and on the dissipation parameter $\lambda\in (0,1)$.  
		\qedhere

\end{proof} 


\section{Birkhoff attractors for strong dissipation} \label{cinque}

In the previous section, as a first example, we have seen that, in the case of a circular table, the Birkhoff attractor is the simplest possible, i.e. the graph of the zero function. We are thus 
	naturally led to investigate when $\Lambda_{\lambda}$ is topologically simple, that is, it is a graph. The main result of the present section is proving that the geometric condition contained in Definition \ref{defi set d k} 
	together with the hypothesis that the dissipation is strong, i.e., with $\lambda$ close to $0$, are sufficient for the Birkhoff attractor to be a graph. The corresponding result (Theorem \ref{coro norm hyp}) contains also details on dynamics' and graphs' regularity and utilizes the notions of dominated splitting for an invariant set and of normally contracted (called also hyperbolic) manifold. These definitions are recalled at the beginning of the next section.
	 
	Throughout this section, for simplicity, we will assume that the dissipative billiard maps $f_\lambda$ have constant dissipation $\lambda \in (0,1)$. Yet, we will argue in Remark \ref{argue variable} that the same results can be obtained for a general dissipative billiard map as in Definition \ref{definit diss bill}.  

\subsection{Normally contracted Birkhoff attractors for strong dissipation}

Let $f_{\lambda}\colon\A \to \A, \ (s,r) \mapsto (s',r')$ be the dissipative billiard map within a convex domain $\Omega \subset \R^2$. We use notations of Section \ref{DBM}.
\begin{proposition}\label{normal hyp}
	Assume that $\Omega \in \mathcal{D}^k$, where $\mathcal{D}^k$ has been introduced in Definition \ref{defi set d k}. 
	Then, there exists $\lambda(\Omega)\in (0,1)$ and a cone-field $\mathcal{C}=(\mathcal{C}(s,r))_{(s,r)\in\A}$ containing the horizontal direction:
	$$
	\mathcal{C}(s,r):=\{v\in T_{(s,r)}\A :\ v=(v_s,v_r), \vert v_r\vert \leq\eta(r)\vert v_s\vert\}\, ,
	$$ 
where $\eta:[-1,1]\to \R_+$ is a continuous function, such that for each $ \lambda\in (0,\lambda(\Omega))$, and for each $(s,r)\in \mathbb{T}\times [-\lambda(\Omega),\lambda(\Omega)]$
	$$
	Df_\lambda(s,r)\mathcal{C}(s,r)\subset \mathrm{int}\,  \mathcal{C} (f_\lambda(s,r))  \cup\{0\} \, .
	$$
\end{proposition}

\begin{proof} We recall that for any $(s,r)\in \A$, and $(s',r'):=f_\lambda(s,r)$, $\tau(s,r):=\ell(s,s')=\|\Upsilon(s)-\Upsilon(s')\|$ is the Euclidean distance between the  consecutive bounces $\Upsilon(s)$, $\Upsilon(s')$, so that the quantity $\tau(s)$ in Definition \ref{defi set d k} is merely $\tau(s) = \tau(s,0)$. Moreover, since $\Omega \in \mathcal{D}^k$, there exists a constant $c_0>0$ such that
	\begin{equation*}\label{geom condit}
	\max_{s \in \mathbb{T}} \tau(s,0)\mathcal{K}(s)<-1-c_0\,.
	\end{equation*}
By compactness and continuity of the involved functions, we can fix $\delta_0 > 0$ and $K_0>0$ such that
$$
\max_{(s,r)\in\A}\tau(s,r)\leq \diam \Omega<\delta_0,\qquad \max_{s \in \mathbb{T}}|\mathcal{K}(s)|<\mathcal{K}_0\, .
$$
We can find $\lambda_1\in(0,1)$ small enough such that
	\begin{equation}\label{cond cones}
	\max_{(s,r)\in \mathbb{T} \times [-\lambda_1,\lambda_1]} \tau(s,r)\mathcal{K}(s)+\nu(r)<-c_0,\quad  \max_{(s,r)\in \mathbb{T} \times [-\lambda_1,\lambda_1]} \frac{\tau(s,r)}{\nu(r)}<\delta_0\, ,
	\end{equation}
	where $\nu(r):= \sqrt{1 - r^2}$. By \eqref{matrice differ}, for each $(s,r)\in \mathrm{int}(\mathbb{A})$, we have
	\begin{equation}\label{cond diff}
	Df_{\lambda}(s,r) e_1=
	\begin{bmatrix}  
	- \frac{\tau \mathcal{K} + \nu}{\nu'} 
	\\
	\lambda\big(\tau \mathcal{K} \mathcal{K}' + \mathcal{K} \nu'+ \mathcal{K}' \nu \big) 
	\end{bmatrix},\qquad
	Df_{\lambda}(s,r) e_2=
	\begin{bmatrix}  
	\frac{\tau}{\nu \nu'} \\ 
	\\
	-\lambda \frac{\tau \mathcal{K}'+ \nu'}{\nu}
	\end{bmatrix},
	\end{equation}
	where $e_{1}=(1,0)^T$, $e_2(0,1)^T$ are the vectors of the canonical basis. Moreover, $\mathcal{K}$, $\mathcal{K}'$ denote the curvatures at the points corresponding to $s,s'$, while $\nu := \nu(r) = \sqrt{1 - r^2}$, $\nu' := \nu' (r') = \sqrt{1 - \left(\frac{r'}{\lambda}\right)^2}$. \\ 
	\noindent Let $\alpha_0:=\frac{c_0}{2\delta_0}>0$. At each $(s,r)\in \A$, we identify $T_{(s,r)}\A$ with $\R^2$, and define the cone $$
	\mathcal{C}^{\alpha_0}(s,r):=\{u=a e_1+ b e_2\in \mathbb{R}^2: |b|\leq \alpha_0\nu(r) |a|\}.
	$$
We note that this cone always contains the horizontal direction $\R\times \{0\}$. 
	Fixed now $\lambda \in (0,\lambda_1)$, let $(s,r)\in f_\lambda(\mathbb{A})=\mathbb{T} \times [-\lambda,\lambda]$, $(s',r'):=f_\lambda(s,r)$. 
By \eqref{cond cones} and \eqref{cond diff}, for any $u=ae_1+be_2\in \mathcal{C}^{\alpha_0}(s,r)$, its image $u'$ by $Df_\lambda(s,r)$ is equal to
$$u' = \left[-a \frac{\tau \mathcal{K} + \nu}{\nu'} + b \frac{\tau}{\nu \nu'} , a \lambda\big(\tau \mathcal{K} \mathcal{K}' + \mathcal{K} \nu'+ \mathcal{K}' \nu \big) - b\lambda \frac{\tau \mathcal{K}'+ \nu'}{\nu} \right]^T=: a'e_1+b'e_2\, ,$$
where $\nu=\nu(r),\nu'=\nu'(r')$. Thus, we have
	\begin{align*}
	\nu(r') |a'|& \geq \nu'(r') |a'| \geq |a|c_0-|b|\delta_0\geq \frac{c_0}{2}|a|,\\
	|b'|&\leq |a| \lambda \left((\delta_0\mathcal{K}_0^2+2\mathcal{K}_0)+\alpha_0(\delta_0 \mathcal{K}_0+1)\right).
	\end{align*}
	Given now $\mu_0\in (0,1)$, it holds 
	$$
	Df_\lambda(s,r) \mathcal{C}^{\alpha_0}(s,r)\subset \mathcal{C}^{\mu_0 \alpha_0}(s',r')\,,
	$$
	provided that  
	$
	\lambda\in (0,\lambda(\Omega))$, with 
	$$
	\lambda(\Omega)=\lambda(\delta_0,\mathcal{K}_0,c_0,\mu_0):=\min\left(\lambda_1,\frac{\mu_0 \alpha_0 c_0}{2(\delta_0\mathcal{K}_0^2+2\mathcal{K}_0)+2\alpha_0(\delta_0 \mathcal{K}_0+1)}\right)\in (0,1)\, .
	$$
	
Setting $\mathcal{C}(s,r)=\mathcal{C}^{\alpha_0}(s,r)$ for each $(s,r)\in\A$, we conclude the proof.
 \end{proof}
\noindent In order to continue on this section, we need to recall some notions and results: the definition of dominated splitting for an invariant set, the definition of normally contracted (hyperbolic) manifold (see e.g. \cite[Definition 2.2]{Sambarino} and \cite{BergerBounemoura} respectively) and a Theorem by Hirsch-Pugh-Shub on the regularity of such normally contracted manifolds (see \cite{HirschPughShub}).
 \begin{definition}[Dominated splitting]
 	Let $M$ be a compact Riemannian manifold without boundary. Let $f\colon M \to M$ be a $C^\ell$ diffeomorphism onto its image, $\ell\geq 1$. Let $K$ be an invariant set for $f$. Then, $K$ has a \emph{dominated splitting} if the tangent bundle over $K$ splits into two subbundles $T_KM=E\oplus F$ such that
 	\begin{enumerate}
 		\item $E$ and $F$ are invariant by $Df$;
 		\item the subbundles $E$ and $F$ vary continuously with respect to the point $x\in K$;
 		\item there exist $C>0$ and $0<\nu<1$ such that for any $x\in K$,
 		$$
 		\Vert Df^n(x)\vert_{E}\Vert \,\cdot \Vert Df^{-n}(f^n(x))\vert_{F}\Vert \leq C\, \nu^n,\quad \forall\, n \geq 0.
 		$$ 
 	\end{enumerate}
 \end{definition}
\noindent Roughly speaking, the previous definition says that any direction not contained in the subbundle $E$ converges exponentially fast to the direction $F$ under iteration of $Df$. 
For the following definitions of ($\ell$-)normal contraction, we refer to \cite{HirschPughShub} and \cite{CrovisierPotrie}.

\begin{definition}[Normally contracted manifold]\label{defi norm contra man}
	Let $M$ be a compact Riemannian manifold without boundary. Let $f\colon M \to M$ be a $C^\ell$ diffeomorphism onto its image, $\ell\geq 1$. Let $N$ be a closed $C^1$ manifold, invariant under $f$. Then, we say that $N$ is \emph{normally contracted} if $N$ has a dominated splitting $T_NM=E^s \oplus TN$ such that $E^s$ is uniformly contracted, i.e., there exists $n_0\in \N$ and $\mu\in (0,1)$ such that for any $n\geq n_0$ it holds 
	$$
	\|Df^n(x)|_{E^s}\|\leq\mu^n,\quad \forall\, x \in N\,.
	$$
	Moreover, we say that $N$ is \emph{$\ell$-normally contracted} if the above splitting satisfies the following stronger condition: there exist $C>0$ and $0<\nu<1$ such that for any $x\in N$, and for any $1 \leq j \leq \ell$,  
	$$
	\Vert Df^n(x)\vert_{E^s}\Vert \,\cdot \Vert Df^{-n}(f^n(x))\vert_{TN}\Vert^j \leq C\, \nu^n,\quad \forall\, n \geq 0\,.
	$$ 
\end{definition}

Once we have a $\ell$-normally contracted manifold, then the following theorem by Hirsch-Pugh-Shub assures that the manifold is as regular as the dynamics.
\begin{theorem}\label{l normal contr man}\cite{HirschPughShub}
	Let $M$ be a compact Riemannian manifold without boundary. Let $f\colon M \to M$ be a $C^\ell$ diffeomorphism onto its image, $\ell\geq 1$. Let $N$ be a closed $C^1$ manifold, invariant under $f$. If $N$ is $\ell$-normally contracted, then $N$ is actually a $C^\ell$ manifold. 
\end{theorem}

We can now state an interesting outcome of Proposition \ref{normal hyp}.
 
 \begin{corollary}\label{coro dom splitting}
 		Let $\Omega\in\mathcal{D}^k$, 
 		$k \geq 2$. Let $\lambda(\Omega)\in(0,1)$ be 
 		given by Proposition \ref{normal hyp}. 
 		Then, 
 		for any $\lambda\in(0,\lambda(\Omega))$, the attractor $\Lambda_\lambda^0$ has a dominated splitting 
 		$
 		E^s\oplus E^c
 		$
 		where the bundle $E^s$ is uniformly contracted, and each point $(s,r)\in \Lambda_\lambda^0$ has a stable manifold $\mathcal{W}^s(s,r)$, which is transverse to the horizontal. Moreover, there exists $0<\lambda'(\Omega)<\lambda(\Omega)$ such that for some $C>0$ and $0<\nu<1$, we have that for any $\lambda \in (0,\lambda'(\Omega))$, for any $x\in \Lambda_\lambda^0$, and for any $1 \leq j \leq k-1$,  
 		\begin{equation}\label{k contracting}
 		\Vert Df_\lambda^n(x)\vert_{E^s}\Vert \,\cdot \Vert Df_\lambda^{-n}(f_\lambda^n(x))\vert_{E^c}\Vert^j \leq C\, \nu^n,\quad \forall\, n \geq 0\,.
 		\end{equation}
 \end{corollary}

 \begin{proof}
	Let $\lambda(\Omega)\in (0,1)$ be as in Proposition \ref{normal hyp}. Take $\lambda \in (0,\lambda(\Omega))$. By the cone-field criterion (see e.g. \cite[Theorem 2.6]{CrovisierPotrie} and \cite[Proposition 2.2]{Sambarino}) for the cone-field $\mathcal{C}=(\mathcal{C}(s,r))_{(s,r)\in\A}$ constructed in Proposition \ref{normal hyp}, we deduce that 
	the attractor $\Lambda_\lambda^0\subset f_\lambda(\mathbb{A})$ has a dominated splitting $E_\lambda^s\oplus E_\lambda^c=E^s\oplus E^c$, where $E^c(s,r)$ is contained in the horizontal cone $\mathcal{C}(s,r)$, for each $(s,r)\in \Lambda_\lambda^0$. Moreover, the fiber bundle $E^s$ is uniformly contracted. Indeed, by the domination, there exist $\mu\in (0,1)$ and $n_0 \in \mathbb{N}$ such that for each $(s,r)\in \Lambda_\lambda^0$, and for each $n \geq n_0$,
	$$
	\|Df_\lambda^n(s,r)|_{E^s}\|\leq \mu^n \|Df_\lambda^n(s,r)|_{E^c}\|\,.
	$$
	Observe that, by \eqref{determinant appl billard}, we have $\det Df_\lambda^n(s,r)=\lambda^n$. As the angle between $E^s$ and $E^c$ is bounded away from zero at each point of $\Lambda_\lambda^0$, thus uniformly as $\Lambda_\lambda^0$ is compact, and by a change of basis, we conclude that there exists $n_1 \in \mathbb{N}$ such that for each $(s,r)\in \Lambda_\lambda^0$, and for each $n \geq n_1$,
	$$
	\|Df_\lambda^n(s,r)|_{E^s}\|\leq \lambda^{\frac n 2},
	$$
	 i.e., the bundle $E^s$ is uniformly contracted. In particular, by the Stable Manifold Theorem (see e.g. \cite{HirschPughShub} or \cite{CrovisierPotrie}), each point $(s,r)\in \Lambda_\lambda^0$ has a stable manifold $\mathcal{W}^s(s,r)$, which is uniformly transverse to the cone-field $\mathcal{C}$ which contains the horizontal direction. 
	 
	 The center bundle $E^c$ is contained in a cone around the horizontal direction and independent of $\lambda$. 
	 By \eqref{cond diff}, the modulus of the projection over the first coordinate of $Df_\lambda(s,r)(1,0)^T$ does not depend on $\lambda$. Since the central direction $E^c$ is contained in a cone around the horizontal direction, we deduce that there exist constants $0<C_1<C_2$ 
	 such that for any $\lambda \in (0,\lambda(\Omega))$, it holds 
	 \begin{equation}\label{premier eq ec control}
	 C_1\leq \|Df_\lambda(s,r)|_{E^c}\|\leq C_2,\quad \forall\, (s,r)\in \Lambda_\lambda^0\,.
	 \end{equation}
	 Since $\det Df_\lambda(s,r)=\lambda$, reasoning as above, we deduce that there exists a constant $C_3>0$ such that for any $\lambda \in (0,\lambda(\Omega))$, it holds
	 \begin{equation}\label{premier eq es control}
	 \|Df_\lambda(s,r)|_{E^s}\|\leq C_3 \lambda,\quad \forall\, (s,r)\in \Lambda_\lambda^0\,.
	 \end{equation}
	 By \eqref{premier eq ec control} and \eqref{premier eq es control}, we conclude that for $\lambda'(\Omega)\in (0,\lambda(\Omega))$ sufficiently small, \eqref{k contracting} holds for any $\lambda \in (0,\lambda'(\Omega))$, for any $x\in \Lambda_\lambda^0$, and for any $1 \leq j \leq k-1$. 
\end{proof} 

 \begin{remark}
	Observe that, if we could say \textit{a priori} that $\Lambda^0_\lambda$ is a $C^1$ manifold, then Corollary \ref{coro dom splitting} would be saying that $\Lambda^0_\lambda$ is $\ell$-normally contracted.
\end{remark}
	
The following proposition guarantees that the center space $E^c$ of the dominated splitting of $\Lambda_\lambda^0$ integrates uniquely to the Birkhoff attractor (see the work \cite{BonattiCrovisier} of Bonatti-Crovisier for related results in this direction). 

\begin{theorem}\label{coro norm hyp}
Let $\Omega\in\mathcal{D}^k$, $k \geq 2$, and let $\lambda(\Omega)\in(0,1)$ be 
given by Proposition \ref{normal hyp}.
Then, for $\lambda \in (0,\lambda(\Omega))$, the Birkhoff attractor $\Lambda_\lambda$ of $f_\lambda$ coincides with the attractor $\Lambda_\lambda^0$ and is a normally contracted $C^{1}$ graph over $\mathbb{T}\times\{0\}$.  
Let $\lambda'(\Omega)<\lambda(\Omega)$ be given by Corollary \ref{coro dom splitting}. Then, for $\lambda \in (0,\lambda'(\Omega))$, $\Lambda_\lambda=\Lambda_\lambda^0$ is actually a $C^{k-1}$ graph and $\Lambda_\lambda$ converges in the $C^1$ topology to the zero section $\T \times \{0\}$ as $\lambda \to 0$.
\end{theorem}	
\begin{proof}
	Fix $\lambda \in (0,\lambda(\Omega))$, 
and let $\mathcal{C}=(\mathcal{C}(s,r))_{(s,r)\in\A}$ be the cone-field in $\mathbb{T}\times [-\lambda,\lambda]$ constructed in Proposition \ref{normal hyp}; let us recall that it contains the horizontal direction, as $\mathcal{C}(s,r)=\{v\in T_{(s,r)}\A : v=(v_s,v_r), \vert v_r\vert \leq \alpha_0\,\nu(r)\,\vert v_s\vert\}$. Let 
	\begin{equation*}
		\mathscr{F}:=\Big\{\gamma\colon\mathbb{T}\to [-\lambda,\lambda]\text{ such that }\gamma\in C^1(\T)\text{ and }(1,\gamma'(s))\in \mathcal{C}(s,\gamma(s)),\ \forall\, s \in \T 
		\Big\}\, .
	\end{equation*} 
	The map $f_\lambda$ acts on $\mathscr{F}$ by the graph transform 
	$$
	\mathscr{G}_{f_\lambda}\colon \mathscr{F} \to \mathscr{F},\quad \gamma \mapsto \left(s \mapsto \pi_2 \circ f_\lambda \big(g_\lambda^{-1}(s),\gamma(g_\lambda^{-1}(s)) \big)\right),
	$$
	where $\pi_1,\pi_2$ denote the projection on the first and second coordinate, respectively, and $g_\lambda\colon \T \to \T$ is the map $s\mapsto \pi_1 \circ f_\lambda(s,\gamma(s))$. Indeed, the cone-field $\mathcal{C}$ around the horizontal direction is contracted by the dynamics, i.e., $Df_\lambda(s,r)\mathcal{C}(s,r)\subset \mathrm{int}(\mathcal{C}(f_\lambda(s,r)))\cup\{0\}$, 
	hence for any $\gamma\in\mathscr{G}_{f_\lambda}$, the image by $f_\lambda$ of the graph of $\gamma$ is still the graph of a $C^1$ function, such that the vector tangent to $f_\lambda(\mathrm{graph}(\gamma))$ is in $\mathcal{C}$, and $\pi_1 \circ f_\lambda|_{\mathrm{graph}(\gamma)}$ is a homeomorphism between $\mathrm{graph}(\gamma)$ and $\mathbb{T}$. In particular, $f_\lambda\big(\mathrm{graph}(\gamma)\big)=\mathrm{graph}(\mathscr{G}_{f_\lambda}(\gamma))$ for a well-defined function $\mathscr{G}_{f_\lambda}(\gamma)\in \mathscr{F}$.
	
	 For any $k\in \mathbb{N}$, let denote $\mathbb{A}_{k}:=f_\lambda^{k}(\mathbb{A})$, and let $\mathscr{F}_k$ be the subset of functions $\gamma \in \mathscr{F}$ whose graph is contained in $\A_k$. Note that $\mathscr{F}_{k+1}\subset \mathscr{F}_k$, and by construction, it holds $\mathscr{F}= \mathscr{F}_1$. Moreover, if $\gamma \in \mathscr{F}_k$, $k \geq 1$, then it holds $\mathscr{G}_{f_\lambda}(\gamma)\in \mathscr{F}_{k+1}$. 
	
	In the sequel, let $\|\cdot\|_\infty$ be the sup-norm on the space $C^0(\T,[-1,1])$. That is, for $\gamma_1,\gamma_2\in C^0(\T,[-1,1])$, we let $\|\gamma_2-\gamma_1\|_\infty:=\max_{s \in \mathbb{T}}|\gamma_2(s)-\gamma_1(s)|$. The graph transform acts as a contraction on $\mathscr{F}$, for $\|\cdot\|_\infty$.
	\begin{claim}\label{claim contr sup norm}
	There exists a contant $c>0$ such that for any
    $\gamma_1,\gamma_2\in \mathscr{F}$, it holds
    $$
    \|\mathscr{G}_{f_\lambda}^n(\gamma_2)-\mathscr{G}_{f_\lambda}^n(\gamma_1)\|_\infty\leq c\lambda^{\frac n 2} \|\gamma_2-\gamma_1\|_\infty,\quad \forall\, n \geq 0.
    $$		
	\end{claim}

\begin{proof}[Proof of the claim]
	For $k_0 \geq 1$ sufficiently large, $\mathbb{A}_{k_0}$ is foliated by stable leaves $\{\mathcal{W}^s(x)\cap \mathbb{A}_{k_0}:x\in \Lambda_\lambda^0\}$, and by the transversality between $E^s$ and $\mathcal{C}$ on $\Lambda_\lambda^0$, there exists $\theta_0>0$ such that 
	\begin{equation}\label{transv cond} 
	\angle(T_{x}\mathcal{W}^s(x),T_x\Gamma_\gamma)\geq \theta_0,\quad \forall\, x\in \mathrm{graph}(\gamma),\, \gamma \in \mathscr{F}_{k_0}\,,
	\end{equation} 
	where $\angle$ denotes the (non-oriented) angle between the considered vector subspaces.
	
	Let $\gamma_1,\gamma_2\in \mathscr{F}_{k_0}$. 
	For each $s \in \mathbb{T}$, we denote by $H_{\gamma_1,\gamma_2}^s(s)\in \mathrm{graph}(\gamma_2)$ the image of $(s,\gamma_1(s))\in \mathrm{graph}(\gamma_1)$ by the holonomy map from $\mathrm{graph}(\gamma_1)$ to $\mathrm{graph}(\gamma_2)$ along the leaves of $\mathcal{W}^s$. That is, follow the stable leaf passing through $(s,\gamma_1(s))$ until it intersects the graph of $\gamma_2$: such intersection point is $H^s_{\gamma_1,\gamma_2}(s)$.
	
	For $n \geq 1$, denote by $\gamma_i^n$ the image $\mathscr{G}_{f_\lambda}^n(\gamma_i)$, $i=1,2$. Then, by \eqref{transv cond}, there exists a constant $c>1$ such that for each $\gamma_1,\gamma_2\in \mathscr{F}_{k_0}$, 
	\begin{equation}\label{compa stable vert}
	c^{-1} |\gamma_2(s)-\gamma_1(s)|\leq d_{\mathcal{W}^s}((s,\gamma_1(s)),H_{\gamma_1,\gamma_2}^s(s))\leq c |\gamma_2(s)-\gamma_1(s)|,\quad \forall\, s \in \mathbb{T}\,,
	\end{equation}
	where $d_{\mathcal{W}^s}$ denotes the distance along a stable leaf\footnote{More precisely, the distance $d_{\mathcal{W}^s}$ comes from the restriction of the Riemannian metric to the stable leaves.}. 
	Moreover, since $(s,\gamma_1(s))$ and $H_{\gamma_1,\gamma_2}^s(s))$ belong to the same stable leaf, there exists $n_0\in \mathbb{N}$ such that for $n \geq n_0$, 
	\begin{equation}\label{contr stabl}
	d_{\mathcal{W}^s}(f_\lambda^n (s,\gamma_1(s)),f_\lambda^n \circ H_{\gamma_1,\gamma_2}^s(s))\leq \lambda^{\frac n 2} d_{\mathcal{W}^s}((s,\gamma_1(s)),H_{\gamma_1,\gamma_2}^s(s)),\quad \forall\, s \in \mathbb{T}. 
	\end{equation}
	For each $s\in \mathbb{T}$ and $n \geq 0$, let us set $s_{-n}:=g_\lambda^{-n}(s)$, with $g_\lambda \colon s\mapsto \pi_1 \circ f_\lambda(s,\gamma_1(s))$ as above, so that $f_\lambda^n (s_{-n},\gamma_1(s_{-n}))=(s,\gamma^n_1(s))$, and $f_\lambda^n \circ H_{\gamma_1,\gamma_2}^s(s_{-n})=H_{\gamma_1^n,\gamma_2^n}^s(s)$. Indeed, for $i=1,2$, $f_\lambda^n$ sends the graph of $\gamma_i$ to the graph of $\gamma_i^n$, and it sends stable leaves to stable leaves. Then, by \eqref{compa stable vert}-\eqref{contr stabl}, for $n \geq n_0$, it holds 
	\begin{align*}
	 |\gamma_2^n(s)-\gamma_1^n(s)|&\leq c d_{\mathcal{W}^s}((s,\gamma_1^n(s)),H_{\gamma_1^n,\gamma_2^n}^s(s))\\
	 &=c d_{\mathcal{W}^s}(f_\lambda^n (s_{-n},\gamma_1(s_{-n})),f_\lambda^n \circ H_{\gamma_1,\gamma_2}^s(s))\\
	 &\leq c\lambda^{\frac n 2} d_{\mathcal{W}^s}((s_{-n},\gamma_1(s_{-n})),H_{\gamma_1,\gamma_2}^s(s_{-n}))\\
	 &\leq c^2 \lambda^{\frac n 2} |\gamma_2(s_{-n})-\gamma_1(s_{-n})|\leq c^2 \lambda^{\frac n 2} \|\gamma_2-\gamma_1\|_\infty.
	\end{align*}
	Now, for any $\bar \gamma_1,\bar \gamma_2\in \mathscr{F}=\mathscr{F}_1$, their images under $\mathscr{G}_{f_\lambda}^{k_0-1}$ are in $\mathscr{F}_{k_0}$, hence, up to enlarging the constant $c$, we conclude that for any $n \geq 0$,
	$$
	\|\mathscr{G}_{f_\lambda}^n(\bar\gamma_2)-\mathscr{G}_{f_\lambda}^n(\bar\gamma_1)\|_\infty\leq c^2\lambda^{\frac n 2} \|\bar\gamma_2-\bar\gamma_1\|_\infty.\qedhere
	$$
	\end{proof}
	Let $\gamma_0^\pm\colon s\in\mathbb{T} \mapsto \pm\lambda\in(-1,1)$; clearly, $\gamma_0^\pm\in\mathscr{F}$. For $n \geq 0$, let  $\gamma_{n}^\pm:=\mathscr{G}_{f_\lambda}^{n}(\gamma_0^\pm)$. Observe that for any $n\geq 0$, $\gamma^\pm_{n}$ is in $\mathscr{F}_{n+1}$. 
	By Claim \ref{claim contr sup norm}, the sequences 
	$(\gamma_n^\pm)_{n \geq 0}$ are Cauchy sequences, hence converge to a continuous function $\gamma_\infty^\pm$. 
	\begin{claim}
		The equality $\gamma_\infty^+=\gamma_\infty^-=:\gamma_\infty$ holds and $\Lambda_\lambda^0=\Lambda_\lambda=\mathrm{graph}(\gamma_\infty)$.
	\end{claim}
\begin{proof}	
	Let us denote by $\Gamma_\infty^\pm$ the graph of $\gamma_\infty^\pm$. By construction,  $\Gamma_\infty^\pm$ is invariant under $f_\lambda$, it is compact, connected, and separates the annulus, i.e., $\Gamma_\infty^\pm\in \mathcal{X}(f_\lambda)$. On one hand, by Proposition \ref{minimality}, we have $\Lambda_\lambda\subset \Gamma_\infty^\pm$. On the other hand, by the graph property, for any $x \in \Gamma_\infty^\pm$, the set $\Gamma_\infty^\pm\setminus \{x\}$ does not separate the annulus. By Lemma \ref{lemma disconnected annulus}, we conclude that $\Gamma_\infty^+=\Gamma_\infty^-=\Lambda_\lambda$. 
	
	We can now show that the Birkhoff attractor coincides with the attractor. By definition of the attractor $\Lambda_\lambda^0$, we have $\Lambda_\lambda^0 \subset \cap_{n \geq 0}\A_n$. Observe that, for each $n \geq 0$,  $\A_n$ is a cylinder bounded by the graphs of $\gamma_n^\pm$. Since the sequences $(\mathrm{graph}(\gamma_n^\pm))_{n \geq 0}$ converge to the same limit graph $\Gamma_\infty^+=\Gamma_\infty^-=\Lambda_\lambda$, it follows that $\Lambda_\lambda^0\subseteq \Gamma_\infty^+=\Gamma_\infty^-=\Lambda_\lambda\subseteq \Lambda^0_\lambda$. 
\end{proof}
	
\begin{claim}
	The function $\gamma_\infty$ is $C^1$, and $(1,\gamma_\infty(s))\in E^c(s,\gamma_\infty(s))$  for every $s\in \T$.
\end{claim}

\begin{proof}
	The function $\gamma_\infty=\gamma_\infty^+$ is the $C^0$ limit of the sequence of $C^1$ functions $(\gamma_n^+)_{n \geq 0}$. 
	To show that $\gamma_\infty$ is $C^1$, it suffices to show that the derivatives $((\gamma_n^+)')_{n \geq 0}$ also converge uniformly. 
	For each $n \geq 0$, let us denote by $\Gamma_n^+$ the graph of $\gamma_n^+$. 
	With the same notations as in Claim \ref{claim contr sup norm},  for each $s \in \mathbb{T}$, it holds
	$$
	T_{(s,\gamma_n^+(s))}\Gamma_{n}^+=Df_\lambda^n(s_{-n},\gamma_0^+(s_{-n}))T_{(s_{-n},\gamma_0^+(s_{-n}))}\Gamma_{0}^+\subset Df_\lambda^n(s_{-n},\gamma_0^+(s_{-n})) \mathcal{C}(s_{-n},\gamma_0^+(s_{-n})),
	$$
	and by the cone-field criterion, the cone $Df_\lambda^n(s_{-n},\gamma_0^+(s_{-n})) \mathcal{C}(s_{-n},\gamma_0^+(s_{-n}))$ is exponentially small with respect to $n$, uniformly in $s \in \T$, that is, the amplitude of each cone $Df_\lambda^n \mathcal{C}$ is, up to a uniform constant, equal to $\mu^n$ times the amplitude of the cone $\mathcal{C}$, for some uniform constant $\mu\in(0,1)$. We conclude that the sequence $((\gamma_n^+)')_{n \geq 0}$ converges in the $C^1$-topology. Moreover, it also implies that at any point $(s,\gamma_\infty(s))=\lim_{n \to+\infty} (s,\gamma_n^+(s))$, the tangent space of the limit graph $\Gamma_{\infty}$ is equal to $E^c(s,\gamma_\infty(s))$. 
\end{proof}

So far, we have shown that the (Birkhoff) attractor $\Lambda_\lambda^0=\Lambda_\lambda$ is the graph of a $C^1$ function $\gamma_\infty$. 
Moreover, the graph of $\gamma_\infty$ is tangent to $E^c$, and so $\Lambda_\lambda$ is a normally contracted  $C^1$ graph, since 
$$
T\A|_{\Lambda_\lambda}=E^s\oplus E^c=E^s \oplus T\Lambda_\lambda. 
$$
	
	Now, let $\lambda'(\Omega)<\lambda(\Omega)$ be given by Corollary \ref{coro dom splitting}. The domain $\partial \Omega$ is of class $C^k$, and thus, the dissipative billiard map $f_\lambda$ is of class $C^{k-1}$ for any $\lambda \in (0,1)$. Fix $\lambda \in (0,\lambda'(\Omega))$. Since $T\Lambda_\lambda=E^c$, and by \eqref{k contracting}, we deduce that $\Lambda_\lambda$ is $(k-1)$-normally contracted in the sense of Definition \ref{defi norm contra man}, hence by Theorem \ref{l normal contr man}, the function $\gamma_\infty$ is actually $C^{k-1}$.  
	
	By construction, $\Lambda_\lambda=\Lambda_\lambda^0\subset f_\lambda(\A)=\T \times [-\lambda,\lambda]$, hence $\Lambda_\lambda$ converges to the zero section $\T \times \{0\}$ in the $C^0$-topology. To show the convergence in the $C^1$-topology, it suffices to show that  $T \Lambda_\lambda=E^c$ converges uniformly to the horizontal space. By construction, at any $s \in \T$, $E^c(s,\gamma^\infty(s))\subset Df_\lambda(s_{-1},\gamma_0^+(s_{-1})) \mathcal{C}(s_{-1},\gamma_0^+(s_{-1}))$, and by \eqref{cond diff}, the vertical component of vectors in the latter cone is less than $\tilde c \lambda$, for some constant $\tilde c>0$. 
\end{proof}

\subsection{Examples and further consequences}
	
	As a first consequence of Theorem \ref{coro norm hyp}, we prove that --when the dissipation is strong, i.e., $\lambda$ is close to zero-- the Birkhoff attractor of ellipses is a normally contracted $C^\infty$ graph. Given an ellipse $\mathcal{E}$ of non-zero eccentricity, we let $\{E_1,E_2\}$ be the $2$-periodic orbit corresponding to the minor axis; it is a sink, by Lemma \ref{lemme vp reelles ellipse}. Then, we define  
	$$
	\lambda_-(\mathcal{E}) :=\frac{1-\sqrt {1-(-2 (\frac{a_2}{a_1})^2+1)^2}}{1+\sqrt {1-(-2 (\frac{a_2}{a_1})^2+1)^2}}\in (0,1),
	$$
	i.e., $\lambda_-(\mathcal{E})=\lambda_-(p)$ for $p = E_1$ or $E_2$ as in \eqref{serve dopo}.

	By Theorem \ref{ellisse}, we have $\Lambda_\lambda= \overline{\mathcal{W}^u(H_1;f_\lambda^2)\cup \mathcal{W}^u(H_2;f_\lambda^2)}$, and for $i=1,2$, $\mathcal{W}^u(H_i;f_\lambda^2)\setminus \{H_i\}$ is the disjoint union of two branches $\mathscr{C}_i^1,\mathscr{C}_i^2$, with $\overline{\mathscr{C}_i^j}=\mathscr{C}_i^j\cup \{H_i,E_j\}$, $j=1,2$. Thus, $\Lambda_\lambda$ is a manifold if and only if for $i,j \in \{1,2\}$, the tangent space $T_x \mathscr{C}_i^j=T_x \Lambda_\lambda$ has a limit $V_i^j$ as $\mathscr{C}_i^j\ni x \to E_j$. Indeed, by Corollary \ref{COR a}, if so, we necessarily have $V_1^j=V_2^j$. Clearly, a necessary condition for this to hold is that the eigenvalues of $Df_\lambda^2(E_i)$ are real, for $i=1,2$, i.e., $\lambda \in (0,\lambda_-(\mathcal{E}))$. Actually, the following holds.
	\begin{corollary}\label{coro dom splitting ellipse} 
		Let $f_\lambda \colon \mathbb{A} \to \mathbb{A}$ be the dissipative billiard map inside an ellipse
		$\mathcal{E}$ with eccentricity $e\in (0,\frac{\sqrt 2}{2})$. Then, there exists $\lambda(\mathcal{E})<\lambda_-(\mathcal{E})$ such that, for $\lambda \in (0,\lambda(\mathcal{E}))$, the Birkhoff attractor $\Lambda_\lambda= \overline{\mathcal{W}^u(H_1;f_\lambda^2)\cup \mathcal{W}^u(H_2;f_\lambda^2)}$ is a normally contracted $C^1$ graph, which is actually $C^\infty$ except possibly at $E_i$, $i=1,2$, where $\Lambda_\lambda$ is tangent to the weak stable space of $Df_\lambda^2(E_i)$.   
	\end{corollary}
	
	\begin{proof}
		Let $\Upsilon\colon \mathbb{T} \to \R^2$ be a parametrization of the boundary by arclength such that $\Upsilon(0),\Upsilon(\frac 12)$ correspond to the trace on $\mathcal{E}$ of the $2$-periodic orbit of maximal length. For  each $s \in \mathbb{T}$, let $\tau(s)=\tau(s,0):=\|\Upsilon(s)-\Upsilon(s')\|$, where $\Upsilon(s),\Upsilon(s')$ are the two points of intersection of $\mathcal{E}$ and the normal to $\mathcal{E}$ at $\Upsilon(s)$, and let $\mathcal{K}(s)\leq 0$ be the curvature at $\Upsilon(s)$. The function $[0,\frac{1}{4}]\ni s \mapsto \tau(s) \mathcal{K}(s)$ is increasing, with $\tau(\frac{1}{4})\mathcal{K}(\frac 14)=-2\big(\frac{a_2}{a_1}\big)^2=2(e^2-1)$, where $e>0$ is the eccentricity  of $\mathcal{E}$. Thus, if $e\in(0,\frac{\sqrt{2}}{2})$, the domain bounded by $\mathcal{E}$ is in $\mathcal{D}^\infty$ (recall Definition \ref{defi set d k}), hence by Corollary \ref{coro dom splitting} and Theorem \ref{coro norm hyp}, there exists $\lambda(\mathcal{E})<\lambda_-(\mathcal{E}))$ such that for $\lambda \in (0,\lambda(\mathcal{E}))$, the Birkhoff attractor $\Lambda_\lambda$ is a normally contracted $C^1$ graph. In fact, it is $C^\infty$ everywhere except possibly at $E_1,E_2$; indeed, near any other point, it coincides with some piece of the unstable manifold of $H_1$ or $H_2$, which is $C^\infty$. 
		By Lemma \ref{lemme vp reelles}, the eigenvalues $\mu_1,\mu_2$ of $Df_\lambda^2(E_i)$ satisfy $\lambda^2<\mu_1<\mu_2<1$. As $\Lambda_\lambda$ is $C^1$ and $f_\lambda$-invariant, for $i=1,2$, any tangent vector $v\in T_{E_i}\Lambda_\lambda$ is an eigenvector of $Df_\lambda^2(E_i)$; since $\Lambda_\lambda$ is normally contracted, any such $v$ has to be in the eigenspace associated to the weak eigenvalue $\mu_2$. 
	\end{proof}
\begin{remark}
		If the eccentricity is larger than $\frac{\sqrt 2}{2}$, then we loose the graph property, even for small dissipation parameters $\lambda \in (0,1)$, see Proposition \ref{propo not graph} below. 
\end{remark}
\begin{remark}
	It was asked to us by Viktor Ginzburg whether the phase space $\A$ of dissipative billiards admits an invariant foliation by curves homotopic to the zero-section $\T \times \{0\}$. Indeed, in the case of a dissipative billiard within a circle considered at the beginning of Section \ref{quattro}, it is clear that the horizontal foliation $\{\T \times \{r\}\}_{r \in [-1,1]}$ is preserved by any dissipative map $f_\lambda$, $\lambda\in(0,1)$. More generally, the existence of such foliations seems much less rigid than in the conservative case, where it is related to the famous Birkhoff conjecture (see e.g. \cite{AvilaKalSimoi,KalSorr,BialyMironov} for recent progress in this direction). 
	
	Indeed, fix a domain $\Omega \in \mathcal{D}^k$, $k \geq 2$, and a dissipation parameter $\lambda \in (0,\lambda(\Omega))$. With the notations of Theorem \ref{coro norm hyp}, for any $k\geq 0$, let $\A_k:=f_\lambda^k(\A)$, and let $\mathcal{F}_1$ be a foliation of $\A_1\setminus \A_2$ defined as follows. Note that $\A_1\setminus \A_2$ has two connected components $\A_1^+$ and $\A_1^-$, where $\A_1^\pm$ is bounded by the leaves $\T \times \{\pm\lambda\}$ (in $\A_1^\pm$) and $f_\lambda(\T \times \{\pm\lambda\})$ (in the complement of $\A_1^\pm$). Let then $\mathcal{F}_1$ be the disjoint union of two foliations $\mathcal{F}_1^+$ and $\mathcal{F}_1^-$, where $\mathcal{F}_1^\pm$ is a foliation of $\A_1^\pm$ by $C^1$ graphs over $\T\times \{0\}$ whose tangent space remains in the cone-field $\mathcal{C}$ constructed in Proposition \ref{normal hyp}, and whose boundary leaves are $\T \times \{\pm\lambda\}$ and $f_\lambda(\T \times \{\pm\lambda\})$.  For $k \geq 0$, let $\mathcal{F}_k$ be the foliation of $\A_{k}\setminus \A_{k+1}$ whose leaves are images by $f_\lambda^{k-1}$ of the leaves of $\mathcal{F}_1$. Since the cone-field $\mathcal{C}|_{\T\times [-\lambda,\lambda]}$ is contracted under forward iteration, for each $k \geq 1$, the leaves of $\mathcal{F}_k$ are $C^1$ graphs over $\T \times \{0\}$ whose tangent space is contained in the cone-field $\mathcal{C}$. Moreover, the same argument as in the proof of Theorem \ref{coro norm hyp} says that the collection of leaves of $\mathcal{F}_k$ converges uniformly to the Birkhoff attractor $\Lambda_\lambda$ in the $C^1$-topology as $k \to +\infty$. Let $\mathcal{F}_0:=f_\lambda^{-1}(\mathcal{F}_1)$, and let $\mathcal{F}$ be the foliation $\mathcal{F}:=\sqcup_{k \geq 0} \mathcal{F}_k\sqcup \Lambda_\lambda$. By construction, it is a foliation of $\A$ by $C^1$ curves, and it is (forward-)invariant under $f_\lambda$. Moreover, the leaves of $\mathcal{F}\cap \A_1$ are $C^1$ graphs over $\T\times \{0\}$. 
\end{remark}

As $\lambda$ gets smaller and smaller, the Birkhoff attractor of $f_\lambda$ is contained in a smaller and smaller strip around the zero section; actually, we can use the $C^1$ convergence of the Birkhoff attractor to $\T\times\{0\}$ to deduce interesting information on the dynamics of $f_\lambda\vert_{\Lambda_\lambda}$, when $\lambda \in (0,1)$ is small, from the degenerate $1$-dimensional dynamics of $f_0$, namely when $\lambda=0$. 
\begin{theorem}\label{theo impr res norm contr}
	Let $k \geq 2$. For a $C^k$-generic billiard $\Omega\in\mathcal{D}^k$ there exists $\lambda''(\Omega)\in(0,1)$ such that for any $\lambda \in (0,\lambda''(\Omega))$, 
	the Birkhoff attractor $\Lambda_\lambda$ is a $C^{k-1}$ normally contracted graph of rotation number $\frac 12$, and, moreover, 
	$$
	\Lambda_\lambda=\bigcup_{i=1}^\ell \overline{\mathcal{W}^u(H_i;f^2_{\lambda}) \cup \mathcal{W}^u(f_{\lambda}(H_i);f^2_{\lambda})}\, ,
	$$
	for some finite collection $\{H_i,f_\lambda(H_i)\}_{i=1,\cdots,\ell}$ of $2$-periodic orbits of saddle type.  
\end{theorem}

\begin{proof}
Let $\Omega \in \mathcal{D}^k$ be a $C^k$-generic billiard as in Corollary \ref{nondeg billiard in dk}, and let $\lambda'(\Omega) \in (0,1)$ be given by Theorem \ref{coro norm hyp}. For any $\lambda \in (0,\lambda'(\Omega))$, the Birkhoff attractor $\Lambda_\lambda$ of $f_\lambda$ is normally contracted, and is equal to the graph $\Gamma_{\gamma_\lambda}$ of some $C^{k-1}$ function $\gamma_\lambda \colon \T \to [-1,1]$. We let $g_\lambda \colon \T\to \T$ be the circle map $\T\ni s \mapsto\pi_1 \circ f_\lambda(s,\gamma_\lambda(s))$ induced by $f_\lambda|_{\Lambda_\lambda}$, where $\pi_1\colon \A=\T \times [-1,1] \to \T$ is the projection over the first coordinate. 
For any $s \in \T$, let $(s_1,\gamma_\lambda(s_1)):=f_\lambda(s,\gamma_\lambda(s))$. By \eqref{matrice differ}, it holds 
\begin{equation}\label{multiplier g lambda}
g_\lambda'(s)=
-\frac{\tau(s,\gamma_\lambda(s)) \mathcal{K}(s) + \nu(s)}{\nu'(s)} +\frac{\tau(s,\gamma_\lambda(s))}{\nu(s) \nu'(s)}\gamma_\lambda'(s),
\end{equation} 
where $\tau(s,\gamma_\lambda(s))$ is the length of the orbit segment for $f_\lambda$ (also of $f_1$) connecting the points $\Upsilon(s)$ and $\Upsilon(s_1)$, $\mathcal{K}(s)$ is the curvature at $\Upsilon(s)$, and $\nu(s) = \sqrt{1 - \gamma_\lambda^2(s)}$, $\nu' (s):= \sqrt{1 - \left(\frac{\gamma_\lambda(s_1)}{\lambda}\right)^2}$. 

Let us note that the function $\T\ni s\mapsto \pi_1\circ f_\lambda(s,0)$ is independent of the value of $\lambda \in [0,1]$. In particular, for any $\lambda\in(0,1)$, it holds $\pi_1\circ f_\lambda|_{\T \times \{0\}}=\pi_1\circ f_1|_{\T \times \{0\}}=\pi_1\circ f_0|_{\T \times \{0\}}$.
We denote such a function by $g_0$. Note that the function $g_\lambda=\pi_1\circ f_\lambda |_{\Lambda_\lambda}$ is $C^0$-converging to $g_0$ as $\lambda \to 0$, since $\Lambda_\lambda$ converges to the zero section, by Theorem \ref{coro norm hyp}. The extended family $(g_\lambda)_{\lambda \in [0,\lambda'(\Omega))}$ satisfies:
\begin{claim}\label{claim circle diffeo}
	The family of maps $(g_\lambda)_{\lambda \in [0,\lambda'(\Omega))}$ depends continuously on $\lambda$ in the $C^1$-topology. 
\end{claim}

\begin{proof}
By the theory of normally contracted invariant manifolds and their persistence (see e.g. \cite[Theorem 2.1 and Corollary 2.2]{BergerBounemoura}), $\gamma_\lambda$ depends continuously on $\lambda\in (0,\lambda'(\Omega))$ in $C^{k-1}$-topology, hence $g_\lambda$ also depends continuously on $\lambda$ in $C^{k-1}$-topology. Moreover, by Theorem \ref{coro norm hyp},  $\gamma_\lambda$ converges to the zero function $0$ in the $C^1$-topology as $\lambda \to 0$, hence $g_\lambda$ converges to the map $g_0$ in the $C^1$-topology. 
\end{proof} 
In particular, $g_\lambda$ converges to the map $g_0$ in the $C^1$-topology, with $g_0'\colon s \mapsto -\tau(s,0) \mathcal{K}(s)-1$, where $\tau(s,0)$ is the length of the first orbit segment for $f_0$ (also for $f_1$) starting at $(s,0)$, and $\mathcal{K}(s)$ is the curvature at $\Upsilon(s)$. Note that $-\tau(s,0) \mathcal{K}(s)-1\neq 0$, since $\Omega$ is in $\mathcal{D}^k$. In particular, $g_0$ is a circle diffeomorphism. Let us denote by $\mathrm{II}$ the set of $2$-periodic points of the family $\{f_\lambda\}_{\lambda \in[0,1]}$. As already observed, the set $\Pi$ is common to every $f_\lambda$. Since the set $\mathrm{II}$ is contained in the zero section $\T\times \{0\}$, the circle diffeomorphism  $g_0$ has rotation number $\frac 12$. Moreover, by Corollary \ref{nondeg billiard in dk}, for a $C^k$-generic domain $\Omega$, for any $\lambda \in [0,1)$, all the $2$-periodic points of the billiard map $f_\lambda$ are either saddles or sinks. In particular, the latter persist under perturbation, even when we consider the $1$-dimensional dynamics on the corresponding Birkhoff attractor, as we are going to show.  In fact, for any $2$-periodic point $p=(s,0)\in \mathrm{II}$, denoting by  $\mathcal{K}_1,\mathcal{K}_2$ the curvatures at the two bounces, and by $\tau$ the Euclidean distance between the two bounces, according to \eqref{multiplier g lambda}, the multiplier $(g_0^2)'(s)=\frac{d}{ds}g_0(g_0(s))$ is equal to $k_{1,2}:=(\tau \mathcal{K}_1 + 1)(\tau \mathcal{K}_2 + 1)$. In particular, for the circle diffeomorphism $g_0$, the $2$-periodic point $s$ is repelling when $|k_{1,2}|>1$, and attracting when $|k_{1,2}|<1$.  
By Claim \ref{claim circle diffeo}, for $\lambda>0$ small, the circle diffeomorphism $g_\lambda$ is $C^1$-close to $g_0$. Thus, for any $p=(s,0)\in \mathrm{II}$, the $2$-periodic point $s$ for $g_0$ admits a continuation for $g_\lambda$. Since the set $\Pi$ is common to all functions $f_\lambda$ and since generically $2$-periodic points are isolated, we deduce that $s$ is also $2$-periodic for $g_\lambda$. Therefore, there exists $\lambda''(\Omega)\in (0,\lambda'(\Omega))$ such that for any $\lambda \in (0,\lambda''(\Omega))$, the restriction $f_\lambda|_{\Lambda_\lambda}$ still has rotation number $\frac 12$. Observe that, on the one hand, if $p=(s,0)$ is a sink for $f_\lambda$, then $s$ is an attracting $2$-periodic point for $g_\lambda$; on the other hand, if $p=(s,0)$ is of saddle type for $f_\lambda$, then $s$ is a $2$-periodic repulsive point, because the Birkhoff attractor is normally contracted. By standard facts of the theory of circle homeomorphisms with rational rotation number, the $\alpha$-limit set $\alpha_{f_\lambda^2}(s,r)$ of any point $(s,r)\in \Lambda_\lambda \setminus \mathrm{II}$ is a $2$-periodic point $H=H(s,r)$, which has to be of saddle type (as sinks are repulsive for the past dynamics). Arguing as in the proof of Proposition \ref{coro attra}, we deduce that $(s,r) \in \mathcal{W}^u(H;f_\lambda^2)$.  Similarly, $\omega_{f_\lambda^2}(s,r)=E\in \mathrm{II}$, with $E$ a sink periodic point in $ \overline{\mathcal{W}^u(H;f_\lambda^2)}$. We conclude that 
$$
\Lambda_\lambda=\bigcup_{i=1}^\ell \overline{\mathcal{W}^u(H_i;f^2_{\lambda}) \cup \mathcal{W}^u(f_{\lambda}(H_i);f^2_{\lambda})}\, ,
$$
for some finite collection $\{H_i,f_\lambda(H_i)\}_{i=1,\cdots,\ell}$ of $2$-periodic orbits of saddle type, which concludes the proof. 
\end{proof}
 
 We will now show that for any $C^k$ convex domain in the interior of the complement of $\mathcal{D}^k$, $k \geq 2$, we loose the graph property of $\Lambda_\lambda$ for small dissipation parameters $\lambda \in (0,1)$. This is the case in particular for any ellipse $\mathcal{E}$ of eccentricity $e$ larger than $\frac{\sqrt 2}{2}$.\footnote{Indeed, if $\{(s_0,0),f_1(s_0,0)\}$ is the $2$-periodic orbit along the minor axis of $\mathcal{E}$, then with the notations of Proposition \ref{propo not graph}, an easy computation shows that  $\tau(s_0,0)\mathcal{K}(s_0)=2(e^2-1)>-1$, hence the assumption of Proposition \ref{propo not graph} is satisfied.} As previously, given a strongly convex billiard $\Omega$, for $(s,r)\in \A$, we denote by $\tau(s,r)$ the length of the first orbit segment for the (conservative) billiard map starting at $(s,r)$, and by $\mathcal{K}(s)<0$ the curvature at the point of $\partial \Omega$ associated to $s$. 
 \begin{proposition}\label{propo not graph}
 	Let $k \geq 2$, and let $\Omega$ be a strongly convex domain with $C^k$ boundary in the complement of $\mathcal{D}^k$, such that there exists $s_0 \in \T$ with $\tau(s_0,0)\mathcal{K}(s_0)>-1$. 
 	 Then, for $\lambda \in (0,1)$ sufficiently small, the Birkhoff attractor $\Lambda_\lambda$ is not a graph over $\T \times \{0\}$. 
 \end{proposition}
 
 \begin{proof}
 	By contradiction, let us assume that there exists a sequence $(\lambda_n)_{n \in \mathbb{N}}\in (0,1)^\mathbb{N}$ converging to $0$ such that $\Lambda_{\lambda_n}$ is the graph 
 	of some function $\gamma_n\colon \T \to [-1,1]$. As $\Lambda_{\lambda_n}$ separates $\A$, the function $\gamma_n$ is necessarily continuous. We can then define the map 
 	$$
 	g_{\lambda_n}\colon \T \to \T,\quad s\mapsto \pi_1 \circ f_{\lambda_n}(s,\gamma_n(s))\,,
 	$$ 
 	where $\pi_1 \colon \A \to \T$ denotes the projection on the first coordinate. Note that by the graph hypothesis, $g_{\lambda_n}$ is invertible. Let us also define $g_0\colon s \mapsto \pi_1 \circ f_0(s,0)$; note that $g_0=\pi_1 \circ f_\lambda(s,0)$, for any $\lambda\in [0,1]$, and that $g_0$ is $C^1$. By construction, $\Lambda_\lambda \subset \T \times [-\lambda,\lambda]$, and $f_\lambda \colon (s,r)\mapsto (s',\lambda r_1')$. Hence for any $\varepsilon>0$, there exists $n_\varepsilon\in \mathbb{N}$ such that for any $n\geq n_\varepsilon$, $d_{C^0}(g_0,g_{\lambda_n})< \varepsilon$. Now, $g_0'(s)=-(\tau(s)\mathcal{K}(s)+1)$. On the one hand, if $\{(s_1,0),(s_2,0)\}$ is a $2$-periodic orbit of maximal perimeter, then, as in the proof of Proposition \ref{remark point crit func}, 
 	$$
 	(\tau(s_1)\mathcal{K}(s_1)+1)(\tau(s_2)\mathcal{K}(s_2)+1)\geq 1\,.
 	$$ 
 	Since $\mathcal{K}< 0$, we conclude that there exists $i\in \{1,2\}$ such that $\tau(s_i)\mathcal{K}(s_i)+1\leq -1$, i.e., $g_0'(s_i)\geq 1$. On the other hand, by assumption, $g_0'(s_0)<0$. We conclude that there exist $s_*\in \T$ and $\eta_1,\eta_2>0$ such that $g_0(s_*-\eta_1)=g_0(s_*+\eta_2)$ but $g_0(s_*)\neq g_0(s_*-\eta_1)$. Let  $\varepsilon:=\frac{1}{3}|g_0(s_*)-g_0(s_*-\eta_1)|$. We deduce that for any $n \geq n_\varepsilon$, 
 	\begin{align*}
 	\text{either }&g_{\lambda_n}(s_*)>g_{\lambda_n}(s_*-\eta_1)\text{ and }g_{\lambda_n}(s_*)>g_{\lambda_n}(s_*+\eta_2)\,,\\
 	\text{or }&g_{\lambda_n}(s_*)<g_{\lambda_n}(s_*-\eta_1)\text{ and }g_{\lambda_n}(s_*)<g_{\lambda_n}(s_*+\eta_2)\,.
 \end{align*} 
 By the continuity of $g_{\lambda_n}$, we deduce that $g_{\lambda_n}$ is not injective, a contradiction. 
 \end{proof}
  
 We conclude this section by discussing the case where the dissipative billiard map $f_\lambda$ has non-constant dissipation. 
 \begin{remark}\label{argue variable}
 	Let us consider a general dissipative billiard map $f_\lambda$ as in Definition \ref{definit diss bill}, for some $C^{k-1}$ dissipation function $\lambda \colon \A \to (0,1)$. We claim that the results presented in this section can be obtained for such a map $f_\lambda$, as long as $\|\lambda\|_{C^1}\ll 1$. Indeed, the main points are: 
 	\begin{itemize}
 		\item the cone-fields constructed in Proposition \ref{normal hyp} will be contracted by $f_\lambda$ provided that $\|\lambda\|_{C^1}\ll 1$; we let $E^s \oplus E^c$ be the dominated splitting on the attractor $\Lambda_\lambda^0$ resulting from the cone-field criterion; 
 		\item as in Corollary \ref{coro dom splitting}, if $\|\lambda\|_{C^1}\ll 1$, then the contraction of $Df_\lambda$ along $E^s$ will be much stronger than the action of $Df_\lambda$ on $E^c$; 
 		\item following the same line as in the proof of Theorem \ref{coro norm hyp}, we can then show that for $\|\lambda\|_{C^1}$ sufficiently small, $\Lambda_\lambda$ will be a $C^1$ graph over $\T \times \{0\}$, and by the previous point, it will be actually $C^{k-1}$ after taking $\|\lambda\|_{C^1}$ possibly even smaller;
 		\item by the previous discussion, for $\|\lambda\|_{C^1}\ll 1$, $\Lambda_\lambda$ will be a smooth graph over $\T\times \{0\}$; then, we claim that the proof of Theorem \ref{theo impr res norm contr} also adapts in this case; indeed, for $\|\lambda\|_{C^1}\ll 1$, the dynamics of $f_\lambda|_{\Lambda_\lambda}$ approximates the limit dynamics of $f_0|_{\T \times \{0\}}$; the key point is that $f_0|_{\T \times \{0\}}$ has only non-degenerate $2$-periodic points, which follows from Lemma \ref{lemma noncst disp} in the case where $\lambda$ is non-constant. 
 	\end{itemize}  
 \end{remark} 

\section{Topologically complex Birkhoff attractors for mild dissipation}\label{section different rho}

\noindent Birkhoff attractors for dissipative billiards described in Sections \ref{quattro} and \ref{cinque} do not make the idea of their possible topological complexity. In fact, following a celebrated result by M. Charpentier \cite[Section 20]{Charpentier} --here Theorem \ref{charp}-- a Birkhoff attractor for a dissipative diffeomorphism may be an ``indecomposable continuum'' (see Fig. \ref{indeccont}), and a sufficient condition for this occurrence is that two rotation numbers associated to the Birkhoff attractor itself are different. 
The aim of this section is proving that such a phenomenon occurs also for Birkhoff attractors of dissipative billiard maps. Moreover, we discuss some topological and dynamical consequences of this phenomenon. \\
\indent The section is organized as follows. After recalling the main definition and properties of a twist map, we present the construction of the upper and the lower rotation numbers associated to the Birkhoff attrcator, as well as the statement of Charpentier's Theorem. Finally, in the case of dissipative billiards, we give a sufficient condition assuring that the corresponding Birkhoff attractor has different upper and lower rotation numbers  and we discuss the dynamical consequences of this fact. 

\begin{figure}[h]
	\centering
	\includegraphics[scale=0.04]{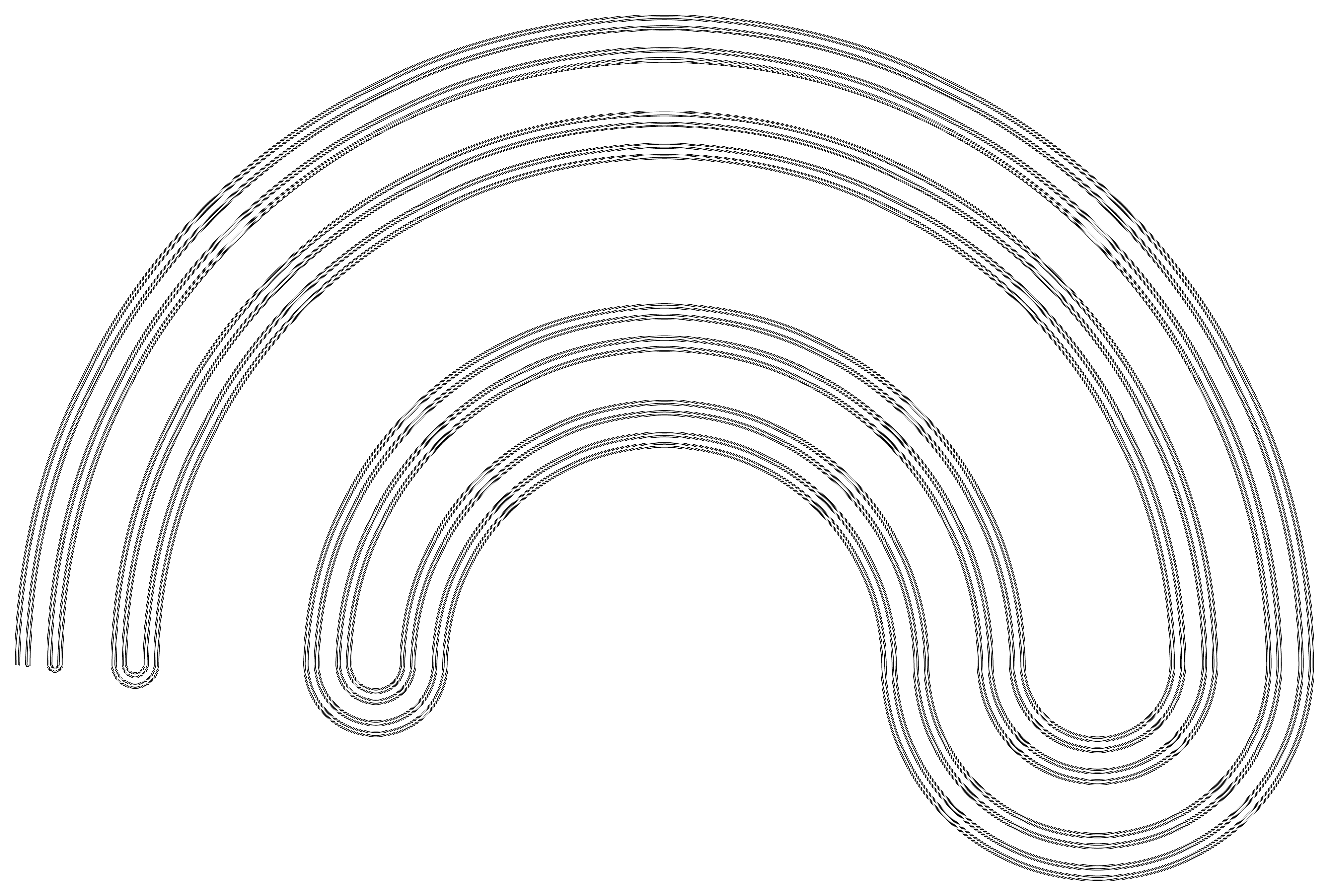}
	\caption{An example of indecomposable continuum (L Rempe-Gillen, CC BY-SA 3.0 \texttt{https://creativecommons.org/licenses/by-sa/3.0}, via Wikimedia Commons).}
	\label{indeccont}
\end{figure} 

\subsection{Twist diffeomorphisms} \label{var}
Fix the standard metric and trivialisation of the tangent space of $\mathbb{A} := \T\times[-1,1]$, as well as the counterclockwise orientation of the plane. Let $\beta\in(0,\frac{\pi}{2})$ and denote by $v$ the unitary vertical vector $(0,1)$. For any $x\in\A$, the cone $C_+(x,\beta)$ is the set of vectors $w\in T_x\A$ such that the angle $\theta(v,w)$ (with respect to the fixed metric and trivialisation) admits a lift in $(-\frac{\pi}{2}+\beta,-\beta)$; similarly, the cone $C_-(x,\beta)$ is the set of vectors $w\in T_x\A$ such that the angle $\theta(v,w)$ admits a lift in $(\frac{\pi}{2}-\beta,\beta)$, see Fig. \ref{figcone}. For the next definition, we refer to \cite[Section 1.2]{Herman}.
\begin{figure}[h]
	\centering
	\includegraphics[scale=0.10]{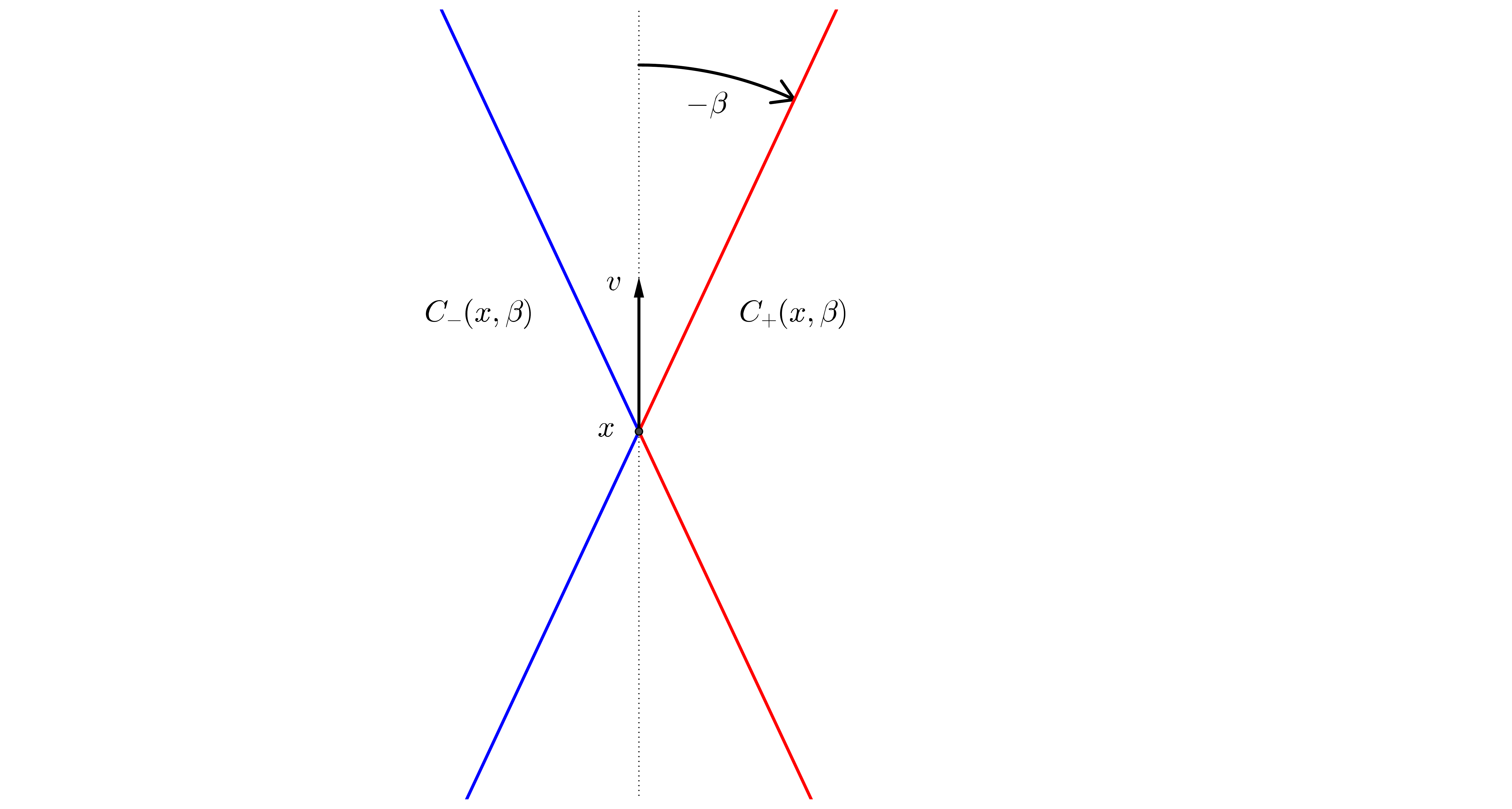}
	\caption{The cones $C_+(,\beta)$ and $C_-(x,\beta)$.}
	\label{figcone}
\end{figure}

\begin{definition}\label{def twist}
	Let $U$ be an open subset of $\A$. A $C^1$ orientation-preserving diffeomorphism $f\colon U\subset \A \to f(U)\subset \A$ is a positive, resp. negative twist map on $U$ if there exists $\beta\in(0,\frac{\pi}{2})$ such that:
	\[
	Df(x)v\in C_+(f(x),\beta), \quad \text{resp. }Df(x)v\in C_-(f(x),\beta),\qquad \forall\, x \in U.
	\]
\end{definition}

We are mostly interested in dissipative twist maps. Nevertheless, if we restrict to constant conformally symplectic twist maps, a variational setting can be described, following [Bangert]. Let $f$ be a constant conformally symplectic twist diffeomorphism of $\text{int}(\mathbb{A})$ into its image of conformality ratio $a > 0$ with respect to the area form $\omega = dr \wedge ds = d\alpha$, where $\alpha = rds$ is the Liouville 1-form. Denote by $F\colon (S,r) \in \R\times[-1,1]\mapsto (S',r')\in\R\times[-1,1]$ a lift of $f$ to the universal cover. The map $f/a$ is an exact symplectic twist diffeomorphism of $\text{int}(\mathbb{A})$; this means that there exists a generating function $H \in C^2(\R^2;\R)$ for $F/a$ such that $F^*\alpha - a \alpha = a dH$, that is
\begin{equation} \label{con a}
	r'dS' - a rdS = a dH(S,S')\, .
\end{equation}
We define a (formal) action functional as
$$
\mathcal{H}\colon \{S_i\}_{i\in\Z}\in\R^\Z\longmapsto \sum_{i\in\Z}\frac{H(S_i,S_{i+1})}{a^i}\, .
$$
\begin{definition}
	A bi-infinite sequence $\{S_i\}_{i\in\Z}\in\R^\Z$ is stationary for $\mathcal{H}$ if
	$$
	\partial_1 H(S_i,S_{i+1})+a\partial_2H(S_{i-1},S_i)=0\qquad\forall i\in\Z\, .
	$$
\end{definition}
We can then characterize the orbits of $F$ in terms of stationary sequences. Indeed, equality (\ref{con a}) means that, for every $S,S'\in\R$,
\begin{equation*} 
	\left\{
	\begin{array}{rcl}
		r &=& -\partial_1H(S,S'), \\
		r' &=& a \partial_2H(S,S').
	\end{array}
	\right.
\end{equation*}
As a consequence, $\{ (S_i,r_i)\}_{i \in \mathbb{Z}}$ is an orbit of $F$ if and only if for every $i \in \mathbb{Z}$, it holds
\begin{equation} \label{comp} 
	-\partial_1H(S_i,S_{i+1}) = r_i = a \partial_2 H(S_{i-1},S_{i})\, .
\end{equation} 
This implies the following.
\begin{proposition} A bi-infinite sequence $\{ (S_i,r_i)\}_{i \in \mathbb{Z}}$ is an orbit of $F$ if and only if the bi-infinite sequence $\{S_i\}_{i \in \mathbb{Z}}$ is stationary.
\end{proposition}

An important example of twist map is given by the billiard map within a convex domain, as recalled in the following proposition. 

\begin{proposition}\label{bill is twist}
	Let $\Omega \subset \mathbb{R}^2$ be a convex domain, with $C^k$ boundary, $k \geq 2$. Then, the associated billiard map $f=f_1\colon \A\to \A$ given by \eqref{stnd billi} is a positive twist map when restricted to $\mathrm{int}(\A)$.  
\end{proposition}

\begin{proof}
	Let $(s,r)\in\mathrm{int}(\A)$. We consider the image of the vertical direction by the differential of $f$: 
	$$
	Df(s,r)\begin{bmatrix}
		0\\1
	\end{bmatrix}=\begin{bmatrix}
		\frac{\tau}{\nu\nu'} \\
		-\frac{(\tau \mathcal{K}'+\nu')}{\nu}
	\end{bmatrix}\,.
	$$
	In order to conclude that $f$ is a positive twist map, it is sufficient to show that for some $M>0$, independent of the point $(s,r)\in\mathrm{int}(\A)$, it holds
	$$
	\dfrac{\vert \tau \mathcal{K}'+\nu'\vert/\nu}{\tau/\nu\nu'}=\vert \tau \mathcal{K}'+\nu'\vert \dfrac{\nu'}{\tau}\leq M\,.
	$$
	Observe that for any point we have $\vert \tau \mathcal{K}'+\nu'\vert\leq \mathrm{diam}(\Omega)\mathcal{K}_0+1$, where  
	$\mathcal{K}_0$ denotes the maximum in absolute value of the curvature of $\partial\Omega$. Thus, it suffices to get a uniform upper bound on $\frac{\nu'}{\tau}$ to conclude. Whenever $\tau$ --which is the Euclidean distance between two consecutive points-- is bounded away from zero, the quantity we are interested in is then clearly bounded. The points for which $\tau$ is approaching zero are points closer and closer to the boundary. Let then $(s_n,r_n)_n\in(\mathrm{int}(\A))^{\mathbb{N}}$ be a sequence of points converging to a point $(s_\infty,\pm 1)$. 
	Without loss of generality, assume that we converge to $(s_\infty, 1)$. Let $\mathcal{K}_\infty\leq 0$ be the curvature at the point on $\partial \Omega$ corresponding to $s_\infty$. We distinguish between two cases: either $\mathcal{K}_\infty<0$, or $\mathcal{K}_\infty=0$. 
	
	In the first case, we let $R_\infty:=|\mathcal{K}_\infty^{-1}|>0$ be the radius of curvature at $s_\infty$. 
	By approximating our convex domain with the osculating circle at the point 
	$s_\infty$, we obtain
	$$
	\tau_n \sim 2\nu_n'R_\infty\,,
	$$
	denoting by $\tau_n$ the Euclidean distance between the points corresponding to $s_n$ and $s_n'$, where $(s_n',r_n'):=f(s_n,r_n)$, and with $\nu_n':=\sqrt{1-(r_n')^2}$ (see e.g. \cite[Chapter 4, I.3.4.]{Douady}). Thus
	$$
	\lim_{n \to+\infty}\dfrac{\nu'_n}{\tau_n}=\frac{1}{2R_\infty}=-\frac{\mathcal{K}_\infty}{2}\leq \frac{\mathcal{K}_0}{2}\, ,
	$$
	which provides the required uniform bound.
	
	In the second case, namely when $\mathcal{K}_\infty=0$, the boundary $\partial\Omega$ is approximated up to order $2$ by the tangent space at $s_\infty$. Let $(\tilde s_n,\tilde r_n)$ and $(\tilde s_n',\tilde r_n')$ be the respective approximations of $(s_n,r_n)$ and $(s_n',r_n'):=f(s_n,r_n)$; then $\tilde r_n=\tilde r_n'=1$ (and the corresponding $\tilde \nu_n,\tilde \nu _n'$ satisfy $\tilde \nu_n=\tilde \nu _n'=0$). Besides, in our approximation, $\tau_n$ is approximated by $|\tilde s_n-\tilde s_n'|$. This yields
	$$
	\lim_{n \to+\infty}\dfrac{\nu'_n}{\tau_n}=0\, .
	$$
	In either case, we go get the required uniform bound. 
\end{proof}

\subsection{Upper and lower rotation numbers and Charpentier's result} \label{6 punto 1}


We follow  the presentation contained in~\cite[Sections 4 and 5]{LeCalvez}. Recall from Definition \ref{def diss map} that 
$$C=\{(s,r)\in\A :\ \phi^-(s)\leq r\leq \phi^+(s)\}\subset \A,$$ 
where $\phi^-,\phi^+\colon\T\to\R$ are continuous maps. For $\lambda\in(0,1)$, let $f_\lambda$ be a dissipative (see Definition \ref{def diss map}) positive twist map of $C$ into its image and $\Lambda_\lambda$ be its corresponding Birkhoff attractor (see Definition~\ref{minimality}). Denote by $C_{\lambda}^+$ (resp. $C_{\lambda}^-$) the connected component of $C\setminus \Lambda_\lambda$ containing $\{(s,\phi^+(s))\in\A :\ s\in\T\}$ (resp. $\{(s,\phi^-(s))\in\A :\ s\in\T\}$). For any $(s,r)\in\A$ the upper (resp. lower) vertical line is
\[
V^+(s,r):=\{(s,y)\in\A :\ y\geq r\}
\]
$\text{(resp.} \ V^-(s,r):=\{(s,y)\in\A :\ y\leq r\}$).
Let us now define (see Fig. \ref{figverticals})
\[
\Lambda_\lambda^+:=\{ x\in\Lambda_\lambda :\  V^+(x)\setminus\{x\}\subset C_{\lambda}^+\}\qquad\text{and}\qquad\Lambda_\lambda^-:=\{ x\in\Lambda_\lambda :\  V^-(x)\setminus\{x\}\subset C_{\lambda}^-\}\,.
\]
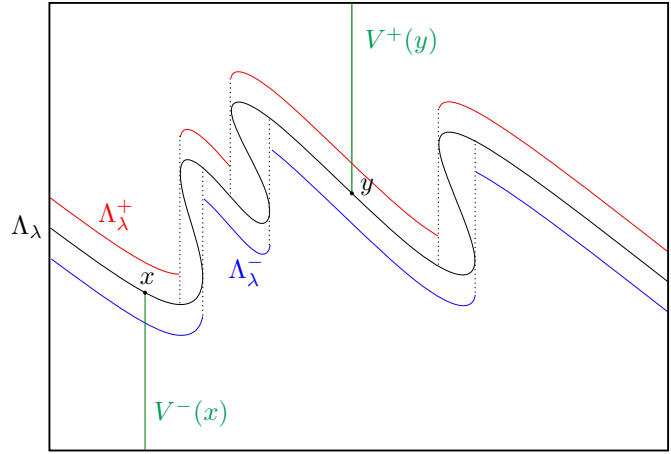
\begin{figure}[h]
	\centering
	\begin{overpic}[width=0.5\textwidth,tics=10]{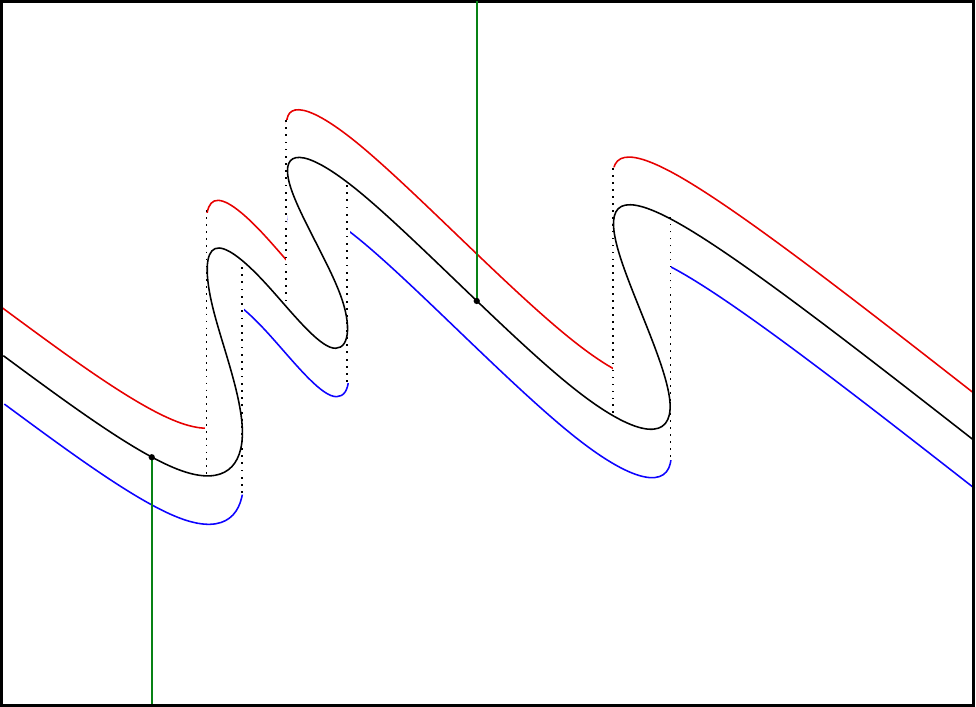}  
		\put(14.7,27){\color{black}$x$}
		\put(17,5){\color{ForestGreen}$V^-(x)$}
		\put(50.2,42.1){\color{black}$y$}
		\put(51,65){\color{ForestGreen}$V^+(y)$}
		\put(-6,35){\color{black}$\Lambda_\lambda$}
		\put(8,37){\color{red}$\Lambda_\lambda^+$}
		\put(29,28){\color{blue}$\Lambda_\lambda^-$}
	\end{overpic}
	\caption{The lower and upper verticals.}
	\label{figverticals}
\end{figure}
Therefore, we can define two functions
$\mu_\lambda^{\pm}\colon\T\to[-1,1]$ 
whose graphs $\Gamma_{\mu_{\lambda}^{\pm}}$ satisfy $\Gamma_{\mu_{\lambda}^\pm}=\Lambda_\lambda^\pm$. \\
\noindent
In the sequel, we fix a covering $\pi\colon\R\times[-1,1]\to\T\times[-1,1]$ of $\A$, and let $\tilde\Lambda_\lambda:=\pi^{-1}(\Lambda_\lambda)$, $\tilde\Lambda_\lambda^\pm:=\pi^{-1}(\Lambda_\lambda^\pm)$; we also denote by  $\tilde\mu_\lambda^\pm\colon\R\to[-1,1]$ the lifts of $\mu_\lambda^{\pm}\colon\T\to[-1,1]$. Moreover,  we let $\pi_1\colon\T\times[-1,1]\to\T$ and $\tilde \pi_1\colon\R\times[-1,1]\to\R$ be the first coordinate projections. The next propositions are Corollary 4.8 and Corollary 4.7-Corollary 4.5 in \cite{LeCalvez} respectively. 

\begin{proposition} \label{cotangente}
	The map $\tilde\mu^+_\lambda\colon\R\to[-1,1]$ $(\text{resp.} \ \tilde\mu_\lambda^-\colon\R\to[-1,1])$ is upper (resp. lower) semi-continuous. Moreover, 
	\[
	\tilde\mu_\lambda^\pm(\tilde\theta')-\tilde\mu_\lambda^\pm(\tilde\theta)\leq (\tilde\theta'-\tilde\theta)\, \mathrm{cotan}\, \beta, \qquad \forall\, \tilde\theta<\tilde\theta',
	\]
	where $\beta\in(0,\frac{\pi}{2})$ is the constant in Definition \ref{def twist}.
\end{proposition}

\begin{proposition} \label{prop technique} The following properties hold:
	\begin{enumerate}[label=(\alph*)]
		\item\label{punt a} $f_\lambda^{-1}(\Lambda_\lambda^\pm) \subset \Lambda_\lambda^\pm$, and the order defined by the first coordinate projection is preserved by $f_\lambda^{-1}$. 
		\item\label{punt b} Let $\hat{U}_\lambda^\pm := \{x \in C : \ V^\pm(x)\subset C_{\lambda}^\pm\}$ be the set of points radially accessible from below/above. If $x\in f_\lambda(C) \cap \hat{U}^\pm_\lambda$, then $f_\lambda^{-1}(x)\in\hat{U}^\pm_\lambda$ and $f_\lambda(V^\pm(f_\lambda^{-1}(x)))\subset\hat{U}^\pm_\lambda$.
	\end{enumerate}
\end{proposition}

\noindent Let $F_\lambda\colon\R\times[-1,1]\to F_\lambda(\R\times[-1,1])$ be a continuous lift of $f_\lambda$. The next result, due to G.D. Birkhoff, is \cite[Proposition 4.11]{LeCalvez}.

\begin{proposition} \label{defproprotnb}
	The sequence $\left( \frac{\tilde \pi_1\circ F_\lambda^n-\tilde \pi_1}{n} \right)_{n\in\N}$ converges uniformly on $\tilde\Lambda_\lambda^+$ $(\text{resp.} \ \tilde\Lambda_\lambda^-)$ to a constant $\rho_\lambda^+$ $(\text{resp.} \ \rho_\lambda^-)$. The constants $\rho_\lambda^+$ and $\rho_\lambda^-$ --called upper and lower rotation numbers-- do depend on the chosen lift, but not their difference.
\end{proposition}

\noindent From the previous result, we immediately conclude that if $\Lambda_\lambda^+\cap\Lambda_\lambda^-\neq \emptyset$ (equivalently, if there is at least a point where $\Lambda_\lambda$ is a graph) then $\rho_\lambda^+=\rho_\lambda^-$. We finally recall that, when the upper and lower rotation numbers are different, then the corresponding Birkhoff attractor is topologically complicated, in the sense made precise by the following result of M. Charpentier (see \cite[Section 20]{Charpentier}).

\begin{theorem}[\cite{Charpentier}] \label{charp}
	If $\rho_\lambda^+-\rho_\lambda^->0$, then $\Lambda_\lambda$ is an indecomposable continuum, i.e., it cannot be written as union of two compact connected non-trivial sets.
\end{theorem}

\subsection{The dissipative billiard case} \label{6 punto 2}

In this section --for a dissipative billiard map-- we give a sufficient condition assuring that the corresponding Birkhoff attractor has different upper and lower rotation numbers (see Proposition \ref{prop different rho}). Clearly, this is not the case of the billiard tables studied respectively in Sections \ref{quattro} and \ref{cinque}. Indeed, for an ellipse, the corresponding Birkhoff attractor --independently from the dissipative parameter-- is not an indecomposable continuum and, in particular, it holds that $\rho_\lambda^+=\rho_\lambda^-=\frac 1 2\, \mod\Z$. This can be proved even more directly. Indeed, since the rotation number is invariant under the dynamics and, for the dissipative billiard map on an ellipse, every point of the Birkhoff attractor is in the omega-limit set of a two periodic point, we can deduce that every point of both $\tilde\Lambda^+_\lambda$ and $\tilde\Lambda^-_\lambda$ has rotation number equal to that of the $2$-periodic point, i.e., equal to $\frac{1}{2}$. In Section \ref{cinque}, we study billiards whose Birkhoff attractor is a graph: in this case, we clearly have that 
$\rho_\lambda^+=\rho_\lambda^-$. \\
\indent Let $\Omega\subset\R^2$ be a strongly convex domain (i.e. whose curvature never vanishes) with $C^k$, $k \ge 2$, boundary $\partial \Omega$. Then the associated (conservative) billiard map $f = f_1$ is a $C^{k-1}$ positive twist map (with respect to some $\beta\in(0,\frac{\pi}{2})$) of $\A:=\T \times [-1,1]$ into itself. Consequently, for every $\lambda\in(0,1)$, the dissipative billiard map $f_\lambda$ defined in Section \ref{DBM} is a $C^{k-1}$ positive twist map (with respect to some $\beta' \ge \beta\in(0,\frac{\pi}{2})$) of $\A:=\T \times [-1,1]$ into its image.

%

We recall that an essential curve in $\A$ is a topological embedding of $\T$ that is not homotopic to a point. 

\noindent The next proposition is an adaptation of \cite{LeCalvez}: mainly, the only difference concerns the type of maps considered, but the proof follows the main lines of \cite[Section 8]{LeCalvez}. Some computations in the proof are simpler because of the kind of maps studied. More precisely, given a $C^1$ function $\lambda\colon \A\to (0,1)\subset \R$, we consider compositions of twist maps with some homothety $\mathcal{H}_{\lambda}$ of factor $\lambda(s,r)$ in the second variable, but in the inverse order with respect to \cite{LeCalvez}. This class of maps contains in particular the dissipative billiard maps considered in the present work. Let us recall the notion of instability region (see e.g. \cite[Definition 2.18]{MCAr}). 

\begin{definition} Let $C=\{(s,r)\in\A:\ \phi^-(s)\leq r\leq \phi^+(s)\}\subset \A$, where $\phi^-,\phi^+\colon\T\to\R$ are continuous maps. Let $f\colon C \to f(C)$ be a twist map on $\mathrm{int}(C)$. Let $\mathscr{V}(f)$ be the union of all $f$-invariant essential curves in $C$. An instability region $\mathscr{I}$ is an open bounded connected component of $C \setminus\mathscr{V}(f)$ that contains in its interior an essential curve.
\end{definition}

\begin{proposition}\label{prop different rho}
	Let $C=\{(s,r)\in\A:\ \phi^-(s)\leq r\leq \phi^+(s)\}\subset \A$, where $\phi^-,\phi^+\colon\T\to\R$ are continuous maps. Let $f\colon C\to C$ be an orientation-preserving homeomorphism, homotopic to the identity, such that 
	\begin{enumerate}
		\item $f$ preserves the standard $2$-form $\omega=dr\wedge ds$;
		\item $f\colon \mathrm{int}(C)\to\mathrm{int}(C)$ is a positive twist map on $\mathrm{int}(C)$ with respect to $\beta\in(0,\frac \pi 2)$;
		\item $\mathscr{I}:=\mathrm{int}(C)$ is an instability region for $f$ that contains the zero section $\T\times\{0\}$.
	\end{enumerate}
	Then, there exists $\bar\epsilon>0$ such that for any $\epsilon\leq \bar\epsilon$, for any $C^1$ function $\lambda\colon C\to (0,1)$ such that $\frac{\epsilon}{2}<d_{C^0}(\lambda,1\footnote{Here, the notation $1$ stands for the constant function.})<\epsilon$ and $\Vert D\lambda \Vert<\epsilon^2$, the Birkhoff attractor\footnote{Let us observe that by the assumptions on $C$, $f$, and by the definition of $f_\lambda$, the Birkhoff attractor of $f_\lambda$ is contained in $\mathrm{int}(C)$.} of $f_\lambda:=\mathcal{H}_\lambda\circ f$ has $\rho^+_\lambda-\rho^-_\lambda>0$, where $\mathcal{H}_\lambda\colon (s,r)\mapsto(s,\lambda(s,r) r)$.
\end{proposition}

\begin{remark}\label{simplify lambda const}
	Observe that, in Proposition \ref{prop different rho}, if we restrict to the class of constant functions $\lambda$, we are stating that there exists $\lambda_0\in(0,1)$ such that, for any $\lambda\in[\lambda_0,1)$, the Birkhoff attractor of the dissipative map $f_\lambda:=\mathcal{H}_\lambda\circ f$ has $\rho^+_\lambda-\rho^-_\lambda>0$, where $\mathcal{H}_\lambda\colon (s,r)\mapsto(s,\lambda r)$.
\end{remark}

\noindent As a corollary of Proposition \ref{prop different rho}, we obtain a sufficient condition to assure that a dissipative billiard map has different rotation numbers. 
\begin{corollary}\label{prop different rho-billiards}
	Let $\Omega\subset\R^2$ be a strongly convex domain with $C^k$, boundary, $k \ge 2$. Let $f=f_1$ be the associated  (conservative) billiard map. If $f$ admits an instability region $\mathscr{I}$ that contains the zero section $\mathbb{T}\times\{0\}$, then there exists $\lambda_0\in (0,1)$ such that, for any $\lambda\in[\lambda_0,1)$, the Birkhoff attractor of the corresponding dissipative billiard map $f_\lambda$ has $\rho_\lambda^+-\rho^-_\lambda>0$, with $\frac 12\in (\rho_\lambda^-,\rho_\lambda^+)\mod\Z$.
\end{corollary}

\begin{remark}\label{rem prop diff}
	Both in Proposition \ref{prop different rho} and in Corollary \ref{prop different rho-billiards}, the boundary of the instability region is made up of the graphs of two continuous functions $\phi^-<0<\phi^+$. These functions are actually Lipschitz by Birkhoff's Theorem, see \cite{Bir3}. In Corollary \ref{prop different rho-billiards}, by the time-reversal symmetry of the conservative billiard map, it even holds $\phi^-=-\phi^+$. 
\end{remark}

\begin{proof}[Proof of Proposition \ref{prop different rho}]
	
	Since $\lambda$ is, in particular, continuous on the compact set $C$, it takes values in $(0,1)$ and since $d_{C^0}(\lambda,1)<\epsilon$, there exist $\lambda_{\min},\lambda_{\max}\in(1-\epsilon, 1)$ such that, for any $(s,r)\in C$ it holds
	$$
	1-\epsilon<\lambda_{\min}\leq \lambda(s,r)\leq \lambda_{\max}<1\, .
	$$
	Since $\Vert D\lambda\Vert<\epsilon^2$, we also have that $\lambda_{\max}-\lambda_{\min}<\epsilon^2$. Since also $d_{C^0}(\lambda,1)>\frac{\epsilon}{2}$, we have that for every $(s,r)\in C$ it holds $\lambda(s,r)\leq\lambda_{\max}<1-\frac{\epsilon}{2}(1-2\epsilon)$.
	
	The map $f_\lambda\colon C \to f_\lambda(C)\subset \mathrm{int}(C)$, defined by $f_\lambda:=\mathcal{H}_\lambda\circ f$, is a dissipative map, according to Definition \ref{def diss map}. Indeed, for every $(s,r)\in \mathrm{int}(C)$, it holds
	$$
	\mathrm{det}(Df_{\lambda}(s,r))=\mathrm{det}(D\mathcal{H}_{\lambda}(s',r'))=r'\partial_2\lambda(s',r')+\lambda(s',r')<1-\dfrac{\epsilon}{2}+2\epsilon^2\, ,
	$$
	where $f(s,r)=(s',r')$; there exists $\epsilon_0$ small enough such that for every $\epsilon<\epsilon_0$ the value $\mathrm{det}(Df_{\lambda}(s,r))$ is uniformly smaller than 1. Let $F$ be a lift of $f$. We denote by $\Lambda_\lambda$ the Birkhoff attractor of $f_\lambda$ and by $\rho_\lambda^\pm$ its lower an upper rotation numbers with respect to the lift $\mathcal{H}_\lambda\circ F$.
	Recall that $f$ is a positive twist map with respect to $\beta\in(0,\frac{\pi}{2})$. Observe that, up to consider a smaller $\epsilon_0$, for any $\epsilon\leq \epsilon_0$, for any function $\lambda$ that is $\epsilon$-$C^1$-close to $1$, the map $f_{\lambda}$ is still a positive twist map on $\mathrm{int}(C)$ with respect to $\frac{\beta}{2}\in(0,\frac{\pi}{4})$.
	
	\noindent Consider now the annulus $\mathcal A$ bounded by $\Gamma_{\phi^-}:=\{(s,\phi^-(s))\in\A:\ s\in\T\}$ and its image $f_\lambda\big(\Gamma_{\phi^-}\big)$. Since $f\big(\Gamma_{\phi^-}\big)=\Gamma_{\phi^-}$, we have
	\[
	f_\lambda\big(\Gamma_{\phi^-}\big)= \mathcal{H}_\lambda\circ f\big(\Gamma_{\phi^-}\big)=\mathcal{H}_\lambda\big(\Gamma_{\phi^-}\big)=\Gamma_{\lambda\phi^-}\,,
	\]
	and consequently $m(\mathcal{A})=-\int_\T(1-\lambda(s,\phi^-(s)))\phi^-(s)\, ds$; in particular, we have
	$$
 -(1-\lambda_{\max})\int_\T \phi^-(s)\, ds \leq m(\mathcal{A})\leq -(1-\lambda_{\min})\int_\T\phi^-(s)\, ds\, .
	$$
	We denote by $C_\lambda^-$ the connected component of $C\setminus\Lambda_\lambda$ containing $\Gamma_{\phi^-}=\{(s,\phi^-(s))\in\A:\ s\in\T\}$. Observe that, for every $n\in \N^*$, it holds $$m(f^n_{\lambda}(\mathcal{A}))=\int_{f^{n-1}_{\lambda}(\mathcal{A})}\vert r\partial_2\lambda(s,r)+\lambda(s,r)\vert\, dr\wedge ds\leq (\lambda_{\max}+\epsilon^2)m(f^{n-1}_{\lambda}(\mathcal{A}))\, .
	$$
	Then, we have
		\begin{eqnarray}\label{calcul somme geom}
			m(C_\lambda^-)=\sum_{n=0}^{+\infty} m(f_\lambda^n(\mathcal{A}))\leq \sum_{n=0}^{+\infty}(\lambda_{\max}+\epsilon^2)^n\, m(\mathcal{A})\leq \\ -(1-\lambda_{\min})\sum_{n=0}^{+\infty}(\lambda_{\max}+\epsilon^2)^n\,  \int_{\T}\phi^-(s)\, ds= \dfrac{1-\lambda_{\min}}{1-\lambda_{\max}-\epsilon^2}\, m(C^-)\nonumber
			\end{eqnarray}
	where $C^\pm:=\{(s,r)\in C,\, \pm r\geq 0\}$. 
	A consequence of \eqref{calcul somme geom} is that, up to choose a smaller $\epsilon_0$, for every $\epsilon\leq \epsilon_0$
	\begin{equation}\label{point ordonn neg}
	\forall \, \lambda\text{ such that } \frac{\epsilon}{2}<d_{C^0}(\lambda,1)<\epsilon, \Vert D\lambda\Vert<\epsilon^2,\, \exists\text{ at least one point } (s_\lambda,r_\lambda) \in \Lambda_\lambda^-\text{ with } r \le \min_{s\in\T}\phi^+(s)/2\, .
	\end{equation}
	
\begin{figure}[h]
	\centering
	\includegraphics[scale=0.8]{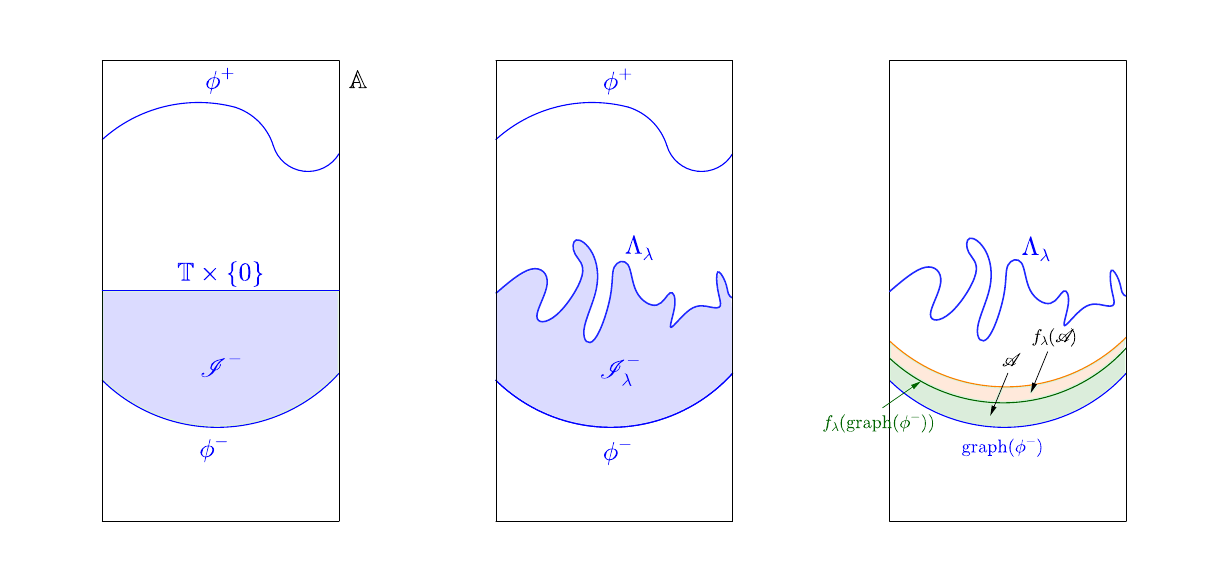}
	\caption{Here, $\mathscr{I}^-:=\mathrm{int}(C^-)$ is the part of the instability region $\mathscr{I}=\mathrm{int}(C)$ that lies below the zero section $\mathbb{T} \times \{0\}$, while $\mathscr{I}_\lambda^-:=\mathrm{int}(C_\lambda^-)$ is the connected component of  $\mathscr{I}\setminus \Lambda_\lambda$ bounded by $\Gamma_{\phi^-}$.}
	\label{ilambdamoins}
\end{figure} 
	
Indeed, if for some function $\lambda$ every point of $\Lambda_\lambda^-$ is contained in the interior of $C^+=\{(s,r)\in C:\ r\geq \min_{s\in\T}\phi^+(s)/2\}$, then $\int_\T\mu_\lambda^-(s)\, ds >\min_{s\in\T}\phi^+(s)/2>0$. We would then obtain $m(C_\lambda^-)=m(C^-)+\int_\T\mu_\lambda^-(s)\, ds>m(C^-)+\min_{s\in\T}\phi^+(s)/2$. Nevertheless, by \eqref{calcul somme geom}, it holds $m(C_\lambda^-)\leq \frac{1-\lambda_{\min}}{1-\lambda_{\max}-\epsilon^2}m(C^-)$ and we get
$$
\min_{s\in\T}\phi^+(s)/2< m(C^-) \left( \dfrac{\lambda_{\max}-\lambda_{\min}+\epsilon^2}{1-\lambda_{\max}-\epsilon^2} \right)\leq m(C^-)\dfrac{2\epsilon}{\frac 1 2 -2\epsilon}\, ;
$$
if $\epsilon$ is small enough, this provides the required contradiction.

\noindent Denote by $\mathscr{C}_\epsilon$ the set of $C^1$ functions $\lambda\colon C\to (0,1)$ such that $\frac{\epsilon}{2}<d_{C^0}(\lambda,1)<\epsilon$ and $\Vert D\lambda\Vert<\epsilon^2$. For every $\epsilon\in(0,\epsilon_0)$, let $\lambda_\epsilon$ be a function in $\mathscr{C}_\epsilon$.
\begin{claim}{\cite[Proposition 8.3]{LeCalvez}}\label{claim 6}
	It holds
	\begin{equation}\label{lim mu}
		\lim_{\epsilon\to 0}\inf_{s\in\T}\mu^\pm_{\lambda_\epsilon}(s)-\phi^\pm(s)=0\,.
	\end{equation}
\end{claim}

\begin{proof}[Proof of the claim:]
	Let us show the claim when $\pm=-$. By contradiction, assume that this does not hold. In particular, there exists $M>0$ and a sequence $\epsilon_n\to 0$ as $n\to+\infty$ such that, for every $n\in\N$, the function $\lambda_{\epsilon_n}\in\mathscr{C}_{\epsilon_n}$ and it holds
	\[
	\Lambda_{\lambda_{\epsilon_n}}\cap H_M=\emptyset\,,
	\]
	where $H_M:=\{(s,r) : \phi^-(s)\leq r\leq \phi^-(s)+M\}\subset C$. \\
	By hypothesis, $f$ preserves the standard $2$-form and it is a positive twist map on $\mathrm{int}(C)$. Thus, by \cite[Section 6]{Birkhoff1} (see also \cite[Proposition 5.9.2]{Herman}), we have
	\begin{equation} \label{inclusione}
		\Gamma_{\phi^+}=\{(s,\phi^+(s))\in \A:\ s\in\T\}\subset \overline{\bigcup_{k\in\Z}f^k(H_M)}=\overline{\bigcup_{n\in\N}f^{-n}(H_M)}\,.
	\end{equation}
	Denote by $L$ the Lipschitz constant of $\phi^+$ and let $E:=\min_{s\in\T}\phi^+(s)>0$. Moreover, recall that every $f_{\lambda_{\epsilon_n}}$ is a positive twist map with respect to the constant $\frac{\beta}{2}\in(0,\frac{\pi}{4})$, where $\beta$ is the twist constant of $f$.  \\
	Choose $j\in\N$ such that
	\begin{equation}
		\dfrac{1}{j\frac E 2-2-L}\leq \dfrac{\tan\left(\beta/2\right)}{2}\,.
	\end{equation}
	For $i=0,\dots,j-1$, denote $s_i=i/j\,\mod 1$ and
	\[
	B_i:=\left\{(s,r)\in C : |s-s_i|<\frac{1}{2j},\, |\phi^+(s)-r|<\frac{2}{j}\right\}\,.
	\]
	
	\begin{figure}[h]
		\centering
		\includegraphics[scale=0.8,trim=2cm 1.5cm 2cm 1.5cm]{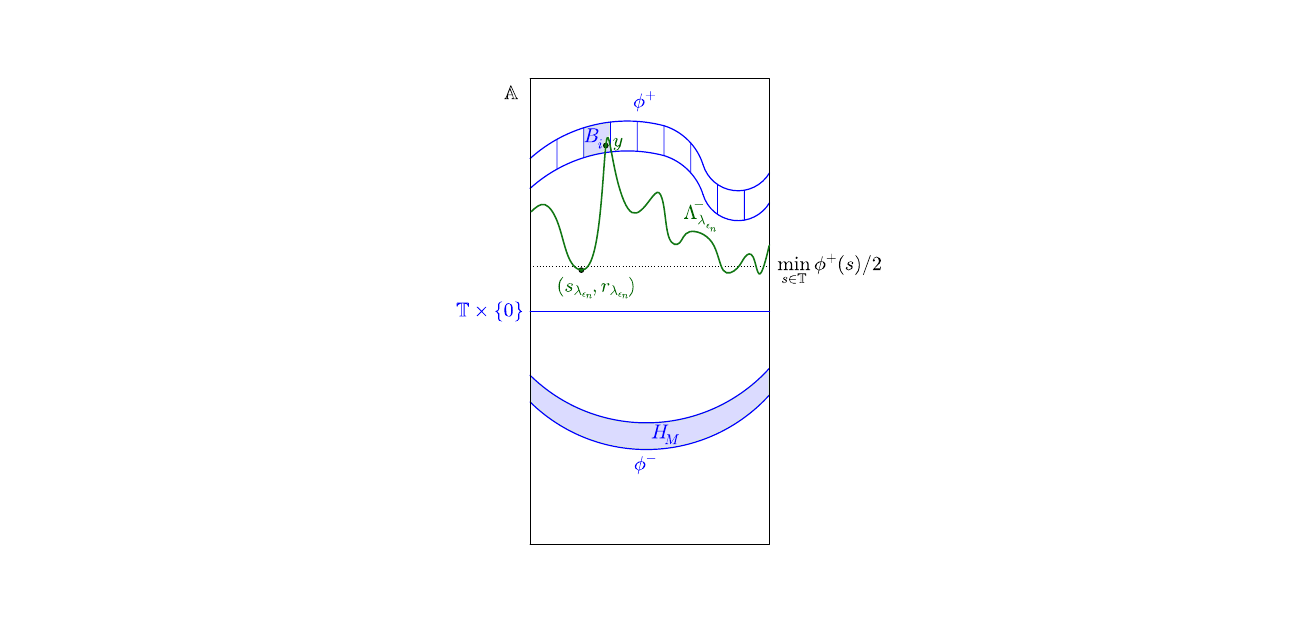}
		\caption{Controlling the shape of $\Lambda_{\lambda_{\epsilon_n}}^-$.}
		\label{endproof}
	\end{figure} 
	
	By using inclusion (\ref{inclusione}), we deduce the existence of an index $m_0\in\N$ such that
	\[
	\bigcup_{0 \le m \le m_0}f^{-m}(H_M)\cap B_i\neq\emptyset\,,
	\]
	for every $i=0,\ldots,j-1$. Therefore, since $f_{\lambda_{\epsilon_n}}$ converges uniformly to $f$ when $\epsilon_n\to 0$, there exists $n_0 \in \N$ such that for all $n\geq n_0$ and for all $i=0,\ldots,j-1$, 
	\begin{equation}\label{int Bi}
		\bigcup_{0\leq m\leq m_0}f_{\lambda_{\epsilon_n}}^{-m}(H_M)\cap B_i\neq \emptyset\,.
	\end{equation}
	Recall that $V^-(s,r):=\{(s,y)\in\A :\ y\leq r\}$ and that, for any $\lambda$, the set $C_\lambda^-$ is the connected component of $C\setminus \Lambda_\lambda$ containing $\{(s,\phi^-(s))\in\A:\ s\in\T\}$.  Let us denote by $\hat U^-_{\lambda}$ the set of points which are radially accessible from below, i.e., the points $(s,r)\in C$ such that $V^-(s,r)\cap C\subset\overline{C_\lambda^-}$. Clearly, $H_M\subset \hat U_{\lambda}^-$ for any $\lambda$. \\
	For any $n\geq n_0$ and any $i\in\{0,\dots,j-1\}$, by (\ref{int Bi}), there exist $m =m(n,i)\in \{0,\dots,m_0\}$ and a point $y=y(n,i,m)\in f^{-m}_{\lambda_{\epsilon_n}}(H_M)$ such that 
	$$y \in f^{-m}_{\lambda_{\epsilon_n}}(H_M) \cap B_i\,.$$
	This means that
	$$x := f^m_{\lambda_{\epsilon_n}}(y) \in H_M \subset \hat{U}^-_{\lambda_{\epsilon_n}} \implies x \in f^m_{\lambda_{\epsilon_n}}(C) \cap \hat{U}^-_{\lambda_{\epsilon_n}},$$
	and --by using Proposition \ref{prop technique}\ref{punt b}-- we get $y=f^{-m}_{\lambda_{\epsilon_n}}(x) \in \hat U^-_{\lambda_{\epsilon_n}}$. Since $y \in B_i$, it holds
	\begin{equation} \label{quasi fine}
		U^-_{\lambda_{\epsilon_n}} \cap B_i \ne \emptyset	\implies \Lambda_{\lambda_{\epsilon_n}}^-\cap B_i\neq \emptyset\,.
	\end{equation}
	In order to conclude the proof, fix $n\geq n_0$. By \eqref{point ordonn neg}, there exists a point $(s_{\lambda_{\epsilon_n}},r_{\lambda_{\epsilon_n}})\in \Lambda_{\lambda_{\epsilon_n}}^-$ such that $r_{\lambda_{\epsilon_n}}\leq \min_{s\in\T}\phi^+(s)/2$. Let us denote by $(\tilde s_{\lambda_{\epsilon_n}},r_{\lambda_{\epsilon_n}})\in \tilde \Lambda_{\lambda_{\epsilon_n}}^-$ a lift of $(s_{\lambda_{\epsilon_n}},r_{\lambda_{\epsilon_n}})$. By (\ref{quasi fine}), we can also find a point $(\tilde s'_{\lambda_{\epsilon_n}}, r'_{\lambda_{\epsilon_n}})\in\tilde \Lambda_{\lambda_{\epsilon_n}}^-$ such that
	\[
	r'_{\lambda_{\epsilon_n}}\geq \phi^+(s'_{\lambda_{\epsilon_n}})-\frac{2}{j}\qquad\text{and}\qquad \tilde s_{\lambda_{\epsilon_n}}<\tilde s'_{\lambda_{\epsilon_n}}<\tilde s_{\lambda_{\epsilon_n}}+\frac{1}{j}\,.
	\]
	Thus
	\[
	r'_{\lambda_{\epsilon_n}}-r_{\lambda_{\epsilon_n}}\geq \frac E 2-\frac{2}{j}-\frac{L}{j}\qquad \text{and}\qquad \frac{\tilde s'_{\lambda_{\epsilon_n}}-\tilde s_{\lambda_{\epsilon_n}}}{r'_{\lambda_{\epsilon_n}}-r_{\lambda_{\epsilon_n}}}\leq \frac{1}{j\frac E 2-2-L}\leq \frac{\tan\left(\beta/2\right)}{2}\,.
	\]
	Since $r_{\lambda_{\epsilon_n}}=\tilde \mu_{\lambda_{\epsilon_n}}^-(\tilde s_{\lambda_{\epsilon_n}})$ and $r'_{\lambda_{\epsilon_n}}=\tilde \mu_{\lambda_{\epsilon_n}}^-(\tilde s'_{\lambda_{\epsilon_n}})$, this contradicts Proposition \ref{cotangente} and completes the proof. 
\end{proof}
\noindent From the twist condition on the conservative map $f$, it holds that, once we fix a lift $F$ of the map, the rotation numbers of the graphs of $\phi^+$ and $\phi^-$ are well-defined. Denoting them by $\rho^+$ and $\rho^-$ respectively, we also have that $\rho^-<\rho^+$. 
It is then sufficient to show that Claim \ref{claim 6} implies
\begin{equation}\label{Final rho}
	\lim_{\epsilon\to 0}\rho_{\lambda_\epsilon}^\pm=\rho^\pm\,.
\end{equation}
We can follow \textit{verbatim} the proof of \cite[Corollary 4.8]{LeCalvez} to deduce \eqref{Final rho} and then conclude the proof.
\end{proof}

Now, Corollary \ref{prop different rho-billiards} is a straightforward consequence of Proposition \ref{bill is twist}, Proposition \ref{prop different rho}, Remark \ref{simplify lambda const}, Remark \ref{rem prop diff}, and the following observation: with the notation of the above proof, $\frac 12\in (\rho^-,\rho^+) \mod \Z$, hence for any $C^1$ function $\lambda$ whose $C^0$-distance from the constant function $1$ is in $(\frac \epsilon 2, \epsilon)$ and such that $\Vert D\lambda\Vert<\epsilon^2$,  \eqref{Final rho} implies that $\frac 12\in (\rho_\lambda^-,\rho_\lambda^+) \mod \Z$. \\

In the sequel we show that the Birkhoff attractor of a dissipative billiard map may have different upper and lower rotation numbers, provided that the dissipation is mild enough. Moreover, we emphasize the main dynamical consequences of this fact. We start by recalling Corollary \ref{coro main theorem different nb rotation}, stated in the Introduction.

\newtheorem*{cor:intro}{\textsc{\textbf{Corollary \ref{coro main theorem different nb rotation}}}}
\begin{cor:intro} 
	Fix $k\geq 3$.  There exists an open and dense subset $\mathscr{U}$ of $C^k$ strongly convex domains such that the following holds. For any $\Omega\in \mathscr{U}$, there exists $\lambda_0(\Omega)\in (0,1)$ such that, for any $\lambda\in[\lambda_0(\Omega),1)$, the Birkhoff attractor $\Lambda_\lambda$ of the corresponding dissipative billiard map $f_\lambda$ has $\rho_\lambda^+-\rho^-_\lambda>0$, with $\frac 12 \in (\rho_\lambda^-,\rho^+_\lambda)\mod \Z$. Moreover, there exists $\lambda_1(\Omega)\in [\lambda_0(\Omega),1)$ such that for any $\lambda\in[\lambda_1(\Omega),1)$, 
	and for any $2$-periodic point $p$ of saddle type (e.g., when the $2$-periodic orbit $\{p,f_\lambda(p)\}$ corresponds to a diameter of the table), there exists a horseshoe $K_\lambda(p)\subset \Lambda_\lambda$ in the homoclinic class $H_\lambda(p):=\overline{\mathcal{W}^s(\mathcal{O}_{f_\lambda}(p))\pitchfork \mathcal{W}^u(\mathcal{O}_{f_\lambda}(p))}$ of $p$; more precisely, it holds 
	$$
	K_\lambda(p)\subset H_\lambda(p):=\overline{\mathcal{W}^s(\mathcal{O}_{f_\lambda}(p))\pitchfork \mathcal{W}^u(\mathcal{O}_{f_\lambda}(p))}\subset \overline{
		\mathcal{W}^u(\mathcal{O}_{f_\lambda}(p))}\subset \Lambda_\lambda.
	$$ 
\end{cor:intro}


\noindent The proof of Corollary \ref{coro main theorem different nb rotation} relies on the following proposition. 

\color{black}

\begin{proposition}\label{generic instability region}
	For each $k \geq 2$, there exists an open and dense subset $\mathscr{U}$ of $C^k$ strongly convex domains such that for any $\Omega\in \mathscr{U}$, the corresponding billiard map has an instability region $\mathscr{I}\subset \mathbb{A}:= \mathbb{T} \times [-1,1]$ that contains (a neighborhood of) the zero-section $\mathbb{T} \times \{0\}$. 
\end{proposition}

\begin{proof}
	The argument follows the work \cite{PintoC} of Dias Carneiro, Oliffson Kamphorst and  Pinto-de-Carvalho. 
	For a convex domain $\Omega$, let $f=f_1\colon \mathbb{A} \to \mathbb{A}$ be the associated (conservative) billiard map. 
	We denote by $I \colon (s,r)\mapsto (s,-r)$ the time-reversal involution; recall that $f \circ I =I \circ f^{-1}$. 
	Let $\Gamma\subset\A$ be a $f$-invariant essential curve. In particular, by Birkhoff's Theorem (see \cite{Bir3}), there exists a Lipschitz function $\phi:\T\to\R$ such that $\Gamma=\{(s,\phi(s))\in\A:\ s\in\T\}$.
	The symmetric graph $I(\Gamma)=\{(s,-\phi(s))\in\A:\ s\in \mathbb{T}\}$ is also $f$-invariant, as $f(I(\Gamma))=I(f^{-1}(\Gamma))=I(\Gamma)$. Moreover, we observe that $\Gamma\cap I(\Gamma)\subset \mathbb{T} \times \{0\}$. In particular, any point in $\Gamma\cap I(\Gamma)$ is a $2$-periodic point. Indeed, the intersection is also $f$-invariant: given any $x_0=(s_0,0)\in\Gamma\cap I(\Gamma)$, then also $f(x_0)=(s_1,0)\in\Gamma\cap I(\Gamma)$. Thus, the bounces at $x_0$ and $f(x_0)$ are perpendicular, hence $\{x_0,f(x_0)\}$ is a $2$-periodic orbit. We conclude that the rotation number of $\Gamma$ (and so of $I(\Gamma)$) is equal to $\frac 12$. 
	
	
	Now, by \cite[Section 3]{PintoC}, given a rational number $\frac{p}{q}\in \mathbb{Q}$, there exists an open and dense subset $\mathscr{U}_{\frac p q}$ of $C^k$ strongly convex billiards which have no rotational invariant curve with rotation number $\frac pq$. Let us briefly recall the argument. If $\gamma$ is such a curve, then the restriction $f|_\gamma$ of the billiard map to $\gamma$ is a homeomorphism of the circle, and since the rotation number is equal to $\frac pq$, there are periodic points on $\gamma$ of type $(p,q)$. But these cannot be linearly elliptic, since the curve is a Lipschitz graph over $\T$, by Birkhoff's Theorem. Then, these periodic points are either degenerate, as in the circular billiard, or hyperbolic, in which case, $\gamma$ will be a union of periodic points and saddle connections. By \cite[Theorem 1 and Theorem 2]{PintoC} (see also \cite{CarneiroOliffsonPC,XiaZhang}), both cases are not allowed for a strongly convex billiard $\Omega$ in an open and dense subset of domains. 
	
	It follows from the previous discussion that there exists an open and dense subset $\mathscr{U}=\mathscr{U}_{\frac 12}$ of $C^k$ strongly convex domains such that any $\Omega\in \mathscr{U}$ has no invariant essential curve crossing the zero-section $\mathbb{T} \times \{0\}$, and thus, has an instability region containing a neighborhood of the zero-section $\mathbb{T} \times \{0\}$. 
\end{proof}

\begin{proof}[Proof of Corollary \ref{coro main theorem different nb rotation}]
	Fix $k \geq 3$. 
	As an immediate outcome of Corollary \ref{prop different rho-billiards} and Proposition \ref{generic instability region}, there exists an open and dense subset $\mathscr{U}$ of $C^k$ strongly convex domains such that for any $\Omega\in \mathscr{U}$, there exists $\lambda_0(\Omega)\in (0,1)$ such that, for any $\lambda\in[\lambda_0(\Omega),1)$, the Birkhoff attractor $\Lambda_\lambda$ of the corresponding dissipative billiard map $f_\lambda$ has $\rho_\lambda^+-\rho^-_\lambda>0$, with $\frac 12 \in (\rho_\lambda^-,\rho^+_\lambda)\mod \Z$. Let us denote by $\mathrm{II}_{\text{max}}$ the set of $2$-periodic points for $f_1$ with locally maximal perimeter. As noted at the beginning of Section \ref{conti 2 orbite}, for any $p\in \mathrm{II}_{\text{max}}$, $\{p,f_\lambda(p)\}$ is still a $2$-periodic orbit for each dissipative map $f_\lambda$, $\lambda\in(0,1)$, and by Proposition \ref{remark point crit func}\ref{premier ppooiinntt}, it is actually of saddle type, for any $\Omega$ in an open and dense subset of  $C^k$ domains. Moreover, for $\lambda\in[\lambda_0(\Omega),1)$, $\frac 12 \in (\rho_\lambda^-,\rho^+_\lambda)\mod \Z$, hence $\mathrm{II}_{\text{max}}\subset \Lambda_\lambda$, by \cite[Proposition 14.2]{LeCalvez}. 
	By Corollary \ref{nondeg billiard in dk}\eqref{ppppppoint trois}, 
	there exists  $\lambda_1(\Omega)\in [\lambda_0(\Omega),1)$ such that, for any $p \in \mathrm{II}_{\text{max}}$, and for any $\lambda\in[\lambda_1(\Omega),1)$, each branch of $\mathcal{W}^s(p;f_\lambda^2)\setminus \{p\}$ and $\mathcal{W}^u(p;f_\lambda^2)\setminus \{p\}$ contains a transverse homoclinic point. Now, Corollary \ref{coro fer a chev} implies that for any $\lambda\in[\lambda_1(\Omega),1)$, the Birkhoff attractor $\Lambda_\lambda$ of $f_\lambda$ contains a horseshoe $K_\lambda(p)$, with
	$$
	K_\lambda(p)\subset H_\lambda(p):=\overline{\mathcal{W}^s(\mathcal{O}_{f_\lambda}(p))\pitchfork \mathcal{W}^u(\mathcal{O}_{f_\lambda}(p))}\subset \overline{
		\mathcal{W}^u(\mathcal{O}_{f_\lambda}(p))}\subset \Lambda_\lambda.\qedhere
	$$ 
\end{proof}
\noindent 

We can guarantee that the upper and lower rotation numbers of the Birkhoff attractor are different also in the case of or every $C^2$-convex domain with a point with vanishing curvature, as explained in the sequel.

\begin{coro}\label{coro mather}
	Let $\Omega$ be a convex domain with $C^2$ boundary such there is a point at which the curvature vanishes. Then for any $\epsilon>0$, there exists $\lambda_0=\lambda_0(\Omega,\epsilon)\in (0,1)$ such that for any $\lambda\in[\lambda_0,1)$, the Birkhoff attractor of $f_\lambda$ has $\rho^+_\lambda-\rho^-_\lambda\in (1-\epsilon,1)$.  
\end{coro}

Corollary \ref{coro mather} is a consequence of Proposition \ref{bill is twist}, Proposition \ref{prop different rho}, Remark \ref{simplify lambda const} and the next well-known result by Mather (see \cite{Mather82} and also \cite[Corollary 5.29]{TabBook}-\cite[Theorem 1.1]{GutkinKatok}).
\begin{theorem}\label{theo mather}
	If the curvature of a $C^{2}$-convex billiard curve vanishes at some point, then the associated conservative billiard map has no invariant essential curves.
\end{theorem}

\begin{proof}[Proof of Corollary \ref{coro mather}]
	Let $\Omega$ be a convex domain with $C^2$ boundary. Then, the associated billiard map $f\colon\A\to\A$ is an orientation-preserving homeomorphism, homotopic to the identity, it preserves the standard $2$-form $\omega=dr\wedge ds$, and the restriction of $f$ to $\mathrm{int}(\A)$ is a $C^1$ diffeomorphism, see \cite{LeCalvez1990} and also \cite{Douady} for all details. By Proposition \ref{bill is twist}, it is a positive twist map on $\mathrm{int}(\A)$. Since there exists a point of zero curvature, by Theorem \ref{theo mather}, the whole $\mathrm{int}(\A)$ is an instability region. 
	We conclude the proof by applying Proposition \ref{prop different rho}. 
\end{proof}

Let us conclude this section by the following remark, which provides a different proof of Proposition \ref{prop different rho}. We are grateful to Patrice Le Calvez for suggesting this argument. 

\begin{proof}[Alternative proof of Proposition \ref{prop different rho}]
	We start by claiming that, as $\lambda$ tends to $1$ (in the $C^0$ topology), the Birkhoff attractor comes closer and closer to both $\Gamma_{\phi^+}$ and $\Gamma_{\phi^-}$. Recall that $\Gamma_{\phi^\pm}:=\{(s,\phi^\pm(s))\in\A:\ s\in\T\}$ are the graphs corresponding to the boundary of the instability region $\mathscr{I}$. For simplicity, we assume in the rest of the proof that $\lambda$ is merely a constant function. 
	
	Let us assume by contradiction that 
	there exists a sequence $(\lambda_n)_{n}\in (0,1)^\mathbb{N}$ with $\lim_{n \to +\infty} \lambda_n=1$, such that for some $\varepsilon>0$, we have
	$d_H(\Lambda_{\lambda_n},\Gamma_{\phi^+})\geq \varepsilon$ or $d_H(\Lambda_{\lambda_n},\Gamma_{\phi^-})\geq \varepsilon$, for all $n \geq 0$, where $d_H$ is the Hausdorff distance. Without loss of generality, we assume that we are in the second case. The other one can be treated analogously. 
	
	Since the set of compact subsets of $\A$ is compact for the Hausdorff distance $d_H$, up to passing to a subsequence, we can assume that there exists a compact subset $\Lambda_\infty\subset \overline{\mathscr{I}}=C$ such that $\lim_{n \to +\infty} d_H(\Lambda_{\lambda_n},\Lambda_\infty)=0$. Moreover, by our assumption, we have $d_H(\Lambda_\infty,\Gamma_{\phi^-})\geq\frac \varepsilon 2$. 
	\begin{claim}\label{first claim patrice}
		The set $\Lambda_\infty$ 
		has the following properties:
		\begin{enumerate}[label=(\roman*)]
			\item\label{zero cll} $\Lambda_\infty$ is connected;  
			\item\label{i clllima} $\Lambda_\infty$ is $f$-invariant;
			\item\label{ii clllima} there is no continuous path $\gamma\colon [0,1]\to C\setminus \Lambda_\infty$ such that $\gamma(0)\in \Gamma_{\phi^-}$ and $\gamma(1)\in \Gamma_{\phi^+}$. 
		\end{enumerate}
		\begin{proof}[Proof of the claim]
		Point \ref{zero cll} follows from the fact that $\A$ is compact, and that each $\Lambda_{\lambda_n}$ is connected. 
		
		Let us show \ref{i clllima}. 
		Let $x_\infty \in \Lambda_\infty$. We want to show that $f(x_\infty)\in \Lambda_\infty$; for that, it is enough to show that $d_H(\{f(x_\infty)\},\Lambda_\infty)=0$. As $x_\infty \in \Lambda_\infty$, there exists a sequence $(x_n)_n$ such that $x_n \in \Lambda_{\lambda_n}$ converging to $x_\infty$. Since $(\lambda,x)\mapsto f_\lambda(x)$ is continuous, we thus have 
		\begin{align*}
			d_H(\{f(x_\infty)\},\Lambda_\infty)&=\lim_{n \to +\infty}d_H(\{f_{\lambda_n}(x_n)\},\Lambda_{\infty})\\
			&\leq \lim_{n \to +\infty}d_H(\{f_{\lambda_n}(x_n)\},\Lambda_{\lambda_n})+d_H(\Lambda_{\lambda_n},\Lambda_\infty)=0\, ,
		\end{align*}
		as $d_H(\{f_{\lambda_n}(x_n)\},\Lambda_{\lambda_n})=0$ for each $n \geq 0$, since  $\Lambda_{\lambda_n}$ is $f_{\lambda_n}$-invariant. 
		
		Let us now turn to the proof of \ref{ii clllima}. Assume by contradiction that there exists a continuous path $\gamma\colon [0,1]\to C\setminus \Lambda_\infty$ such that $\gamma(0)\in \Gamma_{\phi^-}$ and $\gamma(1)\in \Gamma_{\phi^+}$. Since $\gamma([0,1])$ and $\Lambda_\infty$ are compact, we have  $\delta:=d_H(\gamma([0,1]),\Lambda_\infty)>0$. By the definition of $\Lambda_\infty$, we deduce that there exists $n_0 \in \mathbb{N}$ such that for $n \geq n_0$, it holds that $d_H(\gamma[0,1],\Lambda_{\lambda_n})\geq \frac \delta 2>0$. In particular, $\Gamma_{\phi^-}$ and $\Gamma_{\phi^+}$ belong to the same connected component of the open set $\mathscr{I}\setminus \Lambda_{\lambda_n}$, which contradicts the fact that $\Lambda_{\lambda_n}$ separates $\overline{\mathscr{I}}=C$. 
		\end{proof}
	\end{claim}
	Since $d_H(\Lambda_\infty,\Gamma_{\phi^-})\geq \frac \varepsilon 2$, we fix $\delta>0$ small such that 
	$$
	H_\delta:=\big\{(s,r):\phi^-(s)\leq r\leq  \phi^-(s)+\delta\big\}\subset \overline{\mathscr{I}} \setminus \Lambda_\infty\,.
	$$
	Let us consider the set $V:=\bigcup_{k \in \mathbb{N}} f^k(H_\delta)$. Still by \cite[Section 6]{Birkhoff1} (see also \cite[Proposition 5.9.2]{Herman}), we have
		$\Gamma_{\phi^+}\subset \overline{V}$. Observe that the set $V$ is connected, as $\Gamma_{\phi^-}$ is $f$-invariant. Moreover, by the $f$-invariance of $\overline{\mathscr{I}} \setminus \Lambda_\infty$, as $\Lambda_\infty \subset \overline{\mathscr{I}}$ lies above $H_\delta$, and by point \ref{ii clllima} of the above claim, we deduce that $V$ belongs to the same connected component of $\overline{\mathscr{I}} \setminus \Lambda_\infty$ as $H_\delta$, which lies below $\Lambda_\infty$. The inclusion $\Gamma_{\phi^+}\subset \overline{V}$ thus gives
		$$
		\Lambda_\infty=\Gamma_{\phi^+}\,.
		$$
\noindent Then, for any neighborhood $U$ of $\Gamma_{\phi^+}$ there exists $n_0\in\N$ such that for every $n\geq n_0$ the compact set $\Lambda_{\lambda_n}$ is contained in $U$. However, following the first part of the proof of Proposition \ref{prop different rho} presented above, for every $\lambda\in(0,1)$ there exists a point $x_\lambda=(s_\lambda,r_\lambda)\in\Lambda_\lambda$ such that $r_\lambda\leq \min_{s\in\T}\frac{\phi^+(s)}{2}$, see \eqref{point ordonn neg}. Choosing then the neighborhod $U$ of $\Gamma_{\phi^+}$ small enough, we obtain the required contradiction.  
	
	So far, we have shown that for any $\eta>0$, there exists $\bar\lambda=\bar \lambda(\eta)\in (0,1)$ such that for any $\lambda \in [\bar \lambda,1)$, we have 
	$$
	d_H(\Lambda_\lambda,\Gamma_{\phi^-})\leq \eta\quad\text{and}\quad d_H(\Lambda_\lambda,\Gamma_{\phi^+})\leq \eta\, .
	$$

	 From this, again following the argument of Corollary 4.8 in \cite{LeCalvez}, we deduce that $\lim_{\lambda\to 1}\rho_\lambda^\pm=\rho^\pm$, and thus, for $\lambda$ close enough to $1$, it holds $\rho_\lambda^+-\rho_\lambda^->0$, since $\rho^+-\rho^->0$.  
\end{proof}

\appendix
\section{Proof of Lemma \ref{lemme vp reelles} and Lemma \ref{lemma noncst disp}: bifurcation of eigenvalues at $2$-periodic points for dissipative billiard maps}\label{appendix proof bifurcation eigenvalues}
 
Let us recall that $\mathrm{II}$ denotes the set of $2$-periodic points for the conservative billiard map $f=f_1$. For $p=(s,0) \in \mathrm{II}$, we denote by $\tau=\ell(s,s'):=\|\Upsilon(s)-\Upsilon(s')\|$ the Euclidean distance between the points $\Upsilon(s),\Upsilon(s')$, where $(s',0):=f(p)$.  We also denote by $\mathcal{K}_1$, $\mathcal{K}_2$ the respective curvatures at $\Upsilon(s),\Upsilon(s')$.

Let us fix a $C^{k-1}$ function $\lambda\colon \A\to (0,1)$ such that  $f_\lambda:=\mathcal{H}_\lambda \circ f$ is a dissipative billiard map in the sense of Definition \ref{cond lambda dis}, where  $\mathcal{H}_\lambda\colon (s,r)\mapsto(s,\lambda(s,r) r)$. In particular, $f_\lambda$ has the same set $\mathrm{II}$ of $2$-periodic points as $f$, and for any $2$-periodic orbit $\{p=(s,0),f_1(p)=f_\lambda(p)=(s',0)\}$, we have
\begin{equation*}
	D\mathcal{H}_\lambda(p)=\begin{bmatrix}
		1 & 0\\
		0 & \lambda_1
	\end{bmatrix},\quad D\mathcal{H}_\lambda(f_\lambda(p))=\begin{bmatrix}
	1 & 0\\
	0 & \lambda_2
	\end{bmatrix},
\end{equation*} 
with $\lambda_1:=\lambda(p)$, and $\lambda_2:=\lambda(f_\lambda(p))$. From the formula of $Df$ given in \cite[Section 2.11]{CheMar_book}, we deduce that 
\begin{align*}
Df_\lambda(p)=D\mathcal{H}_\lambda(f_\lambda(p))Df(p)&=\begin{bmatrix}
	-(\tau \mathcal{K}_1 + 1) & \tau \\
	\frac{\lambda_2}{\tau} (k_{1,2}-1) &- \lambda_2(\tau \mathcal{K}_2 + 1)
\end{bmatrix},\\
Df_\lambda(f_\lambda(p))=D\mathcal{H}_\lambda(p)Df(f_\lambda(p))&=\begin{bmatrix}
	-(\tau \mathcal{K}_2+ 1) & \tau \\
	\frac{\lambda_1}{\tau} (k_{1,2}-1)&- \lambda_1(\tau \mathcal{K}_1 + 1)
\end{bmatrix},
\end{align*}
where we have set 
$$k_{1,2}:=(\tau \mathcal{K}_1 + 1)(\tau \mathcal{K}_2 + 1)\,.$$
Observe that $\det Df_\lambda(p)=D\mathcal{H}_\lambda(f_\lambda(p))=\lambda_2$, as $f$ is conservative; similarly, $\det Df_\lambda(f_\lambda(p))=\lambda_1$. 
We thus obtain
\begin{equation*} 
	Df_\lambda^2(p)=\begin{bmatrix}
		k_{1,2}(1+\lambda_2)-\lambda_2 & * \\
		* & k_{1,2}\lambda_1(1+\lambda_2)-\lambda_1
	\end{bmatrix}. 
\end{equation*}
In particular,  
\begin{equation}\label{trace d f lambda carre bis}
	\det Df_\lambda^2(p)=\lambda_1\lambda_2,\quad \mathrm{tr}Df_\lambda^2(p)=(1+\lambda_1)(1+\lambda_2)k_{1,2}-(\lambda_1+\lambda_2)\,.
\end{equation}  
\begin{proof}[Proof of Lemma \ref{lemme vp reelles}]
		We consider the case where the dissipation is constant, equal to some $\lambda \in (0,1)$. Let us denote by $\mu_1=\mu_1(\lambda),\mu_2=\mu_2(\lambda)$ the eigenvalues of $Df_\lambda^2(p)$, with $|\mu_1|\leq |\mu_2|$. In particular, with the above notations, we have $\lambda_1=\lambda_2=\lambda$, and 
		\begin{equation}\label{trace d f lambda carre}
			\mu_1\mu_2=\det Df_\lambda^2(p)=\lambda^2,\quad \mathrm{tr}Df_\lambda^2(p)=(1+\lambda)^2k_{1,2}-2\lambda\,.
		\end{equation} 
	By \eqref{trace d f lambda carre}, the characteristic polynomial $\chi_{p,\lambda}(x)=\det (Df_\lambda^2(p)-\mathrm{id})$ is equal to
	$$
	\chi_{p,\lambda}(x)=x^2 -\big((1+\lambda)^2k_{1,2}-2\lambda\big) x + \lambda^2\,,
	$$
	with $k_{1,2}=(1+\tau \mathcal{K}_1)(1+\tau \mathcal{K}_2)$. Then, $\chi_{p,\lambda}$ has discriminant 
	$\Delta=\big((1+\lambda)^2 k_{1,2}-2\lambda\big)^2-4 \lambda^2=k_{1,2}(1+\lambda)^2\big((1+\lambda)^2k_{1,2}-4\lambda\big)$, which has the same sign as 
	$$
	\tilde\Delta:=k_{1,2}\left((1+\lambda)^2k_{1,2}-4\lambda\right)=\lambda^2k_{1,2}^2 + 2\lambda k_{1,2}(k_{1,2}-2)+k_{1,2}^2\,.
	$$
	The quantity $\tilde\Delta$ is a quadratic polynomial in $\lambda$, whose discriminant is equal to 
	$$
	\delta=4k_{1,2}^2\big((k_{1,2}-2)^2-k_{1,2}^2\big)
	=-16k_{1,2}^2(k_{1,2}-1)\,.
	$$
	\noindent $(a)$ If $k_{1,2}>1$, then $\delta<0$, hence $\tilde \Delta> 0$, $\Delta> 0$, and the eigenvalues $\mu_1,\mu_2$ of $Df_\lambda^2(p)$ are real, with $|\mu_1|\leq |\mu_2|$. 
	Their product $\mu_1 \mu_2=\mathrm{det}Df_\lambda^2(p)=\lambda^2$ is positive; their sum $\mu_1+\mu_2$ is also positive because
	\begin{equation}\label{trace carre diff}
		\mu_1+\mu_2=\mathrm{tr}Df_\lambda^2(p)=(1+\lambda)^2k_{1,2}-2\lambda> (1+\lambda)^2-2\lambda=1+\lambda^2>0\, ,
	\end{equation}
	where the first inequality comes from the hypothesis $k_{1,2}>1$. Therefore, both eigenvalues are positive, and $0<\mu_1\leq \mu_2$. In particular, by \eqref{trace d f lambda carre}, $0<\mu_1^2\leq \mu_1\mu_2=\lambda^2$ hence $\mu_1\in (0,1)$. Let us show that in fact, $0<\mu_1\leq \lambda^2<1<\mu_2$. Indeed, for $i=1,2$, by \eqref{trace d f lambda carre}-\eqref{trace carre diff}, we have  $\mu_i+\frac{\lambda^2}{\mu_i}>1+\lambda^2$, hence $P(\mu_i)>0$, where $P(X)=X^2-(1+\lambda^2)X+\lambda^2$. Since the roots of $P$ are $\{\lambda^2,1\}$, and as $\mu_1\in (0,1)$, we deduce that $\mu_1\in (0,\lambda^2)$, and then $\mu_2=\frac{\lambda^2}{\mu_1}>1$. 
	\\
	~\newline
	\noindent $(b)$ If $k_{1,2}=1$, then $\chi_{p,\lambda}(x)=x^2 -(1+\lambda^2) x + \lambda^2=(x-\lambda^2)(x-1)$, thus $\mu_1=\lambda^2$, $\mu_2=1$. \\
	
	\noindent If $k_{1,2}<1$, $k_{1,2}\neq 0$, then $\delta> 0$.  Let $\lambda_\pm:=-1+\frac{2}{k_{1,2}}\pm 2\sqrt{\frac{1-k_{1,2}}{k_{1,2}^2}}$.
	\\
	~\newline
	\noindent $(c)$ Assume now that $k_{1,2}\in (0,1)$. Note that $\lambda_+=-1+\frac{2}{k_{1,2}}(1+\sqrt{1-k_{1,2}})\geq -1+\frac{2}{k_{1,2}}> 1$, and $\lambda_-=-1+\frac{2}{k_{1,2}}(1-\sqrt{1-k_{1,2}})=\frac{\xi^2-2\xi+1}{1-\xi^2}=\frac{1-\xi}{1+\xi}$, with $\xi:=\sqrt {1-k_{1,2}}\in (0,1)$, so that $\lambda_-\in (0,1)$. 
	Then, 
	\begin{enumerate}
		\item\label{premier cassss} for $\lambda \in (\lambda_-,1)$, $\tilde \Delta<0$, $\Delta<0$, hence the eigenvalues of $Df_\lambda^2(p)$ are complex conjugate;
		\item\label{deuxieme cassss} for $\lambda\in (0,\lambda_-]$, $\tilde \Delta\geq 0$, $\Delta\geq 0$, hence the eigenvalues of $Df_\lambda^2(p)$ are real. Moreover, $\mu_1=\mu_2$ if and only if $\lambda=\lambda_-$.
	\end{enumerate} 
	
	In case \eqref{premier cassss}, 
	by \eqref{trace d f lambda carre}, it holds $|\mu_1|=|\mu_2|=\sqrt{\mu_1\mu_2}=\sqrt{\lambda^2}=\lambda\in (0,1)$, hence the $2$-periodic orbit $\{p,f_\lambda(p)\}$ is a sink.
	
	Let us consider case \eqref{deuxieme cassss}. By \eqref{trace d f lambda carre}, $\mu_1,\mu_2$ have the same sign, and $0<\mu_1^2\leq \mu_1\mu_2=\lambda^2$ hence $|\mu_1|\in (0,1)$. Moreover, as $k_{1,2}\in(0,1)$, and by  \eqref{trace d f lambda carre}, for $i=1,2$, we have 
	\begin{equation}\label{aux inq} 
		\mu_i+\frac{\lambda^2}{\mu_i}=\mathrm{tr}Df_\lambda^2(p)=(1+\lambda)^2k_{1,2}-2\lambda \in (-2\lambda,1+\lambda^2)\,.
	\end{equation}
	
	Assume that $\mu_1,\mu_2$ are negative. By \eqref{aux inq}, for $i=1,2$, it holds $\mu_i+\frac{\lambda^2}{\mu_i}>-2\lambda$,  hence $0>\mu_i^2+2\lambda \mu_i+\lambda^2=(\mu_i+\lambda)^2$, a contradiction. 
	Thus, $\mu_1,\mu_2$ are positive, and then, by \eqref{aux inq}, for $i=1,2$, it holds $P(\mu_i)<0$, where $P(X)=X^2-(1+\lambda^2)X+\lambda^2$. Since the roots of $P$ are $\{\lambda^2,1\}$, we deduce that $\mu_1,\mu_2\in (\lambda^2,1)$, and then the $2$-periodic orbit $\{p,f_\lambda(p)\}$ is a sink as well. 
	\\
	~\newline
	\noindent $(d)$ If $k_{1,2}=0$, then $\chi_{p,\lambda}(x)=x^2 +2\lambda x + \lambda^2=(x+\lambda)^2$, hence $\mu_1=\mu_2=-\lambda$. \\
	~\newline
	\noindent $(e)-(f)$ Finally, assume that $k_{1,2}<0$. In that case, it is easy to check that $\lambda_\pm=-1+\frac{2}{k_{1,2}}\big( 1\mp\sqrt{1-k_{1,2}} \big)<0$, and then, for $\lambda \in (0,1)$, $\Delta>0$, hence the eigenvalues of $Df_\lambda^2(p)$ are real. By \eqref{trace d f lambda carre}, $\mu_1,\mu_2$ have the same sign, and $0<\mu_1^2\leq \mu_1\mu_2=\lambda^2$, thus $|\mu_1|\in (0,\lambda)$. Moreover, as $k_{1,2}<0$, and by  \eqref{trace d f lambda carre}, for $i=1,2$, we have 
	$
	\mu_i+\frac{\lambda^2}{\mu_i}=\mathrm{tr}Df_\lambda^2(p)<-2\lambda$, hence $\mu_2<-\lambda<\mu_1<0$. We have
	$$
	\mu_2=\mu_2(\lambda)=\frac{1}{2}\Big((1+\lambda)^2 k_{1,2}-2\lambda-\sqrt{(1+\lambda)^2k_{1,2}\big((1+\lambda)^2 k_{1,2}-4 \lambda\big)}\Big)\,.
	$$
	Observe that 
	\begin{equation*}
		\lim_{\lambda \to 0^+} \mu_2(\lambda)=k_{1,2},\qquad \lim_{\lambda \to 1^-} \mu_2(\lambda)=-1+2k_{1,2}-2\sqrt{k_{1,2}\big(k_{1,2}- 1\big)}<-1\,.
	\end{equation*}
	By direct computation, we see that the equation $\mu_2(\lambda)=-1$ admits a solution in $(0,1)$ if and only if $k_{1,2}\in (-1,0)$; in that case, the only solution in $(0,1)$ is $\lambda=\bar \lambda$, with
	$$
	\bar \lambda=\bar \lambda(p):=\frac{1-\sqrt{-k_{1,2}}}{1+\sqrt{-k_{1,2}}}\in (0,1)\,.
	$$
	We conclude that  
	\begin{enumerate}
		\item[$(e)$] if $k_{1,2}\in (-1,0)$, then $\bar \lambda \in (0,1)$, and 
		\begin{enumerate}[label=(\roman*)]
			\item for any $\lambda \in (0,\bar \lambda)$, $-1<\mu_2<-\lambda<\mu_1<0$, and the $2$-periodic orbit $\{p,f_\lambda(p)\}$ is a sink;
			\item for $\lambda=\bar \lambda$, $\mu_1=-\lambda^2$, $\mu_2=-1$, and the $2$-periodic orbit $\{p,f_\lambda(p)\}$ is parabolic;
			\item for any $\lambda \in (\bar \lambda,1)$, $\mu_2<-1<-\lambda^2<\mu_1<0$, and the $2$-periodic orbit $\{p,f_\lambda(p)\}$ is a saddle;
		\end{enumerate} 
		\item[$(f)$] if $k_{1,2}\leq -1$, then for any $\lambda \in (0,1)$, $\mu_2<-1<-\lambda^2<\mu_1<0$, and the $2$-periodic orbit $\{p,f_\lambda(p)\}$ is a saddle.\qedhere
	\end{enumerate}
\end{proof}
 
\begin{proof}[Proof of Lemma \ref{lemma noncst disp}]
We now consider the case where $\lambda\colon \A\to (0,1)$ is a general $C^{k-1}$ function such that $f_\lambda:=\mathcal{H}_\lambda \circ f$ is a dissipative billiard map in the sense of Definition \ref{cond lambda dis}, where  $\mathcal{H}_\lambda\colon (s,r)\mapsto(s,\lambda(s,r) r)$.

Fix a $2$-periodic orbit $\{p,f_\lambda(p)\}$ for $f_\lambda$. Let us denote by $\mu_1=\mu_1(\lambda),\mu_2=\mu_2(\lambda)$ the eigenvalues of $Df_\lambda^2(p)$, with $|\mu_1|\leq |\mu_2|$. By \eqref{trace d f lambda carre bis}, the characteristic polynomial  $\chi_{p,\lambda}(x)=\det (Df_\lambda^2(p)-\mathrm{id})$ is equal to
$$
\chi_{p,\lambda}(x)=x^2 -\big((1+\lambda_1)(1+\lambda_2)k_{1,2}-(\lambda_1+\lambda_2)\big) x + \lambda_1\lambda_2\,,
$$
with $k_{1,2}=(1+\tau \mathcal{K}_1)(1+\tau \mathcal{K}_2)$. Recall that we assume that $k_{1,2}\geq 0$. 

On the one hand, if the eigenvalues $\mu_1,\mu_2$ are not real, then they are complex conjugate, and as $\mu_1\mu_2=\lambda_1\lambda_2\in (0,1)$,  their modulus is strictly less than $1$, and $\{p,f_\lambda(p)\}$ is a sink.  

On the other hand, if $\mu_1,\mu_2\in \R$, then as $|\mu_1|\leq |\mu_2|$, and $\mu_1\mu_2=\lambda_1\lambda_2\in (0,1)$, we deduce that $|\mu_1|<1$. Thus, the $2$-periodic orbit $\{p,f_\lambda(p)\}$ is a saddle or a sink, unless $\mu_2=1$ or $\mu_2=-1$. 
But 
\begin{equation}\label{eq saddddd}
\chi_{p,\lambda}(1)=1 -(1+\lambda_1)(1+\lambda_2)k_{1,2}+(\lambda_1+\lambda_2) + \lambda_1\lambda_2=(1+\lambda_1)(1+\lambda_2)(1-k_{1,2})\,,
\end{equation}
with $(1+\lambda_1)(1+\lambda_2)>0$, hence $\chi_{p,\lambda}(1)=0$ if and only if $k_{1,2}=1$. In that case, we have
$$
\chi_{p,\lambda}(x)=(x-1)(x-\lambda_1\lambda_2)\,,
$$
hence $1$ is an eigenvalue no matter which $\lambda$ we choose. In particular, $\{p,f_\lambda(p)\}$ is parabolic for the conservative billiard map $f_1$. 

Moreover, as $k_{1,2}\geq 0$, we have
$$
\chi_{p,\lambda}(-1)=1 +(1+\lambda_1)(1+\lambda_2)k_{1,2}-(\lambda_1+\lambda_2) + \lambda_1\lambda_2\geq(1-\lambda_1)(1-\lambda_2)>0\,,
$$
thus $-1$ is never an eigenvalue.

To conclude the proof, it remains to show that under the assumption that there is no parabolic $2$-periodic orbit, then for any $2$-periodic point $p$, for the dissipative billiard map $f_\lambda$,  the point $p$ is a saddle if and only if $k_{1,2}>1$, and a sink if and only if $k_{1,2}<1$.

Indeed, on the one hand, if $k_{1,2}>1$, then \eqref{eq saddddd} above shows that $\chi_{p,\lambda}(1)<0$. Since $\lim_{x \to +\infty} \chi_{p,\lambda}(x)=+\infty$, we deduce that $\chi_{p,\lambda}$ vanishes somewhere on $(1,+\infty)$,  hence $\mu_2>1$, and then the $2$-periodic orbit $\{p,f_\lambda(p)\}$ is a saddle. 

On the other hand, if $k_{1,2}<1$, then \eqref{eq saddddd} above shows that $\chi_{p,\lambda}(1)>0$. But $1>x_{\text{min}}$, where $x_{\text{min}}\in \R$ is the point at which the quadratic polynomial $\chi_{p,\lambda}$ attains its minimum; indeed, as $k_{1,2}<1$, we have 
$$x_{\text{min}}=\frac 12 ((1+\lambda_1)(1+\lambda_2)k_{1,2}-(\lambda_1+\lambda_2))<\frac{1}{2}(1+\lambda_1\lambda_2)<1\,.$$
Thus, $\chi_{p,\lambda}(1)>0$ implies that $\chi_{p,\lambda}$ is positive on $[1,+\infty)$, hence no eigenvalue has modulus $>1$ (recall that if it were the case, then $\mu_2>1$ would be a real zero of $\chi_{p,\lambda}$), and the $2$-periodic orbit $\{p,f_\lambda(p)\}$ is a sink. 
\end{proof}

\bibliographystyle{alpha}
\bibliography{CSbilliards-final.bib} 
\begin{flushleft}
	\vspace{0.5cm}
	{\scshape Olga Bernardi}\\
	Dipartimento di Matematica Tullio Levi-Civita, Universit\`a di Padova,\\ 
	via Trieste 63, 35121 Padova, Italy.\\
	E-mail: \texttt{obern@math.unipd.it}

	\vspace{0.2cm}
	
	{\scshape Anna Florio}\\
	\textsc{Ceremade}-Universit\'e Paris Dauphine-PSL,\\ 
	75775 Paris, France.\\
	E-mail: \texttt{florio@ceremade.dauphine.fr}
	
	\vspace{0.2cm}
	
	{\scshape Martin Leguil}\\
	\'Ecole polytechnique, CMLS,\\
	Route de Saclay, 91128 Palaiseau Cedex, France.\\
	E-mail:  \texttt{martin.leguil@polytechnique.edu}
\end{flushleft}
\end{document}